\documentclass[10pt]{amsart}
\usepackage{amssymb}
\usepackage{graphicx}
\usepackage{xcolor} 
\usepackage{tensor}
\usepackage{slashed}
\usepackage{comment}

\usepackage[margin=1in, marginpar=.6in]{geometry} 

\usepackage[shortlabels]{enumitem}

\usepackage{hyperref}

\allowdisplaybreaks				

\definecolor{green}{rgb}{0,0.8,0} 

\newtheorem{theorem}{Theorem}[section]
\newtheorem{corollary}[theorem]{Corollary}
\newtheorem{lemma}[theorem]{Lemma}
\newtheorem{proposition}[theorem]{Proposition}
\theoremstyle{definition}
\newtheorem{definition}[theorem]{Definition}

\theoremstyle{remark}
\newtheorem{remark}[theorem]{Remark}
\numberwithin{equation}{section}

\makeatletter
\newcommand{\nrm}{\@ifstar{\nrmb}{\nrmi}}
\newcommand{\nrmi}[1]{\Vert{#1}\Vert}
\newcommand{\nrmb}[1]{\left\Vert{#1}\right\Vert}
\newcommand{\abs}{\@ifstar{\absb}{\absi}}
\newcommand{\absi}[1]{\vert{#1}\vert}
\newcommand{\absb}[1]{\left\vert{#1}\right\vert}
\newcommand{\brk}{\@ifstar{\brkb}{\brki}}
\newcommand{\brki}[1]{\langle{#1}\rangle}
\newcommand{\brkb}[1]{\left\langle{#1}\right\rangle}
\newcommand{\set}{\@ifstar{\setb}{\seti}}
\newcommand{\seti}[1]{\{#1\}}
\newcommand{\setb}[1]{\left\{ #1\right\}}
\makeatother

\newcommand{\td}[1]{\widetilde{#1}}

\newcommand{\wh}[1]{\widehat{#1}}

\newcommand{\VERT}[1]{{\left\vert\kern-0.25ex\left\vert\kern-0.25ex\left\vert #1
    \right\vert\kern-0.25ex\right\vert\kern-0.25ex\right\vert}}

\DeclareMathOperator{\supp}{supp}

\newcommand{\aleq}{\lesssim}

\newcommand{\lap}{\Delta}

\newcommand{\ud}{\mathrm{d}}
\newcommand{\rd}{\partial}

\newcommand{\alp}{\alpha}
\newcommand{\bt}{\beta}
\newcommand{\gmm}{\gamma}

\newcommand{\dlt}{\delta}

\newcommand{\eps}{\epsilon}

\newcommand{\omg}{\omega}
\newcommand{\Omg}{\Omega}

\newcommand{\bfe}{{\bf e}}

\newcommand{\bfm}{{\bf m}}

\newcommand{\bbR}{\mathbb R}
\newcommand{\bbS}{\mathbb S}

\newcommand{\bbZ}{\mathbb Z}

\newcommand{\calU}{\mathcal U}

\newcommand{\frkA}{\mathfrak A}

\newcommand{\frkL}{\mathfrak L}

\newcommand{\pfstep}[1]{\smallskip \noindent {\it #1.}}

\usepackage{accents}

\newcommand{\rslap}{\mathring{\slashed{\lap}}{}}

\newcommand{\f}{\frac}
\def \i {\infty}

\newcommand{\ls}{\lesssim}
\newcommand{\ep}{\epsilon}
\newcommand{\de}{\delta}

\makeatletter
\newcommand{\mbrk}{\@ifstar{\mbrkb}{\mbrki}}
\newcommand{\mbrki}[1]{[ {#1} ]}
\newcommand{\mbrkb}[1]{\left[ {#1} \right]}
\makeatother

\def\XXint#1#2#3{{\setbox0=\hbox{$#1{#2#3}{\int}$}
     \vcenter{\hbox{$#2#3$}}\kern-.5\wd0}}

\newcommand{\radius}{R_0}

\begin{document}

\title[]{Late-time tail for a scalar quasilinear wave equation \\satisfying the weak null condition}

\author{Jonathan Luk}%
\address{Department of Mathematics, Stanford University, CA 94304, USA}%
\email{jluk@stanford.edu}%

\author{Sung-Jin Oh}%
\address{Department of Mathematics, UC Berkeley, Berkeley, CA 94720, USA and School of Mathematics, KIAS, Seoul, 02455, Korea}%
\email{sjoh@math.berkeley.edu}%

\author{Dongxiao Yu}%
\address{ Department of Mathematics, Vanderbilt University, Nashville, TN 37240, USA}%
\email{dongxiao.yu@vanderbilt.edu}%


\begin{abstract}
We consider a class of scalar quasilinear wave equations in three spatial dimensions satisfying the weak null condition. For solutions arising from small, localized, smooth data, we give an asymptotic formula describing the global asymptotics towards the future. We prove that the late-time asymptotics is given by a continuous superposition of decay rates, in stark contrast to equations satisfying a null condition.

The asymptotic formula we obtain is given in terms of a solution to the linear wave equation. Combining this with analysis on the linear wave equation, we strengthen some rigidity results of the third author, showing in particular that any solution with a faster time decay than expected away from the wave zone must vanish identically.
\end{abstract}
\maketitle
\section{Introduction}
Consider the scalar quasilinear wave equation in $\bbR^{1+3}$ of the form
\begin{equation} \label{eq:qnlw}
    g^{\alp \bt}(u) \rd^2_{\alp\bt} u = 0
\end{equation}
with small, smooth, and localized initial data
\begin{equation} \label{eq:qnlw-id}
    (u, \rd_t u)\restriction_{t=0} = (\eps u_{0}, \eps u_{1}) \quad \hbox{ for fixed } u_{0}, u_{1} \in C^{\infty}_{c}(\bbR^{3}), \,\mathrm{supp}(u_0), \mathrm{supp}(u_1) \subset \{ x\in \mathbb R^3: |x|\leq R_0\}.
\end{equation}

Here, we used coordinates $(x^0,x^1,x^2,x^3) = (t,x^1,x^2,x^3)$ for $\bbR^{1+3}$. In the above and throughout the paper, we use the Einstein summation convention where repeated indices are summed over; we will implicitly assume that lower case Greek indices run through $0,1,2,3$ and lower case Latin indices run through $1,2,3$. We also use the notation $r = |x|$.

From now on, $g^{\alp\bt}\rd^2_{\alp\bt}$ is assumed to be a perturbation of the Minkowskian wave operator. More precisely, we assume that $g^{\alp\bt}:\mathbb R \to \mathbb R$ is smooth for each $(\alp,\bt)$, $g^{\alp\bt} = g^{\bt\alp}$, $g^{\alp\bt}(0) = \bfm^{\alp\bt}$, where $$\bfm^{\alp\bt} = \begin{cases}
    -1 & \hbox{if $\alp = \bt = 0$} \\
    1 & \hbox{if $\alp = \bt = i $} \\
    0 & \hbox{if $\alp \neq \bt$}
\end{cases}$$
is the Minkowskian metric. In the small data regime, we may assume without loss of generality that $g^{00} = -1$. Moreover, we stipulate that \eqref{eq:qnlw} satisfies the weak null --- but \emph{not} the null --- condition; specifically, we assume that $G:\mathbb S^2 \to \mathbb R$ defined by
\begin{equation}\label{eq:G.intro.def}
    G(\omg) = \Big(\f{\ud}{\ud u} g^{\alp\bt}(u)\restriction_{u=0}\Big) \widehat{\omg}_\alp \widehat{\omg}_\bt \qquad \wh{\omega}=(-1,\omega)\in\bbR\times\bbS^2
\end{equation}
is not identically $0$. A well-studied example within this class is the equation
$$-\rd^2_{tt}u + a^2(u) \Delta u = 0,\quad a(0) = 1,\quad a'(0) \neq 0$$ considered in \cite{MR2003417, MR1177476, MR2382144}, for which $G(\omg) = 2 a'(0)$.  

When $\ep>0$ is sufficiently small, the global existence of solutions was proven by Lindblad \cite{MR2382144}. Our main result is a \textbf{complete description of the global asymptotics of $u$ towards the future} in this small data regime. We prove that as $t \to \infty$, the solution $u$ can be approximated by $v$, which solves the linear wave equation with characteristic initial data determined by the ``radiation field'' of $u$ defined by the third author \cite{Yu1, MR4772266, yu2024timelike}, which takes into account the $u$-dependent spacetime geometry associated with the quasilinear equation \eqref{eq:qnlw}.

We give a brief presentation of the necessary definitions, postponing the more detailed discussion to Section~\ref{dy:sec:wave}. Fix a small parameter $\dlt \in (0, 1)$, and consider the spacetime region $\Omega := \set{(t, x) \in \bbR^{1+3} : t > e^{\delta/\eps}, \, \abs{x} > \frac{1}{2} (t + e^{\delta/\eps}) + 2 R_{0}}$, whose boundary consists of (1)~a subset of $\set{t = e^{\delta/\eps}}$ (which is precisely the time scale up to which $u$ behaves like a solution to the linear wave equation), and (2)~a time-like cone in Minkowski space to the future of $t = e^{\delta/\eps}$. We introduce the \emph{eikonal function} $q(t, x; \eps)$ on $\Omg$ by solving $g^{\alp \bt}(u) \rd_{\alp} q \rd_{\bt} q = 0$ in $\Omg$ with boundary conditions $q = \abs{x} - t$ on $\rd \Omg$ (the $\eps$-dependence of $q$ arises through $u$ and $\Omg$); see Section~\ref{dy:sec:eik} for the discussion on the existence and properties of $q$. In the new coordinate system $(s, q, \omg) := (\eps \ln t - \dlt, q(t, x), \frac{x}{\abs{x}})$, the following limit exists for each fixed $(q, \omg) \in \bbR \times \bbS^{2}$:
\begin{equation*}
\lim_{s \to \infty} \mu \rd_{q} (r u)(s, q, \omg; \eps) =: - 2 \eps A(q, \omg; \eps).
\end{equation*}
We define the \emph{radiation field} $A(q, \omg; \eps)$ by the above relation. Here, $r:=\abs{x}$ and the key factor $\mu(s, q, \omg; \eps)$ is defined by starting with $(\rd_{t} - \rd_{r}) q$ in the standard polar coordinates $(t, r, \omg)$ and changing the variables to $(s, q, \omg)$.

To fix our notation further, let us introduce our convention for cutoff functions:
\begin{definition}[Cutoff function]\label{def:cutoff}
    Let $\chi:\bbR\to [0,1]$ be a fixed $C^\infty$ cut-off function so that $\chi \equiv 0$ on $(-\infty,\f 12]$ and $\chi \equiv 1$ on $[1,\infty)$. For any $y \in (0,\infty)$, define $\chi_{>y}:\bbR \to [0,1]$ by $\chi_{>y}(s) = \chi(y^{-1} s)$.
\end{definition}

We are now ready to define the solution $v$ to the linear wave equation described above in terms of the radiation field $A$.
\begin{definition}\label{def:comparison}
    Define $v$ to be the unique solution to the characteristic initial value problem
    \begin{equation}\label{eq:v.def}
    \left\{\begin{array}{l}
    \Box_\bfm v = 0,\\
    (\rd_t + \rd_r)(rv)\restriction_{t= r}(t,t\omg) = \f{\chi_{>10}(t) \mathfrak L(t,\omg)}{t},
    \end{array}\right.
\end{equation}
where
\begin{itemize}
    \item $\Box_\bfm = \bfm^{\alp\bt}\rd^2_{\alp\bt}$ is the standard wave operator,
    \item the characteristic initial data $\mathfrak L(t,\omg)$ on the null hypersurface $\{(t,x):t = r \}$, which is parametrized by $t \in \bbR$ and $\omg = \frac{x}{r} \in \bbS^2$, is given by
\begin{equation}\label{eq:frkL.def}
    \mathfrak L(t,\omg) := -\f {\ep^2}{2} \int_{-\infty}^\infty G(\omega)A^2 (p,\omega;\eps)\exp\Big(\frac{1}{2}G(\omega)A(p,\omega;\eps)(\ep \ln t )\Big) J(p, \omg; \eps) \ \ud p,
\end{equation}
    where $G$ has been defined in \eqref{eq:G.intro.def}, $J$ is the \emph{Jacobian factor} that takes the form
    \begin{equation*}
        J(p, \omg; \eps) := \frac{-2}{A_{1}(p, \omg; \eps) \exp(\frac{1}{2} G(\omg) A(p, \omg; \eps) \dlt)},
    \end{equation*}
    and $A_{1}$ is a suitably renormalized limit of $\mu$ in the $(s, q, \omg)$-coordinates (see Section~\ref{dy:sec:wave} for details),
    \begin{equation*}
    A_{1}(q, \omg; \eps) := \lim_{s \to \infty} \exp\left(\frac{1}{2} G(\omg) A(q, \omg; \eps) s \right) \mu (s, q, \omg; \eps).
    \end{equation*}
\end{itemize}
Notice that while the characteristic initial data is stated for $(\rd_t + \rd_r)(rv)\restriction_{t= r}$, the data for $(rv)\restriction_{t= r}$ can be recovered by integration after noting that $rv$ necessarily vanishes at $r =0$.
\end{definition}

\begin{remark} [Gauge invariance of $\frkL$ and $v$] \label{rem:frkL}
Note that the construction of $q$ and $A$ depended on the choice of a small parameter $\dlt \in (0, 1)$. The presence of the Jacobian factor $J$ ensures that the $p$-integral $\frkL(t, \omg)$ is \emph{independent} of the choice of $\dlt$; see Section~\ref{dy:sec:gauge} below.
\end{remark}

\begin{remark} [Reparametrized radiation field $\wh{A}$]\label{rem:frkL-whA}
Later in the paper, it will be convenient to reparametrize $p$ so that $A_{1}$ is normalized to be $-2$. The resulting object -- called the \emph{reparametrized radiation field} or simply the \emph{scattering data} -- is denoted by $\wh{A}$; see Section~\ref{dy:sec:approximate}. In terms of this object, it is clear that we have the following equivalent formula for $\frkL$ (obtained by replacing $A$, $A_{1}$ by $\wh{A}$ and $-2$, respectively):
\begin{equation}\label{eq:frkL.def-Ahat}
    \mathfrak L(t,\omg) = -\f {\ep^2}{2} \int_{-\infty}^\infty G(\omega)\wh{A}^2 (p,\omega;\eps)\exp\Big(\frac{1}{2}G(\omega)\wh{A}(p,\omega;\eps)(\ep \ln t - \dlt)\Big) \ \ud p,
\end{equation}
where we abused notation and wrote $p$ again for the integration variable.
\end{remark}

Our main theorem shows that the linear solution $v$ is a good approximation of the nonlinear solution $u$. We control the difference in terms of vector field bounds: For this we define $Z^{\leq N} w = \sum_{|I|\leq N} |Z^{I} w |$, where $Z \in \{ \rd_\alp, S = t \rd_t + x^i \rd_i, \Omg_{ij} = x^i \rd_j - x^j \rd_i, \Omg_{0i} = x^i \rd_t + t \rd_i\}$, and $Z^{I}$ denotes a composition of vector fields from this set for a suitable multiindex $I$ (see Section~\ref{dy:sec:asytdu}). The following is our main theorem:
\begin{theorem}[Main theorem]\label{thm:main}
    Consider the equation \eqref{eq:qnlw} with initial data satisfying \eqref{eq:qnlw-id} for some fixed $(u_0,u_1)$. Then for any $N_0 \in \bbZ_{\geq 0}$, there exists $\ep_0 = \ep_0(N_0, g^{\alp\bt}, u_0,u_1)>0$ such that as long as $\ep \in (0,\ep_0)$, the following estimate holds for every $N\in \mathbb Z_{\geq 0}$, $N \leq N_0$, and $\eta \in (0, 1)$, with $C_N' = C_N'(N, g^{\alp\bt},u_0,u_1)>0$ and $C_N = C_N(N, g^{\alp\bt}, u_0,u_1)>0$ satisfying $C_N \ep_0 \leq \f 12$:
    $$|Z^{\leq N}(u-v)|(t,x) \leq C_N' \ep \eta^{-2} \brk{t}^{C_N \ep}\min\{\brk{t-r}^{-\f 32+\f \eta 2},  \ep^{-\f 12-\f \eta 2} \brk{t}^{-1}\brk{t-r}^{-\f 12+\f \eta 2} \} \quad \hbox{ when } t - r \geq \eps^{-2},$$
    where $v$ is the solution given in Definition~\ref{def:comparison}.
\end{theorem}

Some remarks are in order.

\begin{remark}
       One can compute that typically the solution $v$ has decay $\brk{t}^{-1+C\ep}$, say on a finite-$|x|$ region; see Corollary~\ref{cor:finite} below. In particular, the error term for the difference $u-v$ in Theorem~\ref{thm:main} indeed decays faster away from the wave zone $\{t\approx r\}$.
\end{remark}

\begin{remark}[Comparison with previous works] \label{rem:compare-prev-work}
There are some previous works on the asymptotics of solutions to equation \eqref{eq:qnlw}. In \cite{MR4078713}, Deng--Pusateri gave the asymptotics near the null cone $\{t = r\}$. This is closely related to the existence of the radiation field established by the third author in \cite{MR4772266}.

The work \cite{yu2024timelike} by the third author probes the asymptotics near timelike infinity and has already obtained an asymptotic formula for the solution in terms of the radiation field. In particular, \cite[Theorem 1]{yu2024timelike} showed that $u$ can be approximated by the solution $w$ to the linear wave equation with the same radiation field, i.e.,  
\begin{equation}\label{eq:w.def}
    w(t,x)=\frac{\eps}{2\pi}\int_{\mathbb{S}^2}\wh{A}(x\cdot\theta-t,\theta;\eps) \ud \mathring{\slashed{\sigma}}(\theta).
\end{equation}
Our work is based on \cite{yu2024timelike}, but goes further to approximate the solution $u$ with a linear solution $v$ to the wave equation which is defined by solving a characteristic initial value problem with characteristic data given as a \emph{nonlinear integral expression} of $\wh{A}$. Since $w$ and $v$ both approximate $u$, they are necessarily close to each other asymptotically. (However, $w$ and $v$ are in general not expected to coincide.) Nonetheless, the nonlinear manner in which $v$ is defined allows us to obtain some further conclusions, e.g., it allows us to strengthen the rigidity results in \cite{yu2024timelike}; see Remark~\ref{rmk:cor.compared}. 
\end{remark}

\begin{remark}[Comparison with the case of null condition]
    Equation \eqref{eq:qnlw} is a typical example of an equation that satisfies the weak null condition. The phenomenon we observe here is very different from that for equations satisfying the classical null condition. (In the context of the equation \eqref{eq:G.intro.def}, the classical null condition corresponds to $G(\omg) \equiv 0$. It also allows for more general systems of equations; see for instance \cite{Chr, Kla}.) For those equations, the generic late-time asymptotics are given by
    $$u(t,x) = c t^{-k} + O(t^{-k-\de}) \quad |x|\leq R,$$
    where $k \in \bbZ_{>0}$ and $\de >0$, instead of having a continuous superposition of leading-order rates. Moreover, the rigidity results we obtain in Corollary~\ref{cor:main.1} and Corollary~\ref{cor:main.2} below are expected not to hold for equations satisfying the classical null condition. 

    The phenomenon observed here is also specific to three spatial dimensions and the borderline decay rate in the wave zone. Indeed, for equation \eqref{eq:qnlw} in odd spatial dimensions $\geq 5$, the results of \cite{LO} apply, and the late-time asymptotics is more similar to the case of null condition in $(3+1)$ dimensions. 
\end{remark}

\begin{remark}[Other models satisfying weak null condition]
    Our results are specific to the scalar quasilinear model \eqref{eq:qnlw}. There are other interesting models which satisfy the weak null condition but fail the classical null condition \cite{Keir, Keir.example, MR1994592}. The late-time asymptotics in those settings may be different, in part because there could be other cancellations. A particularly interesting case is the Einstein vacuum equations in wave coordinates \cite{MR1994592}, where the leading order late-time asymptotics is a single power law (as in the case of null condition) that sensitively depends on the gauge choice. See the upcoming work \cite{YMao}. 
\end{remark}

Starting with Theorem~\ref{thm:main}, we can further compute the main contribution in a finite-$|x|$ region by analyzing the solution to the linear wave equation. Note that the asymptotic profile has no dependence on $x$, and has a continuous superposition of decay rates!
\begin{corollary}\label{cor:finite}
Consider the equation \eqref{eq:qnlw} with initial data satisfying \eqref{eq:qnlw-id} for some fixed $(u_0,u_1)$. There exist $\ep_0 = \ep_0(g^{\alp\bt},u_0,u_1)>0$ sufficiently small and $C = C(g^{\alp\bt},u_0,u_1)>0$ sufficiently large such that for every $R>0$, the following estimate holds for some $C'=C'(g^{\alp\bt},u_0,u_1,R)>0$ whenever $\eta \in (0,1)$ and $\ep \in (0,\ep_0)$:
    \begin{equation*}
    \begin{split}
        &\: \sup_{|x|\leq R} \Big| u(t,x) + \f {\ep^2 }{4\pi t} \int_{-\infty}^\infty\int_{\bbS^2}  G(\omega) A^2 (p,\omega;\eps)\exp\Big(\frac{1}{2}G(\omega) A(p,\omega;\eps)(\ep \ln \f t2 )\Big) J(p,\omg;\ep) \ \ud \mathring{\slashed{\sigma}}(\omg) \ud p \Big| \\
        \leq &\: C' \eta^{-2} \ep^{\f 32-\f \eta 2} t^{-\f 32+C\ep+\eta}.
    \end{split}
\end{equation*}
\end{corollary}
In addition to proving Corollary~\ref{cor:finite} using (a refinement of) Theorem~\ref{thm:main} and the analysis of $v$ in a finite-$|x|$ region, we will also provide another proof using results established in \cite{yu2024timelike} instead of Theorem~\ref{thm:main} (which leads to a slightly different right-hand side in the above estimate); see Sections~\ref{subsec:asymp-v} and \ref{dy:sec:finite-alt}, respectively.

Another consequence of Theorem~\ref{thm:main} is the following late-time asymptotics of the (reparametrized) radiation field $\wh{A}$ in terms of the radiation field of $v$:
\begin{corollary}\label{cor:rad-field}
Consider the equation \eqref{eq:qnlw} with initial data satisfying \eqref{eq:qnlw-id} for some fixed $(u_0,u_1)$. There exist $\ep_0 = \ep_0(g^{\alp\bt},u_0,u_1)>0$ sufficiently small and $C = C(g^{\alp\bt},u_0,u_1) >0$, $C'=C'(g^{\alp\bt},u_0,u_1) >0$ sufficiently large such that the following estimate holds whenever $\eta \in (0,1)$ and $\ep \in (0,\ep_0)$:
    \begin{equation*}
    \begin{split}
        \abs*{\wh{A}(p, \omg; \eps) - B(p, \omg; \eps)}  \leq C' \eta^{-2} \eps^{-\frac{1}{2} - \frac{\eta}{2}} \brk{p}^{-\frac{3}{2} + \frac{\eta}{2} + C \eps} \quad \hbox{ when } p < - \eps^{-2},
    \end{split}
\end{equation*}
where $\wh{A}$ is the (reparametrized) radiation field defined in Section~\ref{dy:sec:approximate}, and $B(p,\omg;\ep)$ is the radiation field of $v$ given by
\begin{equation*}
    B(p, \omg; \eps) := -\frac{1}{2 \eps} \lim_{T \to \infty} (r (\rd_{t} - \rd_{r}) v) \restriction_{(t, r, \frac{x}{\abs{x}}) = (T, p + T, \omg)}.
\end{equation*}
\end{corollary}
We prove (a stronger version of) this result in Section~\ref{sec:rad-field}; see Proposition~\ref{prop:rad-field}.

\begin{remark}
The radiation field $B(p,\omg;\ep)$ can, in principle, be explicitly computed using \eqref{eq:v.def} and ideas in \cite[Proposition~6.11]{LO}; this would give $B$ as an integral of a nonlinear expression of $\wh{A}$. The typical decay expected for $\ep B(p, \omg;\eps)$ is $O(\eps^{2} \brk{p}^{-1 + C\eps})$. If we choose $\eta$ small enough, the error term for the difference $\wh{A}(p, \omg; \eps) - B(p, \omg;\eps)$ decays faster for $p$ large enough (depending on $\eps$) in comparison to this behavior. Since $B$ itself is defined in terms of an integral of a nonlinear expression of $\wh{A}$, Corollary~\ref{cor:rad-field} constitutes a nontrivial compatibility condition on the radiation field $\wh{A}$ arising from a solution to \eqref{eq:qnlw} with small, smooth, and localized initial data \eqref{eq:qnlw-id}. Moreover, this implies that the asymptotics of $\wh{A}$ exhibits a continuous superposition of decay rates. See also Remark~\ref{rem:rad-field-rigidity} below.
\end{remark}

\begin{remark}
From the radiation field $B$, we obtain 
\begin{align*}
    v(t,x)&=\frac{\eps}{2\pi}\int_{\mathbb{S}^2}B(x\cdot\theta-t,\theta;\eps) \ud \mathring{\slashed{\sigma}}(\theta);
\end{align*}
cf.~\eqref{eq:w.def}. We prove (a stronger version of) this formula in Section \ref{sec:rad-field}; see Lemma \ref{sec:rad-field:lem:dy:ZIvexpress}.
\end{remark}

By analyzing the behavior of the comparison solution $v$, we can obtain rigidity results on the decay rate, which show that if a solution decays slightly faster than what one expects, then it is the zero solution. Since decomposition into spherical harmonics enters into the statements of our results, before we proceed, we need to introduce relevant notations:
\begin{definition}\label{def:spherical.harmonics}
Let $\ell \in \mathbb Z_{\geq 0}$. Given a function $f$, denote $\bbS_{(\ell)}f$ as the projection to the spherical harmonics of degree $\ell$, i.e., the eigenspace of $\mathring{\slashed{\Delta}}$ with eigenvalue $-\ell(\ell+1)$. Denote also $\bbS_{(\geq \ell)} = I - \sum_{\ell'=0}^{\ell-1} \bbS_{(\ell')}$, $\bbS_{(\leq \ell)} = \sum_{\ell' = 0}^\ell \bbS_{(\ell')}$.
\end{definition}
We give two versions of our rigidity result. The first version assumes decay in a region where $|x|\sim t$:
\begin{corollary}\label{cor:main.1}
    Let $u$ be a future-global-in-time solution to \eqref{eq:qnlw} arising from data in \eqref{eq:qnlw-id}. Suppose $G$ in \eqref{eq:G.intro.def} is not identically vanishing. Then there exist $C_0>0$ and $\ep_0>0$ (both depending on $u_0$, $u_1$ and $g^{\alp\bt}$) such that if $\ep \in (0,\ep_0)$ and $u$ satisfies the decay bound
    \begin{align}
        |\bbS_{(\leq 2)} u|(t,x) \leq &\: C_1 \brk{t}^{-1-C_0 \ep}, \quad \forall t\geq 0,\forall |x| \in [c_1 t, c_2 t], \label{eq:main.cor.assumption}
    \end{align}
    for some $C_1 >0$ and $0<c_1<c_2<1$, then $u \equiv 0$.
\end{corollary}

\begin{remark}
    Notice in particular that the assumption is only made on up to the second spherical harmonic of the solution, but the conclusion holds for the full solution $u$. The same comment applies to Corollary~\ref{cor:main.2} below.
\end{remark}

The second version of our rigidity result on the decay rate assumes decay in a finite $|x|$ region.
\begin{corollary}\label{cor:main.2}
    Let $u$ be a future-global-in-time solution to \eqref{eq:qnlw} arising from data in \eqref{eq:qnlw-id}. Suppose $G$ in \eqref{eq:G.intro.def} is not identically vanishing. Then there exist $C_0>0$ and $\ep_0>0$ (both depending on $u_0$, $u_1$ and $g^{\alp\bt}$) such that if $\ep \in (0,\ep_0)$ and $u$ satisfies the decay bound
    \begin{align}
        |\bbS_{(0)} u|(t,x) \leq &\: C_1 \brk{t}^{-1-C_0 \ep}, \quad \forall t\geq 0,\forall |x|\leq R,\label{eq:main.cor.2.assumption.0}\\
        |\bbS_{(1)} u|(t,x) \leq &\: C_1 \brk{t}^{-2-C_0 \ep}, \quad \forall t\geq 0,\forall |x|\leq R, \label{eq:main.cor.2.assumption.1}\\
        |\bbS_{(2)} u|(t,x) \leq &\: C_1 \brk{t}^{-3-C_0 \ep}, \quad \forall t\geq 0,\forall |x|\leq R,\label{eq:main.cor.2.assumption.2}
    \end{align}
    for some $C_1 >0$ and $R>0$, then $u \equiv 0$.
\end{corollary}

\begin{remark}\label{rmk:cor.compared}
    Both Corollaries~\ref{cor:main.1} and~\ref{cor:main.2} have their analogues in the work of the third author \cite{yu2024timelike}:
    \begin{itemize}
        \item In the case of Corollary~\ref{cor:main.1}, a similar result holds if the decay assumption is made (though for all spherical harmonics) in $|x| \in [t-2t^{\nu_0}, t- \f 12 t^{\nu_0}]$, $0 < \nu_0 < 1$, instead of $|x|\in [c_1 t, c_2 t]$. Notice that $|x| \in [t-2t^{\nu_0}, t- \f 12 t^{\nu_0}]$ is a region that is much closer to the wave zone $|x| = t$.
        \item In the case of Corollary~\ref{cor:main.2}, a similar result holds if the solution $u$ is a priori \emph{assumed} to be spherically symmetric\footnote{For the equation $-\partial_{tt}^2 u + a^2(u) \Delta u=0$, it is easy to check that spherically symmetric data lead to spherically symmetric solutions. In general, however, spherical symmetry needs not be propagated, and there may not be non-trivial solutions which are spherically symmetric.}.
    \end{itemize}

    In both cases, the results in \cite{yu2024timelike} apply in a regime where the angular Laplacian term in the wave equation can be considered perturbative. Our new representation formula allows us to analyze the solution when the angular Laplacian term is no longer perturbative.
\end{remark}

In the special case where $G(\omg)$ is bounded away from $0$, then in fact controlling that spherically symmetric mode would suffice. More precisely, we have
\begin{corollary}\label{cor:special.case}
    Suppose $G$ in \eqref{eq:G.intro.def} satisfies $\inf_{\omg \in \mathbb S^2} G(\omg) >0$ or $\sup_{\omg \in \mathbb S^2} G(\omg) <0$. Then the following improvements to Corollary~\ref{cor:main.1} and Corollary~\ref{cor:main.2} hold:
    \begin{enumerate}
        \item In Corollary~\ref{cor:main.1}, the assumption \eqref{eq:main.cor.assumption} can be replaced by the weaker assumption:
        \begin{align}
        |\bbS_{(0)} u|(t,x) \leq &\: C_1 \brk{t}^{-1-C_0 \ep}, \quad \forall t\geq 0,\forall |x| \in [c_1 t, c_2 t], \label{eq:main.cor.assumption.improved}
    \end{align}
        \item In Corollary~\ref{cor:main.2}, the assumptions \eqref{eq:main.cor.2.assumption.0}--\eqref{eq:main.cor.2.assumption.2} can be replaced by only requiring \eqref{eq:main.cor.2.assumption.0}.
    \end{enumerate}
\end{corollary}
\begin{remark}
    Note that Corollary~\ref{cor:special.case} in particular applies to the class of equations $-\rd^2_{tt}u + a^2(u) \Delta u = 0$, $a(0) = 1$, $a'(0) \neq 0$ considered in \cite{MR2003417, MR1177476}.
\end{remark}

\begin{remark} \label{rem:rad-field-rigidity}
In \cite[Theorem~2.c)]{yu2024timelike}, the following rigidity result in terms of $\wh{A}$ was shown: {\it There exists $C_{0} > 0$ and $\eps_{0} > 0$ (both depending on $u_{0}$, $u_{1}$, and $g^{\alp \bt}$) such that if $\eps \in (0, \eps_{0})$ and, for some $C_{1} > 0$, $\wh{A}$ satisfies the decay bound $\abs{\wh{A}}(p, \omg) \leq C_{1} \brk{p}^{-1-C_{0} \eps}$ for all $p \geq 0$ and $\omg \in \bbS^{2}$, then the solution $u$ to \eqref{eq:qnlw} must be trivial.} This rigidity is closely related to the nontrivial compatibility condition on $\wh{A}$ imposed by Corollary~\ref{cor:rad-field}. Indeed, using Corollary~\ref{cor:rad-field} and arguments similar to those in Section~\ref{sec:rigidity} below, it should be possible to prove the same rigidity result with the pointwise upper bound weakened to that for $\abs{\bbS_{(\leq 2)} \wh{A}}$ in general, and for $\abs{\bbS_{(0)} \wh{A}}$ if $\inf_{\omg \in \bbS^{2}} G(\omg) > 0$.
\end{remark}

\begin{remark}
    Given the representation formula, it seems reasonable to conjecture for more general equations (not covered by Corollary~\ref{cor:special.case}) that there exist (necessarily non-spherically symmetric) non-trivial solutions which satisfy
    $$|u|(t,x) \leq C_1 \brk{t}^{-1-C_0 \ep}, \quad \forall t\geq 0,\forall |x|\leq R.$$
    Perhaps even stronger $\ep$-independent decay rates (e.g., $\brk{t}^{-\f 54}$) may hold. However, such examples are not constructed in this paper.
\end{remark}

\subsection{Related works}
\subsubsection{Weak null condition} 
According to the classical work of John \cite{Joh1, Joh2}, solutions to a quasilinear wave equation on $\bbR^{1+3}$ may blow up in finite time, no matter how small the initial data are (e.g., of the form \eqref{eq:qnlw-id} with $\eps$ arbitrarily small). Christodoulou \cite{Chr} and Klainerman \cite{Kla} identified a specific cancellation condition for the quadratic nonlinearity --- called the \emph{null condition} --- that nevertheless guarantees the global existence of solutions for small localized initial data, and is satisfied by many important equations from physics. The \emph{weak null condition} is a relaxation of the classical null condition, which in many cases still ensures global existence in the same regime. This condition was first clearly stated in the influential work of Lindblad and Rodnianski \cite{MR1994592}. More specifically, it identified the weak null condition for the Einstein (vacuum or coupled to scalar field) equation in wave coordinates, which was used in a crucial way to give a new proof of the nonlinear stability of the Minkowski spacetime \cite{LinRod} (first proved in the monumental work of Christodoulou and Klainerman \cite{ChrKla}). The precise formulation of the weak null condition requires the discussion of H\"ormander's asymptotic system, which we will not get into; see \cite{MR1994592}.

Quasilinear wave equations of the form \eqref{eq:qnlw} are prototypical scalar (i.e., non-system) examples where the classical null condition fails, but the weak null condition holds. Lindblad's seminal paper \cite{MR1177476} first addressed the global existence problem for the simpler equation $-\rd^2_{tt}u + a^2(u) \Delta u = 0$ in radial symmetry. Alinhac \cite{MR2003417} extended this to general non-symmetric initial data, and Lindblad \cite{MR2382144} further generalized the result to the broader class of equations of the form \eqref{eq:qnlw} considered in the present paper.

For general systems of quasilinear wave equations satisfying the weak null condition, the global existence of solutions with small, smooth, and localized initial data is still largely open. We refer to the work of Keir \cite{Keir,Keir.example} for significant progress and a comprehensive discussion in this direction, as well as to the paper of Kadar \cite{Kad} for surprising subtleties (involving asymptotics towards time-like infinity) in the proper formulation of the global existence conjecture.

Returning to equations of the form \eqref{eq:qnlw}, subsequent work has focused on their scattering theory and global-in-time asymptotics. Deng and Pusateri \cite{MR4078713} used the original H\"ormander asymptotic system to establish asymptotics near the null cone $\set{t = r}$ (in particular, towards null infinity), and established partial scattering results. The third author introduced a geometric modification of H\"ormander's asymptotic system --- called the \emph{geometric reduced system} --- and gave a precise definition of the radiation field $\widehat{A}$. This led to the establishment of the existence of a modified wave operator \cite{Yu1}(i.e., associating a solution $u$ to a given radiation field $\widehat{A}$), and asymptotic completeness \cite{MR4772266} (i.e., associating a radiation field $\widehat{A}$ to a given global solution $u$). Moreover,  asymptotics towards time-like infinity were studied in \cite{yu2024timelike}, whose results are extended in the present paper (see Remark~\ref{rem:compare-prev-work}). We also note the works of Lindblad \cite{Lin} and Lindblad--Schlue \cite{LinSch1, LinSch2} on these problems for the Einstein equation in wave coordinates (and a semi-linear model problem for this), as well as the upcoming work of Mao \cite{YMao}.

We end this summary by pointing out that these PDEs outside the small localized data regime are largely unexplored. One exception is an interesting result of Speck \cite{MR3694013}, which demonstrates an unusual large-data phenomenon for equations of the form $- \rd_{t}^{2} u + (1+u)^{p} \lap u = 0$ ($p=1, 2$). More specifically, \cite{MR3694013} constructs an open set of initial data leading to degeneration of hyperbolicity (i.e., $(1+u)^p \to 0$) without blow-up in finite time.

\subsubsection{Results on late-time tails} 
A key motivation for studying the precise asymptotics towards time-like infinity (referred to as late-time tails) stems from black hole perturbation theory in general relativity, particularly the need for generic lower bounds on late-time tails to investigate the Strong Cosmic Censorship Conjecture for perturbations of subextremal Kerr or Reissner–Nordstr\"om black holes. We will not get into any further details; instead, we refer the reader to \cite{Daf, DafLuk, LukOh2, LukOh3, VdM} and the references therein.

Interestingly, the question of late-time tails turns out to be especially subtle in odd spatial dimensions, including the physical case where $d = 3$. This subtlety arises from the Strong Huygens Principle, which implies that for any solution to the classical wave equation $\Box_{\bfm} v = 0$ with localized data, $v(t, x)$ eventually vanishes if $x$ is fixed and $t$ goes to $+\infty$; in short, the late-time tail for $\Box_{\bfm} v = 0$ is trivial. Hence, any nontrivial late-time tail must arise from the curved background and/or nonlinearity, which makes even the prediction of the sharp decay rate a challenging task at first.

The first heuristic derivation of the sharp decay rate for wave equations on black hole backgrounds was achieved in Price's seminal work \cite{Price}. Henceforth, the subject of late-time tails has often been referred to as \emph{Price's law}, and it has been extensively studied in both physics and mathematics. We give a concise overview here focusing on the mathematical general relativity literature, and refer readers to \cite{LO} and its references for a more comprehensive discussion.

Rigorous mathematical proofs of Price's law began with upper bounds. Price's law as an upper bound (up to a small loss) was first proved in the spherically symmetric, but nonlinear, context in \cite{DRPL}. Subsequently, sharp upper bounds were proven for the linear wave equation on Schwarzschild spacetimes without symmetry assumptions \cite{DSS1}. Tataru \cite{Ta} extended this to general stationary, asymptotically flat spacetimes, and Metcalfe--Tataru--Tohanuanu \cite{MTT} provided an alternative proof applicable to non-stationary cases.

The first lower bound result consistent with Price's law, specifically for Reissner–Nordstr\"om spacetimes, was established by the first two authors \cite{LukOh1}. While this lower bound was only in an averaged sense on the event horizon, it was nevertheless sufficient for applications to the Strong Cosmic Censorship Conjecture in a spherically symmetric, nonlinear setting \cite{LukOh2, LukOh3}. Recently, generic sharp pointwise asymptotics in this setting were achieved by Gautam \cite{Gau}. Outside of symmetry, sharp pointwise asymptotics for linear stationary problems were independently obtained by Angelopoulos--Aretakis--Gajic \cite{AAGPrice} and Hintz \cite{HintzPriceLaw}.

Notably, the presence of nonlinearity or non-stationarity can significantly alter late-time decay behavior compared to linear stationary cases. This phenomenon, supported by the present paper, is also manifest in the work of the first two authors \cite{LO, LO.part2}, which introduced a general method for computing and establishing late-time tails, rigorously justifying numerous predictions in the physics literature \cite{BR.dynamical,GPPII}. Related studies on late-time tails for nonlinear problems include \cite{sjDsyMyMxyY2022, LooXIo}. 

\subsection{Outline of the paper}
In {\bf Section~\ref{dy:sec:wave}}, we recall the main results of the scattering theory for \eqref{eq:qnlw} developed by the third author in \cite{MR4772266}. These are the key ingredients in our analysis in the \emph{wave zone}, where $\abs{r - t}$ is small comparable to $r$ (in particular, this region extends null infinity). In {\bf Section~\ref{sec:med-near}}, we show that the linear wave $v$ in Definition~\ref{def:comparison} furnishes a good approximation of $u$ in the \emph{near} and \emph{intermediate zones} (where $r \aleq 1$ and $1 \aleq r \aleq \abs{r - t}$, respectively). Furthermore, in {\bf Section~\ref{sec:asymp-v}} and {\bf Section~\ref{sec:rad-field}} we study the precise asymptotics of the linear wave $v$ towards future, establishing, in particular, Corollaries~\ref{cor:finite} and \ref{cor:rad-field}, respectively. These three sections follow the philosophy of \cite{LO}, and extend the analysis in \cite{yu2024timelike}. Finally, the rigidity results (Corollaries~\ref{cor:main.1}, \ref{cor:main.2}, and \ref{cor:special.case}) are proved in {\bf Section~\ref{sec:rigidity}}.

\subsection*{Acknowledgements}
J.~Luk was partially supported by a National Science Foundation Grant under DMS-2304445. S.-J.~Oh was partially supported by a Sloan Research Fellowship, a National Science Foundation CAREER Grant under NSF-DMS-1945615, and a National Science Foundation Grant under DMS-2452760. D.~Yu was supported by a VandyGRAF Fellowship from Vanderbilt University.

\section{Analysis in the wave zone}\label{dy:sec:wave}

\subsection{Convention}
For the remainder of the paper, fix $g^{\alp\bt}$ in \eqref{eq:qnlw} as well as $u_0,u_1\in C_c^\infty(\bbR^3)$. We  will allow all constants to depend on $g^{\alp\bt}$, $u_0$, and $u_1$.

The parameter $\eps$ always denotes the size of initial data in \eqref{eq:qnlw-id}. We say that a statement (involving $\eps$) holds for all $\eps\ll 1$, if there exists a small constant $\eps_0\in(0,1)$ such that the statement holds for all $\eps\in(0,\eps_0)$. We say that a statement holds for all $\eps\ll_\beta 1$ if we hope to emphasize that the choice of $\eps_0$ above depends on the quantity $\beta$ (in addition to $g^{\alp\bt}$, $u_0$, and $u_1$).

We will also say that $a \ls b$ or $a=O(b)$ if there exists $C>0$ such that $a\leq Cb$, and say $a \ls_\beta b$ or $a=O_\beta(b)$ if the constant $C$ depends on $\beta$ (in addition to $g^{\alp\bt}$, $u_0$, and $u_1$).

\subsection{Global existence}
For all $\eps\ll1$ (only depending on $(u_0,u_1)$ and the coefficients $g^{\alpha\beta}$), the Cauchy problem \eqref{eq:qnlw} with initial data \eqref{eq:qnlw-id} has a unique global $C^\infty$ solution $u=u(t,x;\eps)$ for all $t\geq 0$. Sometimes we omit $\eps$ and write $u(t,x)$ for simplicity. Here we first apply the global existence result for \eqref{eq:qnlw} in Lindblad \cite{MR2382144} to get a global $C^{14}$ solution, and then apply \cite[Theorems I.4.1 and I.4.2]{MR2455195} to show that this $C^{14}$ solution is indeed the unique $C^\infty$ solution.

Moreover, for a fixed integer $N\geq 0$, for all $\eps\ll_N1$ we have the following pointwise bounds:
\begin{align}\label{dy:est:ptw:lindblad}
|(Z^{\leq N}u)(t,x;\eps)|\lesssim_N \eps \brki{t}^{-1+C_N\eps},\qquad \forall t\geq 0,\ x\in\bbR^3.
\end{align}
The notation $Z^{\leq N}$ has been defined above Theorem \ref{thm:main}.
Moreover, there is a better bound for $\partial u$: \begin{align*}|\partial u|\lesssim \eps \brki{t}^{-1},\qquad \forall t\geq 0,\ x\in\bbR^3.\end{align*}
Moreover, if we choose a real constant $\radius>0$ such that $\supp(u_0,u_1)\subset\{x\in\bbR^3:\ |x|\leq \radius\}$, then we have $u\equiv 0$ whenever $r-t\geq \radius$  by the finite speed of propagation.

\subsection{The optical function}\label{dy:sec:eik}
Fix a small parameter $\delta\in(0,1)$. Recall that $\radius>0$ is a real constant such that $\supp(u_0,u_1)\subset\{x\in\bbR^3:\ |x|\leq \radius\}$. Define
\begin{align}
\label{dy:defn:Omega}\Omega=\Omega_{\delta,\radius,\eps}:=\{(t,x)\in(e^{\delta/\eps},\infty)\times\bbR^3:\ |x|>\frac{1}{2}(t+e^{\delta/\eps})+2\radius\}.\end{align}
Heuristically, one can view $\Omega$ as $\{t>e^{\delta/\eps},\ |x|>t/2\}$. If $t\leq e^{\delta/\eps}$, we expect that the behaviors of $u$ are similar to those of a solution to the linear wave equation.

In \cite[Proposition 3.1]{MR4772266}, it was proved that the eikonal equation
\begin{align}\label{dy:eqn:eik}g^{\alpha\beta}(u)\partial_\alpha q\partial_\beta q=0\quad\text{in }\Omega;\qquad q=r-t\quad \text{on }\partial\Omega\end{align}
admits a unique global $C^2$ solution $q=q(t,x;\eps)$ in $\Omega$ as long as $\eps\ll1$. Note that $q=r-t$ whenever $r-t\geq \radius$ because $u\equiv 0$ for $r-t\geq \radius$. By \cite[Lemma 3.8]{MR4772266}, we have
\begin{align*}
    |q-(r-t)|+\frac{\brki{r-t}}{\brki{q}}+\frac{\brki{q}}{\brki{r-t}}\lesssim t^{C\eps},\qquad \text{in }\Omega.
\end{align*}
It has also been proved in \cite[Proposition 4.1]{MR4772266} that for each fixed integer $N\geq 2$, the $C^2$ solution $q$ is $C^N$ as long as $\eps\ll_N1$ and that
\begin{align*}
|Z^{\leq N}q|\lesssim_N \brki{q}t^{C_N \eps}\lesssim_N  \brki{r-t}t^{C_N \eps},\qquad \text{in }\Omega.
\end{align*}

Next, we set
\begin{align}\label{dy:defn:Omega'}
\Omega'&:=\{(s,q,\omega)\in[0,\infty)\times\bbR\times\mathbb{S}^2:\ q>\frac{1}{2}(1-e^{s/\eps})e^{\delta/\eps}+2\radius\}.
\end{align}
Then, the map
\begin{align*}\Omega\to\Omega':\ (t,x)\mapsto (s,q,\omega)= (\eps\ln t-\delta,q(t,x;\eps),x/|x|)\end{align*}
is bijective and has an inverse map. As a result, any function defined for all $(t,x)\in\Omega$ induces a function of $(s,q,\omega)\in \Omega'$. Moreover, for each fixed integer $N\geq 2$, the map $\Omega\to\Omega'$ and its inverse are $C^N$ as long as $\eps\ll_N1$.

\subsection{The geometric reduced system}\label{dy:sec:grc}
In $\Omega$, we set $(\mu,U)(t,x;\eps):=(q_t-q_r,\eps^{-1}ru)(t,x;\eps)$ and obtain an induced function $(\mu,U)(s,q,\omega;\eps)$ in $\Omega'$. In \cite[Section 5]{MR4772266}, it was proved that $(\mu,U)$ is an approximate solution to the geometric reduced system for \eqref{eq:qnlw}:
\begin{align}\label{dy:eqn:grc}
\left\{\begin{array}{l}
\displaystyle \partial_s (\mu U_q)=0,\\[1em]
\displaystyle \partial_s\mu=\frac{1}{4}G(\omega)\mu^2U_q.\end{array}\right.
\end{align}
Recall that $G(\omega)$ is defined by \eqref{eq:G.intro.def}.
In addition, for $\eps\ll1$, the following limits exist for all $(q,\omega)\in\bbR\times\bbS^2$:
\begin{align}
\label{dy:defn:AA1A2}\left\{\begin{array}{l}
\displaystyle A(q,\omega;\eps):=-\frac{1}{2}\lim_{s\to\infty}(\mu U_q)(s,q,\omega;\eps),\\
\displaystyle A_1(q,\omega;\eps):=\lim_{s\to\infty}\exp(\frac{1}{2}G(\omega)A(q,\omega;\eps)s)\mu(s,q,\omega;\eps),\\
\displaystyle A_2(q,\omega;\eps):=\lim_{s\to\infty}\exp(-\frac{1}{2}G(\omega)A(q,\omega;\eps)s)U_q(s,q,\omega;\eps).
\end{array}\right.
\end{align}
We have $A_1A_2\equiv -2A$. Since $u\equiv 0$ whenever $r\geq t+\radius$, we have $(A,A_2)\equiv 0$ and $A_1\equiv -2$ for $q\geq \radius$. By setting
\begin{align}\label{dy:defn:tdmuU}
\left\{\begin{array}{l}
\displaystyle\td{\mu}(s,q,\omega;\eps):=A_1\exp(-\frac{1}{2}G(\omega)As),\\[1em]
\displaystyle \td{U}_q(s,q,\omega;\eps):=A_2\exp(\frac{1}{2}G(\omega)As),\\[1em]\displaystyle
\lim_{q\to\infty}\td{U}(s,q,\omega;\eps)=0,
\end{array}\right.
\end{align}
we obtain an exact solution $(\td{\mu},\td{U})$ to the reduced system \eqref{dy:eqn:grc}. We refer to \cite[Proposition 5.1]{MR4772266} for more details. Since $(A,A_1,A_2)$ is defined for all $(q,\omega)\in\bbR\times\mathbb{S}^2$, we emphasize that $(\td{\mu},\td{U})$ is defined for all $(s,q,\omega)\in\bbR\times\bbR\times\mathbb{S}^2$, while $(\mu,U)$ is only defined in~$\Omega'$.

In \cite[Proposition 5.1]{MR4772266}, several bounds for $A,A_1,A_2$ in the coordinate set $(q,\omega)$ were proved. For each fixed integer $N\geq 0$, the functions $A,A_1,A_2$ are $C^N$ as long as $\eps\ll_N1$. Besides, for all $a,c\geq 0$ with $a+c\leq N$, we have
\begin{align}\label{prevac:esta}\partial_q^a\partial_\omega^c(A,A_2)=O_{a,c}(\brki{q}^{-1-a+C_{a,c}\eps}),\quad \partial_q^a\partial_\omega^cA_1=O_{a,c}(\brki{q}^{-a+C_{a,c}\eps}),\qquad \forall (q,\omega)\in \bbR\times\bbS^2.\end{align}
The angular derivative, $\partial_\omega$, is defined as follows. Given a function $f=f(\omega)$ defined on $\mathbb{S}^2$, in order to define angular derivatives $\partial_\omega$, we first extend $f$ to $\bbR^3\setminus 0$ by setting $f(x):=f(x/|x|)$ and then set $\partial_{\omega_i}f:=\partial_{x_i}f\restriction_{\mathbb{S}^2}$.

In \cite[Lemma 5.8]{MR4772266}, it was proved that
\begin{align}
|A_1+2|\leq \brki{q}^{-1+C\eps},\qquad -3\leq A_1\leq -1<0.
\end{align}

Finally, we recall from \cite[Proposition 5.1, (5.11), and (5.12)]{MR4772266} that for each fixed integer $N\geq 0$, in $\Omega'$ we have
\begin{align}
\partial_q^a\partial_s^b\partial_\omega^c(\mu-\td{\mu})&=O_{a,b,c}(\eps^{-b}\brki{q}^{-a}t^{-1+C_{a,b,c}\eps}),\\
\partial_q^a\partial_s^b\partial_\omega^c(U-\td{U})&=O_{a,b,c}(\eps^{-b}\brki{q}^{1-a}t^{-1+C_{a,b,c}\eps}),\\
\partial_q^a\partial_s^b\partial_\omega^c(\td{\mu},\td{U})&=O_{a,b,c}(\eps^{-b}\brki{q}^{-a}t^{C_{a,b,c}\eps})
\end{align}
for all $a,b,c\geq 0$ with $a+b+c\leq N$ as long as $\eps\ll_N1$. Here we set $t=e^{\frac{s+\delta}{\eps}}$ which is equivalent to $s=\eps\ln t-\delta$.

\subsection{Approximation}\label{dy:sec:approximate}
In \cite[Section 7]{MR4772266}, the limits in \eqref{dy:defn:AA1A2} were used to generate an approximate solution to \eqref{eq:qnlw} that is sufficiently close to the exact solution $u$ in some sense. The construction is as follows. Recall from \eqref{dy:defn:tdmuU}  that we have an exact solution $(\td{\mu},\td{U})(s,q,\omega;\eps)$ to the geometric reduced system \eqref{dy:eqn:grc}. Now, we define a new function $\td{q}=\td{q}(t,x;\eps)$ in $\Omega$ by solving
\begin{align*}
\td{q}_t-\td{q}_r=\td{\mu}(\eps\ln t-\delta,\td{q}(t,x;\eps),x/|x|;\eps)\quad\text{in }\Omega;\qquad \td{q}=|x|-t\text{ for }|x|-t\geq 2\radius.
\end{align*}
We then define
\begin{align*}
\td{u}(t,x;\eps)&=\eps |x|^{-1}\td{U}(\eps\ln t-\delta,\td{q}(t,x;\eps),x/|x|;\eps),\qquad \forall (t,x)\in\Omega.
\end{align*}
Once again, both $\td{q}$ and $\td{u}$ are $C^N$ as long as  $\eps\ll_N1$. In \cite[Proposition 7.1 and Lemma 7.7]{MR4772266}, it was proved that this $\td{u}$ is an approximate solution to \eqref{eq:qnlw}, that for each integer $N\geq 0$, as long as $\eps\ll_N1$,
\begin{align}
|Z^{\leq N}\td{u}(t,x;\eps)|\lesssim_N \eps t^{-1+C_N\eps},\qquad \forall (t,x)\in\Omega,
\end{align}
and that for each $\gamma\in(0,1)$ and each integer $N\geq 0$, as long as $\eps\ll_{\gamma,N}1$,
\begin{align}
\label{dy:est:diffutdu}
|Z^{\leq N}(u-\td{u})(t,x;\eps)|\lesssim_{\gamma,N}\eps t^{-2+C_N\eps}\brki{r-t},\qquad \forall (t,x)\in\Omega\cap\{|r-t|\lesssim t^\gamma\}.
\end{align}

Here we recall an alternative way to define $\td{u}$ from \cite[Section 7.2]{MR4772266} which will simplify our computations. Define
\begin{align}
F(q,\omega;\eps):=2\radius-\int_{2\radius}^q \frac{2}{A_1(p,\omega;\eps)}\, \ud p,\qquad\forall (q,\omega)\in\bbR\times\mathbb{S}^2.\end{align}
Recall that $A_1\in[-3,-1]$ everywhere, so $F$ is well-defined. In \cite[Section 7.2]{MR4772266}, it was proved that $\brki{F(q,\omega;\eps)}\sim \brki{q}$ for all $(q,\omega)\in\bbR\times\mathbb{S}^2$.
For each fixed $\omega\in\mathbb{S}^2$, the map $q\mapsto F(q,\omega;\eps)$ has an inverse $\wh{F}(q,\omega;\eps)$. Set
\begin{align}\label{dy:defn:whAmuU}
\left\{\begin{array}{l}
\displaystyle \wh{A}(q,\omega;\eps):=A(\wh{F}(q,\omega;\eps),\omega;\eps),\\
\displaystyle \wh{\mu}(s,q,\omega;\eps):=-2\exp(-\frac{1}{2}G(\omega)\wh{A}(q,\omega;\eps)s),\\
\displaystyle \wh{U}(s,q,\omega;\eps):=-\int_q^\infty \wh{A}(p,\omega;\eps)\exp(\frac{1}{2}G(\omega)\wh{A}(p,\omega;\eps)s) \, \ud p.\end{array}\right.
\end{align}
Note that $\wh{A}$, $\wh{\mu}$, and $\wh{U}$ are also defined for all $(s,q,\omega)\in\bbR\times\bbR\times \bbS$.
This $\wh{A}$ is called the \emph{scattering data} for the asymptotic completeness problem.
We also notice that $(\wh{\mu},\wh{U})$ is another exact solution to the reduced system \eqref{dy:eqn:grc}.

By \cite[Lemma 3.2]{yu2024timelike}, we have the following estimates. Here we fix an integer $N\geq 0$ and assume $\eps\ll_N1$. First, for all $a,c\geq 0$ with $a+c\leq N$, we have
\begin{align}\label{dy:est:hatA}
\partial_q^a\partial_\omega^c\wh{A}=O_{a,c}(\brki{q}^{-1-a+C_{a,c}\eps}),\qquad \forall (q,\omega)\in\bbR\times\bbS^2.
\end{align}
Next, for all $a,b,c\geq 0$ with $a+b+c\leq N$ and all $(s,q,\omega)\in[0,\infty)\times\bbR\times\mathbb{S}^2$, we have
\begin{align}\label{dy:est:muhatUUq}
\partial_q^a\partial_s^b\partial_\omega^c(\wh{\mu}+2)&=O_{a,b,c}(\brki{q}^{-\max\{1,b\}-a+C_{a,b,c}\eps}e^{C_{a,b,c}s});\\
\partial_q^a\partial_s^b \partial_\omega^c\wh{U}_q&=O_{a,b,c}(\brki{q}^{-1-a-b+C_{a,b,c}\eps}e^{C_{a,b,c}s});\\
\partial_s^b \partial_\omega^c\wh{U}&=O_{b,c}((\eps^{-1}\brki{q}^{C_{b,c}\eps}1_{b=0}+1_{b>0})e^{C_{b,c}s}).
\end{align}
Finally, for all $0\leq c\leq N$, $s\geq 0$, $\omega\in \mathbb{S}^2$, and $q>C_0e^{C_0s}(e^{\delta/\eps}-e^{(\delta+s)/\eps}-1)$ for some constant $C_0>1$, we have
\begin{align} \label{dy:est:hatU2}\partial_\omega^c\wh{U}&=O_{N,C_0}(e^{C_{N,C_0}s}).
\end{align}
Note that the bounds for $\wh{\mu}$ in \eqref{dy:est:muhatUUq} are better than those for $\td{\mu}$ defined in \eqref{dy:defn:tdmuU}. For example, we have $\wh{\mu}_q=O(\brki{q}^{-2+C\eps}e^{Cs})$ but $\td{\mu}_q=O(\brki{q}^{-1+C\eps}e^{Cs})$. Such a difference occurs because we have $\partial_q(-2)=0$ while $\partial_qA_1=O(\brki{q}^{-1+C\eps})$. This explains why we introduce this new map~$F$.

It follows from \cite[Lemma 7.2]{MR4772266} that
\begin{align}\wh{q}(t,x;\eps):=F(\td{q}(t,x;\eps),x/|x|;\eps)
\end{align}
solves the transport equation
\begin{align}
\wh{q}_t-\wh{q}_r=\wh{\mu}(\eps\ln t-\delta,\wh{q}(t,x;\eps),x/|x|;\eps)\qquad\text{in }\Omega,
\end{align}
that $\wh{q}=r-t$ for $r-t\geq \radius$, and that
\begin{align}\label{dy:eqn:tduwtu}
\td{u}(t,x;\eps)=\wh{u}(t,x;\eps):=\eps |x|^{-1}\wh{U}(\eps\ln t-\delta,\wh{q}(t,x;\eps),x/|x|;\eps),\qquad\text{in }\Omega.
\end{align}
By \cite[Lemmas 7.4 and 7.7]{MR4772266} and \cite[Lemma 3.4]{yu2024timelike}, in $\Omega$ we also have
\begin{align}
    \brki{r-t}/\brki{\wh{q}}+\brki{\wh{q}}/\brki{r-t}\lesssim t^{C\eps},
\end{align}
and
\begin{align}
|Z^{\leq N}\wh{q}|&\lesssim_N\brki{r-t} t^{C_N\eps},\\
|Z^{\leq N}(\wh{q}-r+t)|&\lesssim_N (1+\ln\brki{r-t}) t^{C_N\eps}
\end{align}
for each fixed integer $N\geq 0$ as long as $\eps\ll_N1$. It follows from \eqref{dy:est:muhatUUq} and \eqref{dy:est:hatU2} that for each integer $N\geq 0$, as long as $\eps\ll_N1$, in $\Omega$ we have
\begin{align}
|Z^{\leq N}\wh{U}|+|Z^{\leq N}(\eps\wh{U}_s)|&\lesssim_N t^{C_N\eps};\\
|Z^{\leq N}\wh{u}|=|Z^{\leq N}(\eps r^{-1}\wh{U})|&\lesssim_N \eps t^{-1+C_N\eps};\\
|Z^{\leq N}(\wh{\mu}+2)|&\lesssim_N \brki{r-t}^{-1}t^{C_N\eps}.
\end{align}
Here all the functions are of $(t,x)$ via the map $(t,x)\mapsto (\eps\ln t-\delta,\wh{q}(t,x;\eps),x/|x|)=(s,q,\omega)$, which is different from Section \ref{dy:sec:grc}. We refer to \cite[(3.23)]{yu2024timelike}.

\subsection{Asymptotic formulas for $\td{u}$}\label{dy:sec:asytdu}
Because of the estimates \eqref{dy:est:diffutdu}, we can use the approximate solution $\td{u}$ (which equals $\wh{u}=\eps r^{-1}\wh{U}$ in $\Omega$ by \eqref{dy:eqn:tduwtu}) to approximate the exact solution $u$ to the scalar quasilinear wave equation \eqref{eq:qnlw}. Fix an integer $k\geq 0$ and let $Z^k$ denote an arbitrary product of $k$ commuting vector fields defined below \eqref{dy:est:ptw:lindblad}. To approximate $Z^ku$ with $k\geq 1$, we need to apply the chain rule and Leibniz's rule to compute $Z^k\wh{u}=Z^k(\eps r^{-1}\wh{U})$.  However, this computation can be very complicated, because there is no explicit formula for the function $\wh{q}$ in general. To handle this issue, in \cite[Proposition 3.3]{yu2024timelike}, several asymptotic formulas for $Z^k\wh{u}$ were presented.

Before we state these asymptotic formulas, we first introduce a new notation. Recall that we defined $11$ different commuting vector fields above Theorem \ref{thm:main}. Write them as $Z_{1},\dots,Z_{11}$. Fix a multiindex $I=(i_1,i_2,\dots,i_k)$ of the length $k=|I|$ with $1\leq i_1,\dots,i_k\leq 11$, we set
\begin{align*}
    Z^I:=Z_{i_1}Z_{i_2}\cdots Z_{i_k}.
\end{align*}
While we can write such a product as $Z^k$, it is important to specify which product we use in the asymptotic formulas.

We can now state \cite[Proposition 3.3]{yu2024timelike}. For each multiindex $I$, we define functions $U_I=U_I(s,q,\omega;\eps)$ and $A_I=A_I(q,\omega;\eps)$ inductively on $|I|$ for all $(s,q,\omega)\in\bbR\times\bbR\times\mathbb{S}^2$ and $\eps\ll_{|I|}1$. First, recall the definitions of $\wh{U}(s,q,\omega;\eps)$ and $\wh{A}=\wh{A}(q,\omega;\eps)$ from \eqref{dy:defn:whAmuU}. We set $U_0=\wh{U}$ and $A_0=-2\wh{A}$. Next, for each multiindex $I$ with $|I|>0$, we write $Z^I=ZZ^{I'}$ with $|I'|=|I|-1$. We now define
\begin{align}\label{eq:UI.def}
U_I(s,q,\omega;\eps)=\left\{
\begin{array}{ll}
\eps\partial_sU_{I'}+q\partial_qU_{I'}-U_{I'}, & Z^I=SZ^{I'};  \\
(\omega_i\partial_{\omega_j}-\omega_j\partial_{\omega_i})U_{I'}, & Z^I=\Omega_{ij}Z^{I'},\ 1\leq i<j\leq 3;\\
\eps\omega_i\partial_sU_{I'}-q\omega_i\partial_qU_{I'}+\partial_{\omega_i}U_{I'}-\omega_iU_{I'}, & Z^I=\Omega_{0i}Z^{I'},\ 1\leq i\leq 3;\\
0, & Z^I=\partial Z^{I'};
\end{array}\right.
\end{align}
and
\begin{align} \label{eq:AI.def}
A_I(q,\omega;\eps)=\left\{
\begin{array}{ll}
q\partial_qA_{I'}, & Z^I=SZ^{I'};  \\
(\omega_i\partial_{\omega_j}-\omega_j\partial_{\omega_i})A_{I'}, & Z^I=\Omega_{ij}Z^{I'},\ 1\leq i<j\leq 3;\\
-q\omega_i\partial_qA_{I'}+\partial_{\omega_i}A_{I'}-2\omega_iA_{I'}, & Z^I=\Omega_{0i}Z^{I'},\ 1\leq i\leq 3;\\
0, & Z^I=\partial Z^{I'}.
\end{array}\right.
\end{align}
Here, given a function $f(\omega)$ defined for $\omega\in\mathbb{S}^2$, we set $\partial_{\omega_i}f=\partial_i \widetilde{f} \restriction_{\mathbb{S}^2}$ where $\widetilde{f}(x)=f(x/|x|)$ is the extension of $f$ to a neighborhood of $\mathbb{S}^2$.
As seen from Section \ref{dy:sec:approximate}, for each fixed integer $N\geq 0$, the functions $\wh{U}$ and $\wh{A}$ are $C^N$ as long as $\eps\ll_N1$. Thus, we can define $U_I$ and $A_I$ whenever $\eps\ll_{|I|}1$. Note that $U_I\equiv A_I\equiv 0$ for all $q\geq \radius$. Besides, we have the pointwise bounds \eqref{dy:est:hatA} -- \eqref{dy:est:hatU2} with $(\wh{U},\wh{A})$ replaced by $(U_I,A_I)$. More precisely, fix an integer $N\geq 0$ and take $\eps\ll_N1$. For all $a,b,c\geq 0$ with $a+b+c\leq N$ and all $(s,q,\omega)\in[0,\infty)\times\bbR\times\mathbb{S}^2$, we have
\begin{align}
\partial_q^a\partial_\omega^cA_I&=O_{a,c,|I|}(\brki{q}^{-1-a+C_{a,c,|I|}\eps}); \label{eq:AI.bnd} \\
\partial_q^a\partial_s^b \partial_\omega^c\partial_qU_I&=O_{a,b,c,|I|}(\brki{q}^{-1-a-b+C_{a,b,c,|I|}\eps}e^{C_{a,b,c,|I|}s}); \label{eq:UI.est.1} \\
\partial_s^b \partial_\omega^cU_I&=O_{b,c,|I|}((\eps^{-1}\brki{q}^{C_{b,c,|I|}\eps}1_{b=0}+1_{b>0})e^{C_{b,c,|I|}s}).\label{eq:UI.est.2}
\end{align}
For all $0\leq c\leq N$, $s\geq 0$, $\omega\in \mathbb{S}^2$, and $q>C_0e^{C_0s}(e^{\delta/\eps}-e^{(\delta+s)/\eps}-1)$ for some constant $C_0>1$, we have
\begin{align} \partial_\omega^c U_I&=O_{N,C_0,|I|}(e^{C_{N,C_0,|I|}s}).
\end{align}
Finally, we also have estimates that connect $U_I$ and $A_I$. For all $a,b,c\geq 0$ with $a+b+c\leq N$ and all $(s,q,\omega)\in[0,\infty)\times\bbR\times\mathbb{S}^2$, we have
\begin{align} \label{eq:UI.integral.of.AI}\partial_q^a\partial_s^b\partial_\omega^c (2\partial_qU_I+A_I)(s,q,\omega;\eps)
&=O_{a,b,c,|I|}(\brki{q}^{-\max\{1,b\}-1-a+C_{a,b,c,|I|}\eps}e^{C_{a,b,c,|I|}s}).
\end{align}

We can now state the main estimates in \cite[Proposition 3.3]{yu2024timelike}. For each fixed multiindex $I$ and fixed integer $N\geq 0$, in $\Omega$ we have
\begin{equation}
\begin{split}
&\: |Z^{\leq N}\left(r Z^I\wh{u}-\eps U_I(\eps\ln t-\delta,|x|-t,x/|x|;\eps)\right)| \\
\lesssim_{N,|I|} &\: \eps(1+\ln\brki{r-t})\brki{r-t}^{-1}t^{C_{N,|I|}\eps} +\eps \brki{r-t}t^{-1+C_{N,|I|}\eps},\label{dy:est:prop3.3:timelike}
\end{split}
\end{equation}
\begin{equation}
    \begin{split}
        &\: |Z^{\leq N}\left(r(\partial_t-\partial_r)Z^I\wh{u}-\eps A_I(|x|-t,x/|x|;\eps)\right)| \\
\lesssim_{N,|I|} &\: \eps(1+\ln\brki{r-t})\brki{r-t}^{-2}t^{C_{N,|I|}\eps} +\eps t^{-1+C_{N,|I|}\eps},\label{dy:est:prop3.3:timelike.2}
    \end{split}
\end{equation}

as long as $\eps\ll_{|I|,N}1$.
We will apply the estimates \eqref{dy:est:prop3.3:timelike}--\eqref{dy:est:prop3.3:timelike.2} when $|r-t|\gg 1$ (e.g.,\ when $|r-t|\sim t^{\gamma}$ with $\gamma\in(0,1)$).

\subsection{Gauge independence} \label{dy:sec:gauge}
The definitions of the functions $A,A_1,A_2,\wh{A}$ in \eqref{dy:defn:AA1A2} and \eqref{dy:defn:whAmuU} depend on the choice of the optical function.
In Section \ref{dy:sec:eik}, we constructed a speicific optical function $q=q(t,x;\eps)$ by solving the eikonal equation \eqref{dy:eqn:eik}. There, we consider the region $\Omega=\Omega_{\delta,\radius,\eps}$ defined by \eqref{dy:defn:Omega} and assign the boundary condition $q\restriction_{\partial\Omega}=r-t$. Suppose that now we choose a different optical function. If all the derivations in Sections \ref{dy:sec:grc}--\ref{dy:sec:asytdu} still work, we will also obtain the corresponding functions $A,A_1,A_2,\wh{A}$. It is natural to ask how the new functions are related to the old ones. To answer this question, we present the gauge independence result from \cite[Section 3.6]{yu2024timelike}.

We first recall a definition \cite[Definition 3.7]{yu2024timelike}.
\begin{definition}\rm Fix $\delta,\kappa\in(0,1)$. Let $q=q(t,x;\eps)$ be a $C^2$ optical function defined for all $\eps\ll_{\kappa,\delta}1$ and for all $(t,x)\in\mathbb{R}^{1+3}$ with $t>e^{\delta/\eps}$ and $|x|>\kappa t$. That is, we have
\begin{align*}g^{\alpha\beta}(u)\partial_\alpha q\partial_\beta q=0,\qquad \forall t>e^{\delta/\eps}, |x|>\kappa t.\end{align*}
Also suppose that $q\restriction_{r-t\geq \radius}=r-t$. We say that the optical function $q$ is \emph{$(\delta,\kappa)$-admissible} if it satisfies the following assumptions:
\begin{enumerate}[a)]
\item For all $t>e^{\delta/\eps}$ and $|x|>\kappa t$, we have
\begin{align*}\sum_{|I|\leq 1}|Z^I(q_t-q_r,q_r^{-1},(q_t-q_r)^{-1})|&\lesssim t^{C\eps};\\
\sum_{|I|\leq 1}\sum_{i=1,2,3}|Z^I(q_i+\omega_iq_t)|&\lesssim t^{-1+C\eps};\end{align*}
\item The map $(s,q,\omega)=(\eps\ln t-\delta,q(t,x;\eps),x/|x|)$ induces a $C^1$ diffeomorphism from $\{t>e^{\delta/\eps},\ |x|>\kappa t\}$ to a subset of $[0,\infty)\times\mathbb{R}\times\mathbb{S}^2$, so every $C^1$ function of $(t,x)$ induces a $C^1$ function of $(s,q,\omega)$;
\item Define $(\mu,U)(t,x)=(q_t-q_r,\eps^{-1}ru)(t,x)$ and consider the induced functions  $(\mu,U)(s,q,\omega;\eps)$. The limits in \eqref{dy:defn:AA1A2} exist, i.e.\
\begin{align*}\left\{\begin{array}{l}
\displaystyle A(q,\omega;\eps):=-\frac{1}{2}\lim_{s\to\infty}(\mu U_q)(s,q,\omega;\eps),\\
\displaystyle A_1(q,\omega;\eps):=\lim_{s\to\infty}\exp(\frac{1}{2}G(\omega)A(q,\omega;\eps)s)\mu(s,q,\omega;\eps),\\
\displaystyle A_2(q,\omega;\eps):=\lim_{s\to\infty}\exp(-\frac{1}{2}G(\omega)A(q,\omega;\eps)s)U_q(s,q,\omega;\eps).
\end{array}\right.\end{align*}
One can also define the scattering data $\wh{A}$ by \eqref{dy:defn:whAmuU}.
\end{enumerate}
\end{definition}
\rm
It is clear that the optical function $q(t,x;\eps)$ constructed in the previous subsections is $(\overline{\delta},\kappa)$-admissible for all $\overline{\delta}\in[\delta,1)$ and $\kappa>1/2$.

We can now state the gauge independence result which is \cite[Proposition 3.8]{yu2024timelike}.

\begin{proposition}\label{dy:prop:gauge_indep}
Fix $\delta,\kappa,\overline{\delta},\overline{\kappa}\in(0,1)$. Let $q,\overline{q}$ be a $(\delta,\kappa)$-admissible optical function and a $(\overline{\delta},\overline{\kappa})$-admissible optical function, respectively. Define $\delta_0=\max\{\delta,\overline{\delta}\}$ and $\kappa_0=\max\{\kappa,\overline{\kappa}\}$. Then, whenever $\eps\ll_{\delta,\kappa,\overline{\delta},\overline{\kappa}}1$, we have
\begin{enumerate}[\rm i)]
\item  There exists $Q_\infty=Q_\infty(q,\omega;\eps)\in C^1(\mathbb{R}\times\mathbb{S}^2)$ such that
\begin{align*}Q_\infty(q,\omega;\eps)=\lim_{s\to\infty}\overline{q}(s,q,\omega;\eps)\end{align*}
where $\overline{q}(s,q,\omega;\eps)$ is the function induced by $\overline{q}(t,x;\eps)$ via the coordinate changes $(s,q,\omega)=(\eps\ln t-\delta,q(t,x;\eps),x/|x|)$. Besides,  we have
\begin{align*}\partial_qQ_\infty=\lim_{s\to\infty}\partial_q\overline{q},\qquad \partial_{\omega_i}Q_\infty=\lim_{s\to\infty}\partial_{\omega_i}\overline{q}.\end{align*}
All the convergences are here uniform in $(q,\omega)\in\mathbb{R}\times\mathbb{S}^2$.
\item
One can exchange the roles of $q$ and $\overline{q}$ to obtain $\overline{Q}_\infty(\overline{q},\omega;\eps)\in C^1(\mathbb{R}\times\mathbb{S}^2)$ such that
\begin{align*}\overline{Q}_\infty(\overline{q},\omega;\eps)=\lim_{\overline{s}\to\infty}q(\overline{s},\overline{q},\omega;\eps),\\
\partial_q\overline{Q}_\infty=\lim_{\overline{s}\to\infty}\partial_{\overline{q}}q,\quad \partial_{\omega_i}\overline{Q}_\infty=\lim_{\overline{s}\to\infty}\partial_{\omega_i}q.\end{align*}
Moreover,  we have \begin{align*}Q_\infty(\overline{Q}_\infty(\overline{q},\omega;\eps),\omega;\eps)&=\overline{q},\\\overline{Q}_\infty(Q_\infty(q,\omega;\eps),\omega;\eps)&=q.\end{align*}

\item Let $(A,A_1,A_2)(q,\omega;\eps)$ and $(\overline{A},\overline{A}_1,\overline{A}_2)(\overline{q},\omega;\eps)$ be the limits defined by \eqref{dy:defn:AA1A2}, respectively. Then,
\begin{align*}\overline{A}(Q_\infty(q,\omega;\eps),\omega;\eps)&=A(q,\omega;\eps),\\
\overline{A}_1(Q_\infty(q,\omega;\eps),\omega;\eps)&=(A_1\cdot \partial_qQ_\infty)(q,\omega;\eps)\cdot\exp(\frac{1}{2}G(\omega)A(q,\omega;\eps)\cdot (\delta-\overline{\delta})).\end{align*}
\end{enumerate}
\end{proposition}\rm

It follows from Proposition \ref{dy:prop:gauge_indep} that the scattering data depends only on $\delta$. That is, given two admissible optical functions, if they correspond to the same $\delta$, then they will generate the same scattering data $\wh{A}$. We refer to \cite[Corollary 3.9]{yu2024timelike}. It also follows from Proposition \ref{dy:prop:gauge_indep}.iii) that $\frkL$ is independent of the choice of $\dlt$, as pointed out in Remark~\ref{rem:frkL}.

\section{Analysis in the intermediate and near zones} \label{sec:med-near}

\subsection{Preliminaries about the Minkowskian wave operator}

In this subsection, we collect some standard facts about the Minkowskian wave operator.

In the following, it is important to propagate vector field bounds. We first remark that the vector field bounds are equivalent to the following weighted estimates.
\begin{lemma}[Vector field bounds]\label{lem:vector.fields}
    Let $N \in \mathbb Z_{\geq 0}$. Then the following are equivalent for $t \geq 0$:
    \begin{equation*}
        \begin{split}
            &\: |Z^{\leq N} \phi|(t,x)\ls \brk{t+r}^{-\alp}\brk{t-r}^{-\bt} \\
            \iff &\:
    \begin{cases}
        \displaystyle\sum_{|a|\leq N} t^a |\rd^a \phi|(t,x) \ls (t+r)^{-\alp}\brk{t-r}^{-\bt} \quad \hbox{for $r\leq \f t2$} \\
        \displaystyle\sum_{|a|+b+|c|\leq N} \brk{t-r}^{|a|} \brk{t+r}^b |\rd^a (\rd_t + \rd_r)^b \Omg^c \phi|(t,x) \ls \brk{t+r}^{-\alp}\brk{t-r}^{-\bt} \quad \hbox{for $r\geq \f t2$}
    \end{cases}
        \end{split}
    \end{equation*}
    for any $\alp,\bt \in \mathbb R$.
\end{lemma}
\begin{proof}
    This is a standard computation after noting that
    \begin{align}
    \rd_i = \f{- x^j \Omg_{ij} + t \Omg_{0i} - x^i S}{t^2-r^2},\quad \rd_t = \f{tS - x^i \Omg_{0i}}{t^2 - r^2}, \label{eq:unweighted.in.terms.of.weighted.1}\\
    \rd_t + \rd_r = \f{S + \sum_{i=1}^{n} \frac{x^i}{r}\Omg_{0i}}{2(t+r)},\quad \Omg_{ij} = \f 1t(x^i \Omg_{0j} - x^j \Omg_{0i}).\label{eq:unweighted.in.terms.of.weighted.2}
    \end{align}
\end{proof}

\begin{lemma}\label{lem:comm.cutoff}
    The following identity holds for any function $\phi$:
    \begin{equation}
        \Box_{\bfm} (\chi_{>H}(t-r) \phi(t,x)) = \chi_{>H}(t-r) (\Box_\bfm \phi)(t,x) -2 \chi'_{>H}(t-r) r^{-1} (\rd_t + \rd_r)(r \phi)(t,x),
    \end{equation}
    where the cutoff function $\chi_{>H}$ is defined as in Definition~\ref{def:cutoff}.
\end{lemma}
\begin{proof}
    This is an explicit computation; see \cite[Lemma~8.6]{LO}.
\end{proof}

We also record the standard Strong Huygens principle for the classical linear wave equation.
\begin{lemma}[Strong Huygens principle]\label{jl:lem:Huygens}
    Suppose $\phi$ and $F$ satisfy
    $$\Box_\bfm \phi = F,$$
    where $(\phi, \rd_t \phi)\restriction_{t=0} = (0,0)$ and $F$ is supported in $\{(t,x): t\geq 0\}$.

    Given $(T,X)\in \mathbb R^{1+3}$, the value of $\phi(T,X)$ only depends on $F$ in the set
    $$D_{T,X} = \{(t,x): T - |X| \leq t+r \leq T + |X|,\, t-r \leq |T| - |X| \}.$$
\end{lemma}

To conclude this subsection, we give an estimate concerning the flat inhomogeneous linear wave equation. While the bound is standard, we give a proof for completeness. In the proof, we in particular follow \cite{MTT} to reduce to spherically symmetric estimates.
\begin{proposition}\label{jl:prop:invert.wave}
    Suppose $\phi$ and $F$ satisfy
    $$\Box_\bfm \phi = F,$$
    where $(\phi, \rd_t \phi)\restriction_{t=0} = (0,0)$ and $F$ is supported in $\{(t,x): t\geq 0\}$.

    For every fixed $T> 0$, if there exist $B,\, \eta_0 >0$, $\alp>1$ such that $F$ satisfies the estimates for all $t \in [0,T]$,
    $$|Z^{\leq M_0}F|(t,x) \leq \begin{cases}
        B \brk{t-r}^{-\alp} \brk{r}^{-2-\eta_0} & \hbox{when $r\leq \f t 2$} \\
        B \brk{t-r}^{-1-\eta_0} \brk{r}^{-\alp-1} & \hbox{when $r \geq \f t 2$}
    \end{cases},$$
    then $\phi$ satisfies the following estimates
    $$|Z^{\leq M_0} \phi|(T,x) \ls B \eta_0^{-1} \min \{\brk{T-r}^{-\alp}, (\alp -1)^{-1}\brk{r}^{-1} \brk{T-r}^{-\alp+1} \}.$$
\end{proposition}
\begin{proof}
    By the positivity of the fundamental solution to the linear wave equation on Minkowski, it suffices to bound $\psi$ satisfying
    $$\Box_\bfm \psi = \widetilde{F} := \begin{cases}
        B \brk{t-r}^{-\alp} \brk{r}^{-2-\eta_0} & \hbox{when $r\leq \f t 2$} \\
        B \brk{t-r}^{-1-\eta_0} \brk{r}^{-\alp-1} & \hbox{when $r \geq \f t 2$}
    \end{cases},$$
    with $(\psi,\rd_t \psi)\restriction_{t=0} = (0,0)$.

    Such a $\psi$ is spherically symmetric and thus satisfies $(\rd_t-\rd_r)(\rd_t + \rd_r)(r\psi) = - r \widetilde{F}$. Slightly abusing notation, we write $\psi$ and $\widetilde{F}$ as functions of $(t,r)$. An explicit computation using the method of characteristics shows that
    \begin{equation}
        \begin{split}
            r\psi(T,r) = - \int_{0}^r \int_{0}^{T-r+s} (T-r+2s-s') \widetilde{F}(s',T-r+2s-s') \, \ud s' \, \ud s.
        \end{split}
    \end{equation}
    Hence,
    \begin{equation}
        \begin{split}
            |r\psi|(T,r) \leq &\: \int_0^r \int_0^{\f 23(T-r+2s)} (T-r+2s-s') \widetilde{F}(s',T-r+2s-s') \, \ud s' \, \ud s \\
            &\: + \int_0^r \int_{\f 23(T-r+2s)}^{T-r+s} (T-r+2s-s') \widetilde{F}(s',T-r+2s-s') \, \ud s' \, \ud s \\
            \ls &\: B \int_0^r \int_0^{\f 23(T-r+2s)} \brk{T-r+2s-s'}^{-\alp} \brk{2s'-(T-r+2s)}^{-1-\eta_0} \, \ud s' \, \ud s \\
            &\: + B \int_0^r \int_{\f 23(T-r+2s)}^{t-r+s} \brk{T-r+2s-s'}^{-1-\eta_0} \brk{2s'-(T-r+2s)}^{-\alp} \, \ud s' \, \ud s \\
            \ls &\: B \eta_0^{-1} \int_0^r \brk{T-r+2s}^{-\alp}\, \ud s \\
            \ls &\: B \eta_0^{-1} \min \{ r \brk{T-r}^{-\alp}, (\alp -1)^{-1} \brk{T-r}^{-\alp+1} \},
        \end{split}
    \end{equation}
    which gives the desired estimates. \qedhere
\end{proof}

\subsection{Vector field bounds for the comparison solution $v$}

Notice that from \eqref{eq:v.def}, we also obtain the value of $rv$ on the cone $\{t = r\}$ by integration using the condition that $rv = 0$ when $r = 0$.

\begin{lemma}\label{lem:v.vector.fields}
    For any fixed $N_0 \in \bbZ_{\geq 1}$, there exists $\ep_0 >0$ sufficiently small such that the comparison solution $v$ satisfies the following estimates for all $N\leq N_0$ if $\ep \in (0,\ep_0]$:
    \begin{equation}
        |Z^{\leq N} v|(t,x) \ls_N \ep \brk{t}^{-1+C_N \ep},\quad \hbox{when $r\leq t$}.
    \end{equation}
\end{lemma}
\begin{proof}
    The argument in this lemma loses derivatives. We will use the convention that $N$ represents possibly different non-negative integers from line to line, which decrease as the proof proceeds. 

    \pfstep{Step~1: Data bound} Recalling the initial data from \eqref{eq:v.def}, the formula \eqref{eq:frkL.def-Ahat} for $\frkL$, and the bound \eqref{dy:est:hatA}, we immediately obtain
    \begin{equation}
        \sum_{\substack{b+|\alp|\leq N}} |\brk{t}^b(\rd_t+ \rd_r)^b \Omg^\alp  v|(t,x) \ls_N \ep \brk{t}^{-1+C_N\ep}\quad \hbox{when $t = r$}.
    \end{equation}
    Note that we have lost a power of $\epsilon$ when integrating to bound $rv$ from $(\rd_t + \rd_r)(rv)$.

    To control the $\rd_t - \rd_r$ derivatives, we analyze the wave equation as a transport equation
    \begin{equation}
        -(\rd_t +\rd_r)(\rd_t - \rd_r)(rv) + \f 1{r^2} \rslap (rv) = 0.
    \end{equation}
    Note that $\rslap (rv)$ can be controlled by $r \sum_{|\alp| = 2} |\Omg^\alp v|$.
    Using the bounds we have already obtained, we can inductively control higher $\rd_t -\rd_r$ derivatives so as to show that
    \begin{equation}
        \sum_{a+ b+|\alp|\leq N}|t^b(\rd_t+ \rd_r)^b \Omg^\alp (\rd_t - \rd_r)^a v|(t,x) \ls \ep \brk{t}^{-1+C_N\ep}\quad \hbox{when $t = r$}.
    \end{equation}

    \pfstep{Step~2: Weak decay bound} Using Lemma~\ref{lem:vector.fields} and the bounds from Step~1, it follow that $Z^I v$ satisfies
    \begin{equation}\label{eq:ZleqNv.data}
        |Z^{\leq N} v|(t,x) \ls \ep \brk{t}^{-1+C_N \ep},\quad |(\rd_t + \rd_r) Z^{\leq N} v|(t,x) \ls \ep \brk{t}^{-2+C_N \ep} \quad \hbox{when $t = r$}.
    \end{equation}
    In particular, as long as $\ep$ is sufficiently small (with respect to $N$), the energy on the cone satisfies
    \begin{equation}
        \int_{t=r} \Big( ((\rd_t + \rd_r) Z^{\leq N} v)^2 + r^{-2} |\mathring{\slashed{\nabla}} Z^{\leq N}  v|^2 \Big)\, r^2 \ud \mathring{\slashed{\sigma}} \ud t \ls \ep^2 \int_0^\infty \brk{t}^{-2+2C_N \ep} \ud t \ls \ep^2,
    \end{equation}
    where $|\mathring{\slashed{\nabla}} v|^2$ denotes that the squared norm with respect to the gradient on the unit sphere and satisfies $|\mathring{\slashed{\nabla}} v|^2 \ls \sum_{|\alp| = 1}|\Omg^{\alp} v|^2$.

    Standard energy conservation thus gives that
    \begin{equation}
        \int_{\bbR^3} (\rd Z^{\leq N} v)^2(t,x) \, \ud x \ls_N \ep^2.
    \end{equation}
    We then deduce from Klainerman's Sobolev inequality (see e.g., \cite[Chapter 2, Theorem~1.3]{MR2455195}) that
    \begin{equation}\label{eq:trivial.Klainerman.1}
        |\rd Z^{\leq N} v|(t,x) \ls_N \ep \brk{t-r}^{-\f 12}\brk{t}^{-1}
    \end{equation}
    in the region $r\leq t$. Integrating along the integral curves of $\rd_t -\rd_r$ starting from the set $\{t=r\}$, and using the bounds \eqref{eq:ZleqNv.data} and \eqref{eq:trivial.Klainerman.1}, we then obtain
    \begin{equation}\label{eq:trivial.Klainerman.2}
        |\rd Z^{\leq N} v|(t,x) \ls_N \ep \brk{t}^{-\f 12}.
    \end{equation}

    \pfstep{Step~3: Improvement of the decay estimates for $(\rd_t+\rd_r)(rZ^I v)$ when $t-r \leq 2$} Since $\Box_{\bfm} v = 0$, we have $\Box_\bfm Z^I v = 0$ for any $I$. The wave equation can be rewritten as follows:
    \begin{equation} \label{eq:wave4ZIv}
        -(\rd_t +\rd_r)(\rd_t - \rd_r)(r Z^Iv) + \f 1{r^2} \rslap (r Z^I v) = 0,
    \end{equation}
    Using the weak decay estimate \eqref{eq:trivial.Klainerman.2} in Step~2, we have (for $|I|\leq N$ and $\ep$ sufficiently small)
    \begin{equation}\label{eq:wave.for.v.commuted.to.integrate}
        |(\rd_t +\rd_r)(\rd_t - \rd_r)(r Z^Iv)|(t,x) \ls_N \eps r^{-1} \brk{t}^{-\f 12}.
    \end{equation}
    We now integrate this along the integral curves of $\rd_t - \rd_r$ to control $(\rd_t +\rd_r)(r Z^Iv)$ in the region $t - r \leq 2$. Observe that by the cut-off of the initial data in \eqref{eq:v.def} and finite speed of propagation, we have $v \equiv 0$ in the set $\{(t,x): 0\leq t \leq 5, r \leq t\} \cup \{ (t,x): 5\leq t \leq 10, r \leq 10-t\}$. In particular, on the support of $v$, we must have $r \geq 4$ when $t-r\leq 2$.

    Therefore, integrating \eqref{eq:wave.for.v.commuted.to.integrate} along the integral curves of $\rd_t - \rd_r$ and using the initial data bound \eqref{eq:ZleqNv.data}, we obtain
    \begin{equation}\label{eq:dvrZIv.est}
        \begin{split}
            |(\rd_t +\rd_r)(r Z^Iv)|(t,x) \ls_N &\: \eps^2\brk{r}^{-1+C_{|I|} \ep}+\eps \brk{t-r}\brk{r}^{-\f 32+C_{|I|}\eps} \\
            \ls_N &\: \eps^2\brk{r}^{-1+C_{|I|} \ep} + \eps\brk{r}^{-\f 32+C_{|I|}\eps}\quad \hbox{when $t-r\leq 2$}.
        \end{split}
    \end{equation}

    \pfstep{Step~4: Conclusion of the argument} Applying Lemma \ref{lem:comm.cutoff} with $H=2$, we obtain
    \begin{equation}\label{eq:RHS.of.cutoff.v}
        \Box_{\bfm} (\chi_{>2}(t-r) Z^I v(t,x)) = -2 \chi'_{>2}(t-r) r^{-1} (\rd_t + \rd_r)(r Z^I v)(t,x).
    \end{equation}
    Observe that by \eqref{eq:dvrZIv.est},
    \begin{equation}\label{eq:est.RHS.of.cutoff.v.1}
        |\hbox{RHS of \eqref{eq:RHS.of.cutoff.v}}| \ls_N \eps^2\brk{r}^{-1+C_{|I|} \ep} + \eps\brk{r}^{-\f 32+C_{|I|}\eps}.
    \end{equation}

    We now use $(T,X)$ (with $T-|X| \geq 2$) to denote the output point at which we estimate $Z^I v$ (i.e., we estimate $\chi_{>2}(T-|X|) Z^I v(T,X)$), and use $(t,x)$ to denote an input point.

    First, by the strong Huygens principle (Lemma~\ref{jl:lem:Huygens}), we only need control the right-hand side of \eqref{eq:RHS.of.cutoff.v} in the set $D_{T,X}\cap \{(t,x): 1\leq t-r \leq 2\}$. In this region, we can trade $r$ decay for $T$ growth so that by \eqref{eq:est.RHS.of.cutoff.v.1}, we have
    \begin{equation}\label{eq:est.RHS.of.cutoff.v.2}
        |\hbox{RHS of \eqref{eq:RHS.of.cutoff.v}}| \ls_N \ep^2 \brk{r}^{-2-\ep} \brk{T}^{(C_{|I|}+1)\ep} + \eps\brk{r}^{-\f 52+C_{|I|} \ep}.
    \end{equation}
    We now apply Proposition~\ref{jl:prop:invert.wave} separately to each term of \eqref{eq:est.RHS.of.cutoff.v.2}. For the first term, we apply it with $(B,\alp,\eta_0) = (\ep^2 \brk{T}^{(C_{|I|}+1)\ep}, 1+\ep, 1)$, while for the second term, we apply it with $(B,\alp,\eta_0) = (\ep, \f 32-C_{|I|} \ep, 1)$. By Proposition~\ref{jl:prop:invert.wave}, we then obtain
    \begin{equation}
        \begin{split}
            |\chi_{>2}(T-|X|) Z^I v(T,X)| \ls_N &\: \min\{ \ep^2 \brk{T}^{(C_{|I|}+1)\ep} \brk{T-|X|}^{-1-\ep}, \ep \brk{T}^{(C_{|I|}+1)\ep} \brk{|X|}^{-1} \brk{T-|X|}^{-\ep}\} \\
            &\: + \min\{ \ep \brk{T-|X|}^{-\f 32+C_{|I|} \ep}, \ep  \brk{|X|}^{-1} \brk{T-|X|}^{-\f 12+C_{|I|} \ep}\} \\
            \ls_N &\: \ep \brk{T}^{-1+(C_{|I|}+1)\ep},
        \end{split}
    \end{equation}
    which implies the desired bound after relabelling the constant $C_{|I|}+1$. \qedhere
\end{proof}

\subsection{Estimates for the comparison solution $v$ in the region $\{\f t2 \leq r \leq t\}$}

We next would like to show that $(\rd_t + \rd_r)(rv)(t,x)$ is close to $\lim_{q\to -\infty} \f{\eps^2}{t} (\rd_s U_I)(\eps\ln t-\delta,q,x/|x|;\eps)$ in the region $\f t2 \leq r \leq t$; see Lemma~\ref{lem:dvrv.in.wave.zone}. The proof has two steps (at least for the harder case where all the $Z$'s are weighted vector field. In the first step, we compute the characteristic data for $Zv$ on the cone $\{t = r\}$ (Lemma~\ref{lem:data.v.commuted}). In the second step, we use Step~1 and consider the wave equation as a transport equation to obtain the desired bound.

We now begin the first step and compute the characteristic initial data for $Z^I v$ on $\{t=r\}$. For this purpose, it is convenient to start with a general lemma concerning the characteristic initial data for the commuted equation for the weighted commutators $Z \in \{ S, \Omg_{ij}, \Omg_{0i}\}$. Notice that we only consider the weighted commutators since they are tangential to the cone $\{t=r\}$.

\begin{lemma}\label{lem:data.differentiate.once}
    Suppose $\phi$ is a solution to $\Box_\bfm\phi = 0$ with characteristic data $$r\phi\restriction_{t=r}(t,t\omg) = f(t,\omg).$$

    Then $Z\phi$ has characteristic data
    \begin{equation}\label{eq:commuted.data.weighted.Z}
        r Z\phi\restriction_{t=r}(t,t\omg) = \begin{cases}
        t \rd_t f(t,\omg)  - f(t,\omg) & \hbox{if $Z = S$} \\
        \omg_i \rd_{\omg_j} f(t,\omg) - \omg_j \rd_{\omg_i} f(t,\omg)  & \hbox{if $Z = \Omg_{ij}$} \\
        t\omg_i \rd_t f(t,\omg) + \rd_{\omg_i} f(t,\omg) - \omg_i f(t,\omg) & \hbox{if $Z = \Omg_{0i}$}.
    \end{cases}.
    \end{equation}

\end{lemma}
\begin{proof}
    We first observe that
    \begin{align}
        rS\phi =  S(r\phi) - r\phi,\quad
        r \Omg_{ij} \phi =  \Omg_{ij}(r \phi), \quad
        r \Omg_{0i} \phi = \Omg_{0i}(r\phi) - \f {t \omg_i}r (r\phi),
    \end{align}
    and
    \begin{equation}
        \begin{split}
            S = \f 12 (t+r)(\rd_t + \rd_r) + \f 12(t-r)(\rd_t - \rd_r),\quad \Omg_{ij} = x^i \rd_{x^j} - x^j \rd_{x^i},\\
            \Omg_{0i} =  t (\rd_{x^i} - \f{x^i}r \rd_r) + \f {x^i}2 (1+ \f tr) (\rd_t + \rd_r) + \f {x^i}2 (1-\f tr)(\rd_t - \rd_r).
        \end{split}
    \end{equation}
    Then the conclusion follows from combining the above identities, setting $t = r$, noting $(\rd_t + \rd_r)(r\phi)\restriction_{t=r}(t,t\omg) = \rd_t f(t,\omg)$ and recalling that $\rd_{\omg_i} = r(\rd_{x^i} - \f {x^i}r \rd_r)$. \qedhere
\end{proof}

The above lemma can be rephrased as follows in a form that will be used below.

\begin{lemma}\label{lem:data.differentiate.once.to.use}
    Suppose $\phi$ is a solution to $\Box_\bfm\phi = 0$ with characteristic data $$(\rd_t + \rd_r)(r\phi)\restriction_{t=r}(t,t\omg) = f(t,\omg).$$

    Then $Z\phi$ has characteristic data
    \begin{equation}\label{eq:commuted.data.weighted.Z}
        (\rd_t + \rd_r) (r Z\phi)\restriction_{t=r}(t,t\omg) = \begin{cases}
        t \rd_t f(t,\omg)  & \hbox{if $Z = S$} \\
        \omg_i \rd_{\omg_j} f(t,\omg) - \omg_j \rd_{\omg_i} f(t,\omg)  & \hbox{if $Z = \Omg_{ij}$} \\
        t\omg_i \rd_t f(t,\omg) + \rd_{\omg_i} f(t,\omg)  & \hbox{if $Z = \Omg_{0i}$}.
    \end{cases}.
    \end{equation}
\end{lemma}
\begin{proof}
    This is an easy computation starting with Lemma~\ref{lem:data.differentiate.once}. (Note that we have renamed $f$.) \qedhere
\end{proof}

Before we apply Lemma~\ref{lem:data.differentiate.once.to.use} to compute the asymptotic data for $Z^I v$, we first carry out a simple computation concerning the asymptotic profiles $U_I$:
\begin{lemma}\label{lem:UI.derivative.transform} For any multiindex $I$, we have
    \begin{align}
\lim_{q\to -\infty} \rd_s U_I(s,q,\omega;\eps)=\left\{
\begin{array}{ll}
\lim_{q\to -\infty}(\eps\partial_s^2U_{I'}- \rd_s U_{I'}), & Z^I=SZ^{I'};  \\
\lim_{q\to -\infty}(\omega_i\partial_{\omega_j}-\omega_j\partial_{\omega_i})\rd_s U_{I'}, & Z^I=\Omega_{ij}Z^{I'},\ 1\leq i<j\leq 3;\\
\lim_{q\to -\infty}(\eps\omega_i\partial_s^2 U_{I'}+\partial_{\omega_i} \rd_s U_{I'}-\omega_i \rd_s U_{I'}), & Z^I=\Omega_{0i}Z^{I'},\ 1\leq i\leq 3.
\end{array}\right.
\end{align}
\end{lemma}
\begin{proof}
   Starting with \eqref{eq:UI.def} and differentiating, we obtain
    \begin{align}
\rd_s U_I(s,q,\omega;\eps)=\left\{
\begin{array}{ll}
\eps\partial_s^2U_{I'}+q\partial_q\rd_sU_{I'}-\rd_sU_{I'}, & Z^I=SZ^{I'};  \\
(\omega_i\partial_{\omega_j}-\omega_j\partial_{\omega_i})\rd_sU_{I'}, & Z^I=\Omega_{ij}Z^{I'},\ 1\leq i<j\leq 3;\\
\eps\omega_i\partial_s^2U_{I'}-q\omega_i\partial_q \rd_s U_{I'}+\partial_{\omega_i} \rd_s U_{I'}-\omega_i \rd_s U_{I'}, & Z^I=\Omega_{0i}Z^{I'},\ 1\leq i\leq 3.
\end{array}\right.
\end{align}

    We then take the $q\to -\infty$ limit. Note that the terms $\lim_{q\to -\infty} q\partial_q\rd_sU_{I'}$ and $\lim_{q\to -\infty}  q\omega_i\partial_q \rd_s U_{I'}$ both vanish thanks to \eqref{eq:UI.est.1}. We then obtain the desired conclusion. \qedhere
\end{proof}

We now compute the characteristic initial data for $Z^I v$:
\begin{lemma}\label{lem:data.v.commuted}
    Suppose $Z^I$ contains only weighted vector fields. Then for $t \geq 10$,
    \begin{equation}\label{eq:data.v.commuted}
        (\rd_t + \rd_r) (rZ^I v)\restriction_{t=r}(t,t\omg) = \lim_{q\to -\infty} \f{\eps^2}{t} (\rd_s U_I)(\eps\ln t-\delta,q,\omg;\eps).
    \end{equation}
\end{lemma}
\begin{proof}
    We first remark the limit on the right-hand side of \eqref{eq:data.v.commuted} is well-defined thanks to the estimate \eqref{eq:UI.est.1} and the fundamental theorem of calculus.

    Let
    $$\widehat{f}_I(t,\omg) = \lim_{q\to -\infty} \f{\eps^2}{t} (\rd_s U_I)(\eps\ln t-\delta,q,\omg;\eps).$$

    Our goal is thus to show that the data $(\rd_t + \rd_r) (rZ^I v)\restriction_{t=r}(t,t\omg)$ are given by $\widehat{f}_I(t,\omg)$ when $t \geq 10$. To achieve this, we first show in Step~1 that they agree when $|I|=0$. We then show in Step~2 that $\widehat{f}_I(t,\omg)$ obeys
    \begin{align}\label{eq:inductive.step.for.v.data}
\widehat{f}_I (t,\omg) =\left\{
\begin{array}{ll}
t \rd_t \widehat{f}_{I'}, & Z^I=SZ^{I'};  \\
\omega_i\partial_{\omega_j}\widehat{f}_{I'}-\omega_j\partial_{\omega_i}\widehat{f}_{I'}, & Z^I=\Omega_{ij}Z^{I'},\ 1\leq i<j\leq 3;\\
\omega_i t \rd_t \widehat{f}_{I'} +\partial_{\omega_i} \widehat{f}_{I'}, & Z^I=\Omega_{0i}Z^{I'},\ 1\leq i\leq 3.
\end{array}\right.
\end{align}
     The conclusion then follows inductively after using Lemma~\ref{lem:data.differentiate.once.to.use}.

    \pfstep{Step~1: The $|I|=0$} For the base case, recall from \eqref{dy:defn:whAmuU} and the definition $U_0 = \widehat{U}$ that
    $$\rd_s U_0(s,q,\omg;\ep) = \rd_s \widehat{U}(s,q,\omg;\ep) = -\int_{q}^\infty \frac{1}{2}G(\omega) \wh{A}^2(p,\omega;\eps)\exp(\frac{1}{2}G(\omega)\wh{A}(p,\omega;\eps)s) \, \ud p.$$
    The desired identity therefore follows from the definition of $v$ in Definition~\ref{def:comparison} (recall also from Remark~\ref{rem:frkL} that \eqref{eq:frkL.def-Ahat} is equivalent to \eqref{eq:frkL.def}).

    \pfstep{Step~2: Proof of \eqref{eq:inductive.step.for.v.data}} Observe that
    \begin{equation}
        \begin{split}
            &\: \rd_t \Big( \f{\eps^2}{t} (\rd_s U_I)(\eps\ln t-\delta,q,x/|x|;\eps) \Big) \\
            = &\: -  \f{\eps^2}{t^2} (\rd_s U_I)(\eps\ln t-\delta,q,x/|x|;\eps) +  \f{\eps^3}{t^2} (\rd_s^2 U_I)(\eps\ln t-\delta,q,x/|x|;\eps).
        \end{split}
    \end{equation}
    Thus
    \begin{equation}
        \lim_{q\to -\infty} \f{\eps^3}{t^2} (\rd_s^2 U_I)(\eps\ln t-\delta,q,x/|x|;\eps) = \rd_t \widehat{f}_I + \f 1t \widehat{f}_I.
    \end{equation}

    By Lemma~\ref{lem:UI.derivative.transform}, we therefore have \eqref{eq:inductive.step.for.v.data}, as desired. \qedhere
\end{proof}

\begin{lemma}\label{lem:dvrv.in.wave.zone}
    $(\rd_t + \rd_r)(rv)$ satisfies the following estimates for $|I|\leq N$ and $\ep \ll_N 1$:
    \begin{equation}
        \begin{split}
            &\: \Big| (\rd_t + \rd_r)(rZ^I v(t,x)) -\lim_{q\to -\infty} \f{\eps^2}{t} (\rd_s U_I)(\eps\ln t-\delta,q,x/|x|;\eps) \Big| \\
            \ls_N &\: \ep \brk{r}^{-2+C_N \ep} \brk{t-r} + \ep \brk{r}^{-1+C_N \ep} \brk{t-r}^{-1},\quad \hbox{when $\f t2 \leq r\leq t$}.
        \end{split}
    \end{equation}
\end{lemma}
\begin{proof}
    \pfstep{Step~1: When $Z^I$ consists of at least one unweighted vector field} In this case, $U_I \equiv 0$ by \eqref{eq:UI.def}. We are thus reduced to proving decay for $(\rd_t + \rd_r)( rZ^I v)$. Now because at least one of the vector field is unweighted, we can convert it to a weighted vector field using \eqref{eq:unweighted.in.terms.of.weighted.1} to gain a factor of $\brk{t-r}^{-1}$. Thus, by \eqref{lem:v.vector.fields}, we have
    $$|(\rd_t + \rd_r)( rZ^I v)|(t,x)\ls \ep \brk{r}^{-1+C_N \ep} \brk{t-r}^{-1},\quad \hbox{when $\f t2 \leq r\leq t$},$$
    which is what we need to prove.

    \pfstep{Step~2: When all vector fields in $Z^I$ are weighted vector fields} In this case, we use the wave equation for $Z^I v$ together with the information about the data in Lemma~\ref{lem:data.v.commuted}.

    Since $\Box_\bfm v = 0$, we have $\Box_\bfm Z^I v = 0$ and thus \eqref{eq:wave4ZIv} holds.
    By Lemma~\ref{lem:v.vector.fields} (as well as the fact that if $r \geq \f t2$, then $r \geq \f {10}{3}$ on the support of $v$), we have $|\f 1{r^2} \rslap (r Z^I v)| \ls \ep \brk{r}^{-2+C_{|I|}\ep}$ when $\f t2 \leq r \leq t$. Plugging this bound into \eqref{eq:wave4ZIv}, we obtain
    \begin{equation}\label{eq:wave.rZIv}
        |(\rd_t - \rd_r)(\rd_t +\rd_r)(rZ^Iv)|(t,x) \ls \ep \brk{r}^{-2+C_{|I|}\ep}.
    \end{equation}
    Using Lemma~\ref{lem:data.v.commuted}, we now integrate \eqref{eq:wave.rZIv} in the direction of $\rd_t - \rd_r$ starting at $\{t=r\}$. This yields the estimate
    $$\Big| (\rd_t + \rd_r)\Big( rZ^I v(t,x) -\lim_{q\to -\infty} \f{2\eps^2}{t+|x|} (\rd_s U_I)(\eps\ln \f{t+|x|}2-\delta,q,x/|x|;\eps) \Big)  \Big|
            \ls_N \ep \brk{r}^{-2+C_N \ep} \brk{t-r} .$$
    Finally, observing that $|\f 2{t+|x|} - \f 1 t|\ls \brk{r}^{-2} \brk{t-r}$ and $|\ln \f{t+|x|}2 - \ln t| \ls \brk{r}^{-1}\brk{t-r}$ (when $\f t2 \leq r \leq t$), and using \eqref{eq:UI.est.1}--\eqref{eq:UI.est.2} (and the mean value theorem), we obtain the desired estimate.\qedhere
\end{proof}

\subsection{Proof of the main asymptotic formula}

In this subsection, we prove the main asymptotic formula stated in Theorem~\ref{thm:main}. In order to control $u-v$, we show that $\Box_\bfm (u-v)$ also exhibits improved estimates, and that $u-v$ exhibits improved decay near the wave zone. We start with estimating $\Box_\bfm (u-v) = \Box_\bfm u$ (Lemma~\ref{lem:Boxu.est}) and then turn to estimating $(\rd_t + \rd_r)(rZ^I (u-v))(t,x)$ in the wave zone (Lemma~\ref{lem:dvru.in.wave.zone}). The main estimate will then be carried out in Propositions~\ref{jl:prop:main.terms.for.cutoff} and Proposition~\ref{prop:main.error}.

\begin{lemma}\label{lem:Boxu.est}
    The following estimates hold:
    \begin{equation}
        \Big| Z^{\leq N} (\Box_{\bfm} u)\Big| \ls_N
            \ep^2 \brk{t}^{-2+C_N\ep} \brk{t-r}^{-2}.
    \end{equation}
\end{lemma}
\begin{proof}
    Using the form of the equation, this is an easy consequence of \eqref{dy:est:ptw:lindblad}, Lemma~\ref{lem:vector.fields}, and the chain and product rule. See the first equation in \cite[Section~5]{yu2024timelike} for details. \qedhere
\end{proof}

Using Lemma~\ref{lem:dvrv.in.wave.zone} and the fact that $U_I$ is an asymptotic profile for $Z^I u$ (derived in in Section~\ref{dy:sec:wave}), we show that $u-v$ exhibits improved decay in the wave zone.
\begin{lemma}\label{lem:dvru.in.wave.zone}
For $\gamma \in (0,1)$, $|I|\leq N$ and $\ep \ll_{\gamma,N} 1$, the following estimate holds when $(t,x) \in \Omg \cap \{|t-r| \ls t^\gamma\}$:
    $$\Big| (\rd_t + \rd_r)(rZ^I (u-v))(t,x) \Big| \ls_{N} \ep (1 + \ln\brk{t-r}) \brk{t-r}^{-1} \brk{t}^{-1+C_{|I|}\ep} + \ep \brk{t-r} \brk{t}^{-2+C_{|I|}\ep}.$$
\end{lemma}
\begin{proof}

    By \eqref{dy:est:diffutdu}, \eqref{dy:eqn:tduwtu}, and \eqref{dy:est:prop3.3:timelike},
    \begin{equation}
        \begin{split}
            &\: \Big| (\rd_t + \rd_r)(rZ^I u)(t,x) - \f {\ep^2} t (\rd_s U_I)(\ep\ln t - \de, |x|-t,x/|x|;\ep) \Big| \\
            \ls_{|I|} &\: \ep (1 + \ln\brk{t-r}) \brk{t-r}^{-1} t^{-1+C_{|I|}\ep} + \ep \brk{t-r} t^{-2+C_{|I|}\ep}
        \end{split}
    \end{equation} for $(t,x) \in \Omg \cap \{|r-t|\ls t^\gamma\}$.
    Therefore, using also Lemma~\ref{lem:dvrv.in.wave.zone}, we obtain
    \begin{equation}
        \begin{split}
            &\: \Big| (\rd_t + \rd_r)(rZ^I (u-v))\Big|(t,x) \\
            \ls &\:  \f{\ep^2}t \Big| (\rd_s U_I)(\ep\ln t - \de, |x|-t,x/|x|;\ep) - \lim_{q\to -\infty} (\rd_s U_I)(\ep\ln t - \de, q,x/|x|;\ep)\Big| \\
            &\: + \ep (1 + \ln\brk{t-r}) \brk{t-r}^{-1} \brk{t}^{-1+C_{|I|}\ep} + \ep \brk{t-r} \brk{t}^{-2+C_{|I|}\ep}.
        \end{split}
    \end{equation}
    To control the first term on the right-hand side, we use \eqref{eq:UI.integral.of.AI} to obtain
    \begin{equation}
        \begin{split}
            &\: \Big| (\rd_s U_I)(\ep\ln t - \de, |x|-t,x/|x|;\ep) - \lim_{q\to -\infty} (\rd_s U_I)(\ep\ln t - \de, q,x/|x|;\ep)\Big| \\
            \leq &\: \int_{-\infty}^{|x|-t}  (\rd_q \rd_s U_I)(\ep\ln t - \de, q,x/|x|;\ep) \,\ud q
            \ls \int_{-\infty}^{|x|-t}  \brk{q}^{-2+C_{|I|}\ep} \brk{t}^{C_{|I|}\ep}\,\ud q \\
            \ls &\: \brk{t}^{C_{|I|}\ep} \brk{t-r}^{-1+C_{|I|}\ep} \ls  \brk{t}^{C_{|I|}\ep} \brk{t-r}^{-1},
        \end{split}
    \end{equation}
    where in the last step we used $\brk{t-r} \ls \brk{t}$ and absorbed the loss in the power by making $C_{|I|}$ larger.

    Combining all the estimates above yields the desired conclusion. \qedhere
\end{proof}

We now turn to the proof of Theorem~\ref{thm:main}. As before, we use $(T,X)$ to denote an output point and $(t,x)$ to denote an input point.

\begin{proposition}\label{jl:prop:main.terms.for.cutoff}
Let $T,X$ be fixed. Suppose $1 \leq H \leq (T-|X|)^{\f 34}$. Then the following estimates hold on $D_{T,X}$ for $|I|\leq N$ and $\ep \ll_N 1$:
    \begin{equation}
        \begin{split}
            &\: \Big| \Box_\bfm(\chi_{>H}(t-r)Z^I (u-v)(t,x)) \Big|\\
            \ls_N &\: \begin{cases}
                \ep \brk{t}^{-4+C_N \ep}   & \hbox{when $r \leq \f t2$} \\
                \ep 1_{\{t-r\geq \f H2\}} (1 + \ln \brk{t-r})\brk{t-r}^{-2} \brk{r}^{-2+C_N\ep} + \ep 1_{\{\f H2 \leq t-r\leq H\}} \brk{r}^{-3+C_N\ep} & \hbox{when $r \geq \f t2$}.
            \end{cases}
        \end{split}
    \end{equation}
\end{proposition}
\begin{proof}
    By Lemma~\ref{lem:comm.cutoff}, we have
    \begin{equation}\label{eq:cutoff.eq.for.u-v}
        \begin{split}
            &\: \Box_\bfm(\chi_{>H}(t-r) Z^I (u-v)(t,x))\\
            = &\: \chi_{>H}(t-r) (\Box_\bfm (Z^I (u-v))(t,x) -2 \chi'_{>H}(t-r) r^{-1} (\rd_t + \rd_r)(r Z^I(u-v))(t,x) \\
            = &\: \chi_{>H}(t-r) (\Box_\bfm Z^I u)(t,x) -2 \chi'_{>H}(t-r) r^{-1} (\rd_t + \rd_r)(r Z^I(u-v))(t,x),
        \end{split}
    \end{equation}
    where in the last line we used $\Box_\bfm Z^I v = 0$.

    For the first term on the right-hand side of \eqref{eq:cutoff.eq.for.u-v}, we use Lemma~\ref{lem:Boxu.est} to obtain
    \begin{equation}
        |\chi_{>H}(t-r) (\Box_\bfm Z^I u)(t,x)| \ls | \chi_{>H}(t-r)Z^{\leq N} (\Box_\bfm u)(t,x)| \ls_N
        \begin{cases}
            \ep^2 \brk{t}^{-4+C_N\ep} & \hbox{for $r \leq \f t 2$} \\
            \ep^2 \brk{t-r}^{-2} \brk{t}^{-2+C_N\ep}& \hbox{for $r \geq \f t 2$}.
        \end{cases}
    \end{equation}
    Noticing that this term is manifestly supported in $t-r \geq \f H2$, we see that the term is acceptable.

    We turn to the second term on the right-hand side of \eqref{eq:cutoff.eq.for.u-v}. Note that this term is supported in $\{\f H2 \leq t-r\leq H\}$. We first make a few observations on $D_{T,X} \cap \{\f H2 \leq t-r\leq H\}$:
    \begin{enumerate}
        \item $D_{T,X} \cap \{(t,x):\f H2 \leq t-r\leq H\}\subset \{(t,x): T-|X| \leq t+r \leq T+|X|, \f H2\leq  t-r\leq H\}$.
        \item As a result, $|t-r| \ls t^{\f 34}$ in $D_{T,X} \cap \{\f H2 \leq t-r\leq H\}$. In particular, Lemma~\ref{lem:dvru.in.wave.zone} is applicable when $t$ is sufficiently large.
        \item $\brk{r}$ and $\brk{t}$ are comparable in this region.
    \end{enumerate}
    Applying Lemma~\ref{lem:dvru.in.wave.zone}, we obtain
    \begin{equation}
        \Big|  \Big((\rd_t + \rd_r)(rZ^I (u-v))(t,x) \Big) \Big| \ls_N \ep \brk{t}^{-2+C_N\ep} \brk{t-r}+ \ep(1 + \ln \brk{t-r}) \brk{t}^{-1+C_N\ep} \brk{t-r}^{-1}.
    \end{equation}
    (Note that we only apply Lemma~\ref{lem:dvru.in.wave.zone} when $t$ is large. For finite $t$, it suffices to use \eqref{dy:est:ptw:lindblad}.)
    Using that $|\chi'_H(t-r)| \ls 1_{\{\f H2 \leq t-r\leq H\}} \brk{t-r}^{-1}$, and that $\brk{r}$ and $\brk{t}$ are comparable, it thus follows that the second term on the right-hand side of \eqref{eq:cutoff.eq.for.u-v} is acceptable. \qedhere
\end{proof}

\begin{proposition}\label{prop:main.error}
    Fix $T,X$ such that $|X| - T \leq R_0$ and $\brk{T-|X|} \geq \ep^{-2}$. For any $\eta \in (0, 1)$ and $N\in \bbZ_{\geq 0}$, the following holds when $\ep\ll_N 1$ for some $C_N >0$:
    \begin{equation}
        |Z^{\leq N} (u-v)|(T,X) \ls_{R_0,N} \ep \eta^{-2}  T^{C_N\ep} \min\{\brk{T-|X|}^{-\f 32+\f{\eta}{2}}, \ep^{-\f 12 -\f \eta 2} \brk{|X|}^{-1} \brk{T-|X|}^{-\f 12+\f{\eta}{2}} \}.
    \end{equation}
\end{proposition}
\begin{proof}
    For $H \in [1, \f{(T-|X|)^{\f 34}}3]$ to be chosen, we consider the equation $\Box_\bfm(\chi_{>H}(t-r) Z^I (u-v)(t,x))$ in Proposition~\ref{jl:prop:main.terms.for.cutoff}. By the strong Huygens principle (Lemma~\ref{jl:lem:Huygens}), it suffices to consider the terms in $D_{T,X}$. Denote by $\Box_\bfm^{-1} F$ by the forward solution to $\Box_\bfm \phi = F$. We consider the contribution from each term on the right-hand side in the equation in Proposition~\ref{jl:prop:main.terms.for.cutoff}.

    For the term $\ep\brk{t}^{-4+C_N \ep}$ in $D_{T,X} \cap \{r \leq \f t2\}$, we write
    $$\ep 1_{D_{T,X}} 1_{\{r\leq \f t2\}}\brk{t}^{-4+C_N\ep} \ls \ep \brk{t-r}^{-2+C_N\ep + \eta} \brk{r}^{-2-\eta}$$
    and then apply Proposition~\ref{jl:prop:invert.wave} with $B=\ep$, $\alp = 2 - C_N\ep - \eta$, $\eta_0 = \eta$ to obtain
    \begin{equation}\label{eq:u-v.main.err.1}
        \begin{split}
            &\: \Big| \Big(\Box_{\bfm}^{-1} (1_{D_{T,X}} 1_{\{r\leq \f t2\}} \ep \brk{t}^{-4+C_N\ep})\Big)(T,X) \Big| \\
            \ls &\: \ep \eta^{-1} \min \{ \brk{|T|-|X|}^{-2+C_N\ep+\eta}, \ep^{-\f 12 -\f \eta 2} \brk{|X|}^{-1} \brk{T-|X|}^{-1+C_N \ep + \eta} \},
        \end{split}
    \end{equation}
    which is much better than the estimate we need.

    We next consider the term $\ep 1_{\{t-r\geq \f H2\}} (1 + \ln \brk{t-r}) \brk{t-r}^{-2} \brk{r}^{-2+C_N\ep}$ in $D_{T,X} \cap \{r \geq \f t2\}$. Noting that in $D_{T,X} \cap \{t-r\geq \f H2\} \cap \{r\geq \f t2\}$, we have
    $$r \leq T + |X| - t \leq 2T + R_0 \ls_{R_0} T.$$
    Hence,
    \begin{align*}
        & \ep 1_{D_{T,X}} 1_{\{r\geq \f t2\}} 1_{\{t-r\geq \f H2\}} (1 + \ln \brk{t-r}) \brk{t-r}^{-2} \brk{r}^{-2+C_N\ep} \\
        \ls_{R_0} & \ep 1_{\{r\geq \f t2\}} H^{-1+\eta} T^{C_N\ep} (1 + \ln \brk{t-r}) \brk{t-r}^{-1-\eta} \brk{r}^{-2-\ep} \\
        \ls_{R_0} & \ep 1_{\{r\geq \f t2\}} \eta^{-1} H^{-1+\eta} T^{C_N\ep} \brk{t-r}^{-1-\frac{\eta}{2}} \brk{r}^{-2-\ep},
    \end{align*}
    where we used the simple bound $1 + \ln \brk{t-r} \aleq \eta^{-1} \brk{t-r}^{\frac{\eta}{2}}$ on the last line.
    Applying Proposition~\ref{jl:prop:invert.wave} with $B = \ep H^{-1+\eta} T^{C_N\ep}$, $\alp = 1+\ep$, $\eta_0 = \frac{\eta}{2}$, we obtain
    \begin{equation}\label{eq:u-v.main.err.2}
        \begin{split}
            &\: \Big| \Big(\Box_{\bfm}^{-1} (\ep 1_{D_{T,X}} 1_{\{r\geq \f t2\}} 1_{\{t-r\geq \f H2\}} \brk{t-r}^{-2} \brk{r}^{-2+C_N\ep})\Big)(T,X)\Big| \\
            \ls_{R_0} &\: \ep \eta^{-2} T^{C_N\ep} H^{-1+\eta} \min \{  \brk{T-|X|}^{-1-\ep}, \ep^{-1} \brk{|X|}^{-1} \brk{T-|X|}^{-\ep} \}.
        \end{split}
    \end{equation}

    Finally, we consider the term $\ep 1_{\{\f H2 \leq t-r\leq H\}} \brk{r}^{-3+C_N\ep}$ in $D_{T,X} \cap \{r \geq \f t2\}$. For this we used $t-r\leq H$ to obtain
    $$\ep 1_{D_{T,X}} 1_{\{r\geq \f t2\}} 1_{\{\f H2 \leq t-r\leq H\}} \brk{r}^{-3+C_N\ep} \ls \ep 1_{\{r\geq \f t2\}} H^{1+\eta} \brk{t-r}^{-1-\eta} \brk{r}^{-3+C_N\ep}.$$
    Hence, applying Proposition~\ref{jl:prop:invert.wave} with $B = \ep H^{1+\eta}$, $\alp = 2-C_N \ep$, $\eta_0=\eta$, we obtain
    \begin{equation}\label{eq:u-v.main.err.3}
        \begin{split}
            &\: \Big| \Big(\Box_{\bfm}^{-1} (\ep 1_{D_{T,X}} 1_{\{r\geq \f t2\}} 1_{\{\f H2 \leq t-r\leq H\}} \brk{r}^{-3+C_N\ep})\Big)(T,X) \Big| \\
            \ls &\: \ep \eta^{-1} H^{1+\eta} \min \{  \brk{T-|X|}^{-2+C_N\ep}, \brk{|X|}^{-1} \brk{T-|X|}^{-1+C_N\ep} \}.
        \end{split}
    \end{equation}

    We now use \eqref{eq:u-v.main.err.2} and \eqref{eq:u-v.main.err.3} with $H =\brk{T-|X|}^{\f{1}{2}}$ and $H = \ep^{-\f 12}\brk{T-|X|}^{\f{1}{2}}$.
    Notice that since $\brk{T-|X|} \geq \ep^{-2}$, we have $H \leq (T-|X|)^{\f 34}$ as needed. Using $H =\brk{T-|X|}^{\f{1}{2}}$, we obtain from the first bounds in \eqref{eq:u-v.main.err.2} and \eqref{eq:u-v.main.err.3} that
\begin{equation}\label{eq:u-v.main.err.4}
    \begin{split}
        \hbox{\eqref{eq:u-v.main.err.2} + \eqref{eq:u-v.main.err.3}}\ls_{R_0} &\: \ep \eta^{-2} T^{C_N\ep} \brk{T-|X|}^{-\f 32+\f{\eta}{2}},
    \end{split}
    \end{equation}
    where we have used $\brk{T-|X|}^{C_N\ep} \ls T^{C_N\ep}$.

    Similarly, using $H = \ep^{-\f 12}\brk{T-|X|}^{\f{1}{2}}$, we obtain from the second bounds in \eqref{eq:u-v.main.err.2} and \eqref{eq:u-v.main.err.3} that
    \begin{equation}\label{eq:u-v.main.err.5}
        \begin{split}
            \hbox{\eqref{eq:u-v.main.err.2} + \eqref{eq:u-v.main.err.3}}\ls_{R_0} &\: \ep^{\f 12-\f \eta 2} \eta^{-2} T^{C_N\ep} \brk{|X|}^{-1} \brk{T-|X|}^{-\f 12+\f{\eta}{2}} .
        \end{split}
    \end{equation}

    Combining \eqref{eq:u-v.main.err.4} and \eqref{eq:u-v.main.err.5}, we have
    \begin{equation}
    \begin{split}
        \hbox{\eqref{eq:u-v.main.err.2} + \eqref{eq:u-v.main.err.3}}\ls_{R_0} &\: \ep \eta^{-2} T^{C_N \ep} \min\{\brk{T-|X|}^{-\f 32+\f{\eta}{2}}, \ep^{-\f 12 - \f \eta 2} \brk{|X|}^{-1} \brk{T-|X|}^{-\f 12+\f{\eta}{2}} \}.
    \end{split}
    \end{equation}
    We have thus checked that all terms are acceptable and the proposition is proven. \qedhere
\end{proof}

Theorem~\ref{thm:main} follows from Proposition~\ref{prop:main.error}.

\section{Asymptotics in finite-$|x|$ region (Proof of Corollary~\ref{cor:finite})} \label{sec:asymp-v}

\subsection{Proof of Corollary~\ref{cor:finite}} \label{subsec:asymp-v}
Before we turn to the proof of Corollary~\ref{cor:finite}, we first show that in the finite-$|x|$ region, the lowest spherical mode dominates, i.e., the higher spherical modes decay faster. The idea is that for $\ell \geq 1$, we can trade additional powers of $r$ for powers of $t$ using elliptic estimates; see \cite{AAGPrice, MTT} and \cite[Section~7]{LO}. (For the purpose of Corollary~\ref{cor:finite}, we only need the improvement for $\bbS_{(\geq 1)}$, but it is not necessary to obtain the improvement for $\bbS_{(\geq 2)}$. However, the improved estimates after applying $\bbS_{(\geq 2)}$ will be used in the proof of Corollary~\ref{cor:main.2} in the next section.)

\begin{lemma}\label{lem:elliptic.u-v}
For any $\eta >0$ and $N\in \bbZ_{\geq 0}$, the following improved decay rates for higher angular modes hold in the region $r \leq \f t2$ when $\ep \ll_N 1$:
\begin{enumerate}
\item $$\sum_{a+b+|\alp|\leq N} | (t\rd_t)^a(r\rd_r)^b \Omg^\alp \bbS_{(\geq 1)} (u-v)|(t,x) \ls_{N} \ep \eta^{-2} r^{1-\f \eta 2} \brk{t}^{-\f 52+\eta+C_{N} \ep},$$
\item $$\sum_{a+b+|\alp|\leq N} | (t\rd_t)^a(r\rd_r)^b \Omg^\alp \bbS_{(\geq 2)} (u-v)|(t,x) \ls_{N} \ep \eta^{-3} r^{2-\f \eta 2} \brk{t}^{-\f 72+\eta+C_{N} \ep}.$$
\end{enumerate}
\end{lemma}
\begin{proof}
\pfstep{Step~1: General elliptic estimates and Sobolev embedding} We first use a standard weighted elliptic estimate. For $\ell \in \{ 1, 2\}$, suppose we are given a function $\varphi = \bbS_{(\geq \ell)} \varphi$. The following elliptic estimate holds for $h = \Delta_{\bfe} \varphi$ (see for instance \cite[Lemma~7.7]{LO}):
\begin{equation}\label{eq:elliptic.from.LO}
        \begin{split}
            \int \Big( ((r\rd_r)^2 \varphi)^2 + |\mathring{\slashed{\nabla}}{}^2 \varphi|^2 + |\mathring{\slashed{\nabla}} \varphi|^2 + r^2 |\mathring{\slashed{\nabla}} \rd_r \varphi|^2 + \varphi^2 \Big) r^{2\gamma+2} \, \ud r \ud \mathring{\slashed{\sigma}}
            \ls_\gamma  \int h^2 r^{2\gamma+6} \, \ud r \ud \mathring{\slashed{\sigma}},
        \end{split}
    \end{equation}
    where $-\f 32-\ell < \gamma < -\f 12+\ell$. Notice that for every rotation vector field $\Omg_{ij}$, $\Delta_\bfe \Omg_{ij} \varphi = \Omg_{ij} h$.
    Moreover, $\bbS_{(\ell)} \Omg_{ij} \bbS_{(\ell)} = \Omg_{ij} \bbS_{(\ell)}$ (since $\mathring{\slashed{\Delta}} \Omg_{ij}\bbS_{(\ell)} \varphi = \Omg_{ij} \mathring{\slashed{\Delta}} \bbS_{(\ell)} \varphi = \ell(\ell+1) \Omg_{ij} \varphi$). Thus, \eqref{eq:elliptic.from.LO} also applies to $\Omg_{ij} \varphi$ with the same range of exponents for $\gamma$. As a result, we have
    \begin{equation}\label{eq:elliptic.from.LO.w}
        \begin{split}
            \int \Big(|\mathring{\slashed{\nabla}}{}^2 \varphi|^2 + |\mathring{\slashed{\nabla}} \varphi|^2 + r^2|\mathring{\slashed{\nabla}}{}^2 \rd_r \varphi|^2 + r^2|\mathring{\slashed{\nabla}} \rd_r \varphi|^2 + \varphi^2 \Big) r^{2\gamma+2} \, \ud r \ud \mathring{\slashed{\sigma}}
            \ls_\gamma  \int \Big(|\mathring{\slashed{\nabla}} h|^2 + h^2\Big) r^{2\gamma+6} \, \ud r \ud \mathring{\slashed{\sigma}}.
        \end{split}
    \end{equation}

Since $\varphi = \bbS_{(\geq \ell)} \varphi$ for $\ell \geq 1$, we have $\varphi(0) = 0$. Hence, using the fundamental theorem of calculus (in $r$ for each fixed $\omg$), the Cauchy--Schwarz inequality, and the Sobolev embedding on $\bbS^2$, we have
\begin{equation}\label{eq:elliptic.Sobolev}
\begin{split}
&\: |\varphi|(r,\omg) \leq \Big| \int_{0}^r \rd_r \varphi(r',\omg) \, \ud r'\Big| \\
\leq &\: \Big(\int_{0}^r (\rd_r \varphi)^2(r',\omg) (r')^{2\gamma+4} \, \ud r' \Big)^{\f 12} \Big( \int_{0}^r (r')^{-2\gamma-4}\, \ud r'\Big)^{\f 12} \\
\ls_\gamma &\: \Big(\int_{0}^r (\rd_r \varphi)^2(r',\omg) (r')^{2\gamma+4} \, \ud r' \Big)^{\f 12} r^{-\gamma-\f 32} \\
\ls_\gamma &\:  \Big(\int_{\bbS^2} \int_{0}^r \Big((\rd_r \varphi)^2 + |\mathring{\slashed{\nabla}} \rd_r \varphi|^2 + |\mathring{\slashed{\nabla}}{}^2 \rd_r \varphi|^2 \Big)(r',\omg) (r')^{2\gamma+4} \, \ud r' \ud \mathring{\slashed{\sigma}}\Big)^{\f 12} r^{-\gamma-\f 32}.
\end{split}
\end{equation}
for $\gmm < - \f 32$. Combining \eqref{eq:elliptic.Sobolev} with \eqref{eq:elliptic.from.LO.w}, we thus obtain
\begin{equation}\label{eq:elliptic.main.pointwise}
|\varphi|(r,\omg) \ls_\gamma \Big( \int_{\bbR^3} \Big(|\mathring{\slashed{\nabla}} h|^2 + h^2\Big) r^{2\gamma+6} \, \ud r \ud \mathring{\slashed{\sigma}} \Big)^{\f 12}r^{-\gamma-\f 32},\quad \hbox{for $-\f 32 -\ell < \gamma< - \f 32$}.
\end{equation}

\pfstep{Step~2: Application of the elliptic estimates} We now apply \eqref{eq:elliptic.from.LO} to $\bbS_{(\geq \ell)} (u-v)$ where $\ell \in \{ 1,2\}$. Using Lemma~\ref{lem:Boxu.est} and that $\Box_\bfm v = 0$, we write
$$\Delta_{\bfe} \bbS_{(\geq \ell)} (u-v) = \rd_t^2 \bbS_{(\geq \ell)} (u-v) + O(\ep^2 \brk{t}^{-4+C\ep})\quad \hbox{when $r \leq \f t 2$}.$$
Using the bounds in Proposition~\ref{prop:main.error} for $u-v$ (and recalling from Lemma~\ref{lem:vector.fields} that the vector field bounds give an extra $\brk{t}^{-1}$ decay for each $\rd_t$ derivative in this region), we thus have
$$|\Delta_{\bfe} \bbS_{(\geq \ell)} (u-v)|(t,x) \ls \ep \eta^{-2}\brk{t}^{-\f 72+ \f \eta 2+C\ep}\quad \hbox{when $r \leq \f t 2$}.$$
Introducing a cutoff $\chi_{<\f t4}(r)$ so that
\begin{equation}\label{eq:elliptic.cutoff}
\begin{split}
|\Delta_{\bfe} \bbS_{(\geq \ell)} \chi_{<\f t4}(r)(u-v)|(t,x) \ls &\: \ep \eta^{-2} \brk{t}^{-\f 72+ \f \eta 2+C\ep} + |[\Delta_{\bfe}, \chi_{<\f t4}(r)] (u-v) |(t,x) \\
\ls &\: \ep \eta^{-2} \brk{t}^{-\f 72+ \f \eta 2+C\ep},
\end{split}
\end{equation}
where we have again used Proposition~\ref{prop:main.error}.

We also have a version of \eqref{eq:elliptic.cutoff} with higher derivative bounds. We need to be careful that a general commuting vector field $Z$ needs not behave well with $\bbS_{(\geq \ell)}$ or $\Delta_{\bfe}$. We thus only restrict to commutation with $t\rd_t$, $S$ and $\Omg_{ij}$. Note that
\begin{itemize}
\item $[\Box_\bfm,S] = 2\Box_\bfm$, $t\rd_t$ commutes with $\Delta_{\bfe}$, and $\Omg_{ij}$ commutes with both $\Box_\bfm$ and $\Delta_\bfe$.
\item Higher order $S$, $\rd_t$ and $\Omg$ derivatives of $S\bbS_{(\geq \ell)} (u-v)$ are still bounded at $r = 0$ (even though for instance higher $\rd_i$ derivatives become singular). 
\end{itemize}
Hence, we have, for $a+b+|\alp| = N$,
\begin{equation*}
\Delta_\bfe ((t\rd_t)^a S^b \Omg^\alp \bbS_{(\geq \ell)} (u-v)) = (t\rd_t)^a \rd^2_{tt} (S^b  \Omg^\alp \bbS_{(\geq \ell)} (u-v)) + O(\ep^2\brk{t}^{-4+C_N\ep}).
\end{equation*}

Using the bounds in Proposition~\ref{prop:main.error} and introducing a cutoff as before, we thus obtain the following analogue of \eqref{eq:elliptic.cutoff} for $\ell \in \{1,2\}$:
\begin{equation}\label{eq:elliptic.cutoff.commuted}
\begin{split}
|\Delta_{\bfe} \bbS_{(\geq \ell)} \chi_{<\f t4}(r)(t\rd_t)^a S^b \Omg^\alp (u-v)|(t,x) \ls_N \ep \eta^{-2}\brk{t}^{-\f 72+ \f \eta 2+C_{N}\ep} \quad a+b+|\alp| \leq N +1.
\end{split}
\end{equation}
In particular, for $h = \Delta_{\bfe} \bbS_{(\geq \ell)} \chi_{<\f t4}(r)(t\rd_t)^a S^b \Omg^\alp (u-v)$ and $a+b+|\alp| \leq N$, we have
\begin{equation}
\begin{split}
\Big( \int_{\bbR^3} \Big(|\mathring{\slashed{\nabla}} h|^2 + h^2\Big) r^{2\gamma+6} \, \ud r \ud \mathring{\slashed{\sigma}} \Big)^{\f 12} \ls_N
(2\gamma+7)^{-1} \ep \eta^{-2}\brk{t}^{-\f 72+ \f \eta 2+C_{N}\ep} \brk{t}^{\gamma+\f 72} \quad \hbox{if $2\gamma+6 > -1$}.
\end{split}
\end{equation}
For $\ell \in \{1,2\}$, take $\gamma = -\f 32 -\ell + \f \eta 2$. Notice that for $\eta>0$ sufficiently small, $-\f 32-\ell < \gamma<-\f 32$ and $2\gamma+6>-1$ so that everything above applies. We thus obtain using \eqref{eq:elliptic.main.pointwise} that
\begin{equation}
\begin{split}
|\bbS_{(\geq \ell)} \chi_{<\f t4}(r)(t\rd_t)^a S^b \Omg^\alp (u-v)|(t,x) \ls_{N} &\: \begin{cases} \ep \eta^{-2}\brk{t}^{\gamma+ \f \eta 2+C_{N}\ep} r^{-\gamma-\f 32} & \hbox{if $\ell = 1$} \\
\ep \eta^{-3}\brk{t}^{\gamma+ \f \eta 2+C_{N}\ep} r^{-\gamma-\f 32} & \hbox{if $\ell = 2$}
\end{cases}\\
\ls_{N} &\: \begin{cases}
\ep \eta^{-2}\brk{t}^{-\f 32 -\ell+ \eta +C_{N}\ep} r^{\ell-\f \eta 2} & \hbox{if $\ell = 1$} \\
\ep \eta^{-3}\brk{t}^{-\f 32 -\ell+ \eta +C_{N}\ep} r^{\ell-\f \eta 2} & \hbox{if $\ell = 2$},
\end{cases}
\end{split}
\end{equation}
for $a+b+|\alp| = N$, as desired. \qedhere
\end{proof}

A very similar argument also gives improved estimates for higher angular modes of $v$ (and therefore also for higher angular modes of $u$):
\begin{lemma}\label{lem:elliptic.v}
For any $\eta >0$ and $N\in \bbZ_{\geq 0}$, the following improved decay rates for higher angular modes hold in the region $r \leq \f t2$ when $\ep \ll_N 1$:
\begin{enumerate}
\item $$\sum_{a+b+|\alp|\leq N} | (t\rd_t)^a(r\rd_r)^b \Omg^\alp \bbS_{(\geq 1)} v|(t,x) \ls_{N} \ep \eta^{-2} r^{1-\f \eta 2} \brk{t}^{-2+\f \eta 2+C_{N} \ep},$$
\item $$\sum_{a+b+|\alp|\leq N} | (t\rd_t)^a(r\rd_r)^b \Omg^\alp \bbS_{(\geq 2)} v|(t,x) \ls_{N} \ep \eta^{-3} r^{2-\f \eta 2} \brk{t}^{-3+ \f \eta 2+C_{N} \ep}.$$
\end{enumerate}
The same estimate also holds with $v$ replaced by $u$.
\end{lemma}
\begin{proof}
    Once the estimate is proven for $v$, that for $u$ follows from Lemma~\ref{lem:elliptic.u-v}. Since $\Box_\bfm v = 0$, we have
    $$\Delta_{\bfe} \bbS_{(\geq \ell)} v = \rd_t^2 \bbS_{(\geq \ell)} v \quad \hbox{when $r \leq \f t 2$}.$$
    Using Lemma~\ref{lem:v.vector.fields}, we derive the following analogue of \eqref{eq:elliptic.cutoff.commuted}:
    \begin{equation}\label{eq:elliptic.cutoff.commuted.v}
\begin{split}
|\Delta_{\bfe} \bbS_{(\geq \ell)} \chi_{<\f t4}(r)(t\rd_t)^a S^b \Omg^\alp v|(t,x) \ls \ep \brk{t}^{-3+C_{N}\ep} \quad a+b+|\alp| \leq N.
\end{split}
\end{equation}
    The remainder of the argument proceeds as in Lemma~\ref{lem:elliptic.u-v} using \eqref{eq:elliptic.main.pointwise}. \qedhere
\end{proof}

We are now ready to turn to the proof of Corollary~\ref{cor:finite}.

\begin{proof}[Proof of Corollary~\ref{cor:finite}]
    Since $v_{(\geq 1)}$ decays faster (by Lemma~\ref{lem:elliptic.v}), it suffices to consider $v_{(0)}$, which is a function of $(t,r)$ alone. Notice that
    $\Box_\bfm v_{(0)} = 0$ is equivalent to $-(\rd_t - \rd_r)(\rd_t + \rd_r)(rv_{(0)}) = 0$ and thus $(\rd_t + \rd_r)(rv_{(0)})$ is a function of $t+r$ alone. Comparing with the characteristic initial data when $t = r$ in \eqref{eq:v.def}, we thus obtain
    $$(\rd_t + \rd_r)(rv_{(0)})(t,r) = \f{2\chi_{>10}(\f{t+r}{2}) \mathfrak L_{(0)}(\f{t+r}2)}{t+r}.$$
    Using the condition $(rv_{(0)})(t,0) = 0$ and integrating along the integral curves of $\rd_t + \rd_r$, we obtain
    \begin{equation}\label{eq:rv0.on.finite.r}
    (r v_{(0)})(t,r) = \int_{\f{t-r}{2}}^{\f{t+r}{2}} \, \f{\chi_{>10}(\rho) \mathfrak L_{(0)}(\rho)}{\rho}\ud \rho.
    \end{equation}
    Recall the definition of $\frkL$ in \eqref{eq:frkL.def}. We see that for $\rho \in [\f{t-r}{2},\f{t+r}{2}]$, $t$ large and $r$ bounded, we have
    \begin{equation}
        \Big| \f 1\rho \mathfrak L_{(0)}(\rho) - \f 2 t\mathfrak L_{(0)}(\f t2) \Big| \ls_{R} \ep^2 t^{-2+C\ep}.
    \end{equation}
    Plugging this into \eqref{eq:rv0.on.finite.r}, we obtain
    \begin{equation}
        \Big| v_{(0)}(t,r) - \f 2 t\mathfrak L_{(0)}(\f t2)\Big| \ls_{R} \ep^2 t^{-2+C\ep}
    \end{equation}
    for all $t$ sufficiently large. Combining this with Theorem~\ref{thm:main} and part 1 of Lemma~\ref{lem:elliptic.v}, we obtain
    \begin{equation}
        \begin{split}
            &\: \Big| u(t,x) - \f 2 t\mathfrak L_{(0)}(\f t2) \Big|\\
            \leq &\: \Big| u(t,x) - v(t,x) \Big| + \Big| v(t,x) - v_{(0)}(t,x)\Big| + \Big| v_{(0)}(t,x) - \f 2 t\mathfrak L_{(0)}(\f t2) \Big|
            \ls_{R}  \ep \eta^{-2} t^{-\f 32+C\ep+\eta}.
        \end{split}
    \end{equation}
    Finally, plugging in the definition of $\frkL$ in \eqref{eq:frkL.def} yields the desired conclusion. \qedhere
\end{proof}

\subsection{An alternative proof} \label{dy:sec:finite-alt}
In this subsection, we prove a slightly different version of Corollary \ref{cor:finite} by applying only the results in \cite{yu2024timelike}. To be more specific, we will prove that for each $R>0$,
\begin{align}\label{est:altpf:cor:finite}
\sup_{|x|\leq R} \Big| u(t,x) -\frac{2}{t}\mathfrak{L}_{(0)}(\frac{t}{2})\Big|
        \ls_{R} t^{-\frac{4}{3}+C\eps},\qquad \forall t\geq 4e^{\delta/\eps}.
\end{align}
The only difference between \eqref{est:altpf:cor:finite} and the estimate in Corollary \ref{cor:finite} is the decay rate on the right side.

Since \begin{align*}
    |\partial u|\lesssim \brk{t-r}^{-1}|Z^1u|\lesssim  \eps \brk{t}^{-2+C\eps}
\end{align*}
whenever $t>0$ and $r<t/2$,
we have
\begin{align*}
    \sup_{|x|\leq R} |u(t,x)-u(t,0)|&\lesssim R\sup_{|x|\leq R}|\partial u(t,x)|\lesssim \eps R\brk{t}^{-2+C\eps}
\end{align*}
for each $R>0$. It thus suffices to prove \eqref{est:altpf:cor:finite} for $x=0$.

Recall \cite[Proposition 6.1]{yu2024timelike}. For each multiindex $I$, if we define $U_I$ inductively by $U_0=\wh{U}$ and \eqref{eq:UI.def}, then for all $t\geq 2e^{\delta/\eps}$,
\begin{align*}
\left| \int_{\mathbb{S}^2} U_I(\eps\ln t-\delta,-2t,\omega)\ \ud \mathring{\slashed{\sigma}}(\omega) \right|&\lesssim_I \eps^{-1}t^{-\frac{1}{3}+C_I\eps}.
\end{align*}
In particular, if $Z^I=S$,  we have
\begin{align*}
    U_I&=\eps\partial_s\wh{U}+q\partial_q\wh{U}-\wh{U}.
\end{align*}
Since $U_0=\wh{U}$, we conclude that for all $t\geq 2e^{\delta/\eps}$,
\begin{align}
\label{est:altpf:cor:finite:step1}
    \left| \int_{\mathbb{S}^2} (\eps\partial_s\wh{U}+q\partial_q\wh{U})(\eps\ln t-\delta,-2t,\omega)\ \ud \mathring{\slashed{\sigma}}(\omega) \right|&\lesssim \eps^{-1}t^{-\frac{1}{3}+C\eps}.
\end{align}

Moreover, for all $s\geq 0$, $q\leq 0$, and $\omega\in\mathbb{S}^2$, we have
\begin{align*}
    \eps\partial_s\wh{U}(s,q,\omega)&=-\int_q^\infty\frac{\eps}{2}G(\omega)\wh{A}(p,\omega;\eps)^2\exp(\frac{1}{2}G(\omega)\wh{A}(p,\omega;\eps)s)\ \ud p,\\
    q\partial_q\wh{U}(s,q,\omega)&=q\wh{A}(q,\omega;\eps)\exp(\frac{1}{2}G(\omega)\wh{A}(q,\omega;\eps)s).
\end{align*}
Here we apply \eqref{dy:defn:whAmuU}. Since $|\wh{A}|\lesssim \brk{q}^{-1+C\eps}$ and $|e^x-1|\lesssim |x|e^{|x|}$, we have
\begin{align*}
    \eps\partial_s\wh{U}(s,q,\omega)&=-\int_{-\infty}^\infty\frac{\eps}{2}G(\omega)\wh{A}(p,\omega;\eps)^2\exp(\frac{1}{2}G(\omega)\wh{A}(p,\omega;\eps)s)\ \ud p+O(\eps\int_{-\infty}^q \brk{p}^{-2+C\eps}e^{Cs}\ \ud p)\\
    &=\eps^{-1}\mathfrak{L}(e^{\frac{s+\delta}{\eps}},\omega)+O(\eps\brk{q}^{-1+C\eps}e^{Cs}),\\
    q\partial_q\wh{U}(s,q,\omega)&=q\wh{A}(q,\omega;\eps)+q\wh{A}(q,\omega;\eps)(\exp(\frac{1}{2}G(\omega)\wh{A}(q,\omega;\eps)s)-1)\\
    &=q\wh{A}(q,\omega;\eps)+O(|q\wh{A}|\cdot |\wh{A}s|e^{Cs})=q\wh{A}(q,\omega;\eps)+O(\brk{q}^{-1+C\eps}se^{Cs}).
\end{align*}
By plugging these estimates into \eqref{est:altpf:cor:finite:step1}, we obtain
\begin{align*}
    \left|\int_{\mathbb{S}^2}\eps^{-1} \mathfrak{L}(t,\omega)\ \ud \mathring{\slashed{\sigma}}(\omega)-2t\int_{\mathbb{S}^2}\wh{A}(-2t,\omega;\eps)\ \ud \mathring{\slashed{\sigma}}(\omega)\right|&\lesssim \eps^{-1}t^{-\frac{1}{3}+C\eps}+ t^{-1+C\eps}\lesssim \eps^{-1}t^{-\frac{1}{3}+C\eps}.
\end{align*}

Finally, we recall from \cite[Theorem~1]{yu2024timelike} that
\begin{align*}
    u(t,x)&=\frac{\eps}{2\pi}\int_{\mathbb{S}^2}\wh{A}(x\cdot \omega-t,\omega;\eps)\ \ud \mathring{\slashed{\sigma}}(\omega)+O(\brk{|x|-t}^{-2}t^{C\eps})
\end{align*}
whenever $t\geq e^{\delta/\eps}$ and $|x|<t$. Thus, for $t\geq 2e^{\delta/\eps}$ and $x=0$, we have
\begin{align*}
    u(2t,0)&=\frac{\eps}{2\pi}\int_{\mathbb{S}^2}\wh{A}(-2t,\omega)\ \ud \mathring{\slashed{\sigma}}(\omega)+O(t^{-2+C\eps})\\
    &=\frac{1}{4\pi t}\int_{\mathbb{S}^2}\mathfrak{L}(t,\omega)\ \ud \mathring{\slashed{\sigma}}(\omega)+O(\eps t^{-1}\cdot \eps^{-1}t^{-\frac{1}{3}+C\eps}+t^{-2+C\eps})\\
    &=t^{-1}\mathfrak{L}_{(0)}(t)+O(t^{-\frac{4}{3}+C\eps}).
\end{align*}
We finish the proof by replacing $t$ with $\frac{t}{2}$, which explains why we need $t\geq 4e^{\delta/\eps}$.

\section{Asymptotics of the radiation field (Proof of Corollary~\ref{cor:rad-field})} \label{sec:rad-field}

We begin by justifying the existence of the radiation field of $Z^{I} v$ and its properties.
\begin{lemma} \label{lem:B-vf}
Consider a multiindex $I$. If $\eps\ll_{|I} 1$, then the limit
\begin{equation*}
	B_{I}(p, \omg;\ep) := \lim_{T \to \infty} \eps^{-1} (r (\rd_{t} - \rd_{r}) Z^{I} v) \restriction_{(t, r, \frac{x}{\abs{x}}) = (T, p + T, \omg)}
\end{equation*}
exists and the convergence is uniform in $(p, \omg)$ lying in any compact subset of $\bbR \times \bbS^{2}$. Furthermore, we have the error bound
\begin{equation} \label{eq:B-vf-error}
    \abs{r (\rd_{t} - \rd_{r}) Z^{I} v (t,x)- \eps B_{I}(\abs{x} - t, x/\abs{x}; \eps)} \aleq_{\abs{I}} \eps \brk{t}^{-1 + C_{\abs{I}} \eps} \quad \hbox{ when $\frac{t}{2} < r < t$},
\end{equation}
as well as the upper bound
\begin{equation} \label{eq:B-vf}
    \abs{B_{I}(p,\omg; \eps)} \aleq_{\abs{I}} \brk{p}^{-1 + C_{\abs{I}} \eps} \quad \hbox{ when } p < 0.
\end{equation}
\end{lemma}
Observe that $B$ in Corollary~\ref{cor:rad-field} coincides with $-\frac{1}{2} B_{0}$ in Lemma~\ref{lem:B-vf}. We also remark that $B_{I}$ enjoys a recursive relation akin to \eqref{eq:AI.def}, but we will not need this fact in what follows.
\begin{proof}
Since $\Box_{\bfm} v = 0$, we have $\Box_{\bfm} Z^{I} v = 0$, and thus \eqref{eq:wave4ZIv} holds. Hence, in $\frac{t}{2} < r < t$,
\begin{align*}
	\abs*{(\rd_{t} + \rd_{r}) (r (\rd_{t} - \rd_{r}) Z^{I} v)}
	\leq \abs*{(\rd_{t} + \rd_{r}) Z^{I} v} + \abs*{\frac{1}{r^{2}} \rslap (r Z^{I} v)}
	\aleq_{\abs{I}} \eps \brk{t}^{-2+C_{\abs{I}} \eps},
\end{align*}
where we used Lemmas~\ref{lem:vector.fields} and \ref{lem:v.vector.fields} (which require the smallness of $\eps$ depending on $\abs{I}$). Integrating this bound on rays of the form $(T, p+T, \omg)$ on $T \in [- 2p, \infty)$ (note that one endpoint lies in $\set{r = \frac{t}{2}}$), we obtain the existence of $B_{I}$. Starting instead from an arbitrary point in $\set{\frac{t}{2} < r < t}$, we obtain the error bound \eqref{eq:B-vf-error}. Finally, \eqref{eq:B-vf} follows from \eqref{eq:B-vf-error} and an application of Lemmas~\ref{lem:vector.fields} and \ref{lem:v.vector.fields} to estimate $r (\rd_{t} - \rd_{r}) Z^{I} v$ on $r = \tfrac{t}{2}$ (so that $t \sim r \sim r-t$).
\end{proof}

The following is the main result of this section, which is a stronger version of Corollary~\ref{cor:rad-field}.
\begin{proposition} \label{prop:rad-field}
For each multiindex $I$, the following holds whenever $\eta \in (0, 1)$ and $\eps \ll_{|I|} 1$:
    \begin{equation*}
\abs*{A_{I}(p, \omg) - B_{I}(p, \omg)} \aleq_{\abs{I}}  \eps^{-\frac{1}{2} - \frac{\eta}{2}} \eta^{-2} \brk{p}^{-\frac{3}{2} + \frac{\eta}{2} + C_{\abs{I}} \eps} \quad \hbox{ when } p < - \eps^{-2}.
\end{equation*}
\end{proposition}
\begin{proof}
Throughout the proof, we freely use the fact that in $\set{\tfrac{t}{2} < r < t}$, we have $t \sim r \sim t + r$. We first recall \eqref{dy:est:prop3.3:timelike.2} with $N = 0$:
\begin{align*}
	\abs{r (\rd_{t} - \rd_{r}) Z^{I} \wh{u} - \eps A_{I}(\abs{x} - t, x/\abs{x}; \eps)} \aleq_{\abs{I}} \eps (1 + \ln \brk{r - t}) \brk{r-t}^{-2} t^{C_{\abs{I}} \eps} + \eps t^{-1+C_{\abs{I}} \eps} \quad \hbox{ in } \Omg.
\end{align*}
Given $\gmm \in (0, 1)$, by \eqref{dy:est:diffutdu} (with $N = \abs{I} + 1$; recall that $\wh{u} = \td{u}$) and Lemma~\ref{lem:vector.fields} in the region $\frac{t}{2} < r < t$,
\begin{align*}
\abs{r (\rd_{t} - \rd_{r}) Z^{I} (u - \wh{u})} \aleq_{\gmm, \abs{I}} \eps \brk{t}^{-1 + C_{\abs{I}} \eps} \quad \hbox{ in } \Omg \cap \set{\abs{r - t} \aleq t^{\gmm}} \cap \set{\tfrac{t}{2} < r < t}.
\end{align*}
By Theorem~\ref{thm:main} and Lemma~\ref{lem:vector.fields} in the region $\frac{t}{2} < r < t$,
\begin{align*}
\abs{r (\rd_{t} - \rd_{r}) Z^{I} (u - v)} \aleq_{\abs{I}} \eps^{\frac{1}{2} - \frac{\eta}{2}} \eta^{-2}\brk{t-r}^{-\frac{3}{2} + \frac{\eta}{2}} \quad \hbox{ in } \set{t - r \geq \eps^{-2}} \cap \set{\tfrac{t}{2} < r < t}.
\end{align*}
Finally, by Lemma~\ref{lem:B-vf},
\begin{align*}
    \abs{r (\rd_{t} - \rd_{r}) Z^{I} v - \eps B_{I}(\abs{x} - t, x/\abs{x}; \eps)} \aleq_{\abs{I}} \eps \brk{t}^{-1 + C_{\abs{I}} \eps} \quad \hbox{ in } \set{\tfrac{t}{2} < r < t}.
\end{align*}
Observe that all these bounds hold in the region $\Omg \cap \set{\brk{r - t} \sim \brk{t}^{\frac{1}{2}}} \cap \set{t - r > \eps^{-2}}$ for $\eps$ sufficiently small, where the exponent $\frac{1}{2}$ has been chosen to optimize the bound. By concatenating these bounds in this common region, we obtain
    \begin{equation*}
\abs*{A_{I}(\abs{x}-t, x/\abs{x}; \eps) - B_{I}(\abs{x}-t, x/\abs{x}; \eps)} \aleq_{\abs{I}} \brk{t-r}^{-2 + C_{\abs{I}} \eps} + \eps^{-\frac{1}{2} - \frac{\eta}{2}} \eta^{-2} \brk{t-r}^{-\frac{3}{2} + \frac{\eta}{2} + C_{\abs{I}} \eps},
\end{equation*}
which holds for all $p = \abs{x} - t = r - t < -\eps^{-2}$, and implies the desired bound. \qedhere
\end{proof}
Recalling $A_{0} = - 2 \wh{A}$ and $B_{0} = - 2 B$, Corollary~\ref{cor:rad-field} follows from Proposition~\ref{prop:rad-field}.

It is well-known that by using the Radon transform, solutions to the linear wave equation with sufficiently regular and decaying initial data can be written in terms of the radiation field with a plane wave decomposition. The following lemma shows that this representation holds for $Z^I v$. This should be compared to the approximate solution 
\begin{align}\label{dy:defn:wI}
w_I(t,x)&=-\frac{\eps}{4\pi}\int_{\mathbb{S}^2} A_I(x\cdot \omega-t,\omega)\ \ud \mathring{\slashed{\sigma}}(\omega)
\end{align}
that is used in \cite[Theorem 1]{yu2024timelike}, where $A_I$ is defined in Section \ref{dy:sec:asytdu}.

\begin{lemma}\label{sec:rad-field:lem:dy:ZIvexpress}
For each multiindex $I$, we have 
\begin{align*}
    Z^Iv(t,x)&=-\frac{\eps}{4\pi}\int_{\mathbb{S}^2} B_I(x\cdot \omega-t,\omega)\ \ud \mathring{\slashed{\sigma}}(\omega)
\end{align*}
whenever $|x|<t$ and $t>0$.
\end{lemma}
\begin{proof}
Note that $\Box_\bfm Z^Iv=0$. By \cite[Proposition 4.1]{yu2024timelike}, if we set 
\begin{align*}
    \Phi(t,x):=[(-\partial_t+\partial_r)(rZ^Iv)](t,x),
\end{align*}
then for all $T>4(t+1)>0$ and $|x|<t$,
\begin{align}
    \label{sec:rad-field:lem:dy:ZIvexpress:est} Z^Iv(t,x)&=\frac{1}{4\pi}\int_{\mathbb{S}^2} \Phi(T,x-(T-t)\theta)\ \ud \mathring{\slashed{\sigma}}(\theta)+O(|x|\|\partial v(T)\|_{L^\infty(\{x\in\mathbb{R}^3:\ |x|<T\})}).
\end{align}
The implicit constant in $O(\cdot)$ is independent of $T$. Recall from Lemma \ref{lem:v.vector.fields} that $|Z^{\leq N}v|(t,x)\lesssim_N\eps \brk{t}^{-1+C_N\eps}$ whenever $|x|<t$. Since
\begin{align*}
    |x-(T-t)\theta|&\leq T-t+|x|\leq T,\\
    |x-(T-t)\theta|&\geq T-t-|x|\geq T-2t\geq T/2,
\end{align*} we have
\begin{align*}
    \Phi(T,x-(T-t)\theta)&=-(r(\partial_t-\partial_r)Z^Iv)(T,x-(T-t)\theta)+O(\eps T^{-1+C\eps})\\
    &=-\eps B_I(|x-(T-t)\theta|-T,\frac{x-(T-t)\theta}{|x-(T-t)\theta|})+O(\eps T^{-1+C\eps})
\end{align*}
and that
\begin{align*}
   |x|\|\partial v(T)\|_{L^\infty(\{x\in\mathbb{R}^3:\ |x|<T\})}&\lesssim t\cdot \eps T^{-1+C\eps}.
\end{align*} Here we apply \eqref{eq:B-vf-error}. Moreover, it was proved in \cite[Lemma 5.1]{yu2024timelike} that
\begin{align*}
    \lim_{T\to\infty} \left(|x-(T-t)\theta|-T,\frac{x-(T-t)\theta}{|x-(T-t)\theta|}\right)=\left(-t-x\cdot\theta,-\theta\right).
\end{align*}
By sending $T\to\infty$ in \eqref{sec:rad-field:lem:dy:ZIvexpress:est} and using the continuity of $B_I$ (by Lemma~\ref{lem:B-vf}), we finish the proof.
\end{proof}
In Theorem \ref{thm:main} and \cite[Theorem 1]{yu2024timelike}, we present two approximations, $Z^Iv$ and $w_I$, for $Z^Iu$ far inside the light cone. Thus, we can estimate the difference between $Z^Iv$ and $w_I$ by combining these theorems. However, we do not expect $A_I\equiv B_I$ and $Z^Iv\equiv w_I$ in general.

\section{Proof of corollaries on the rigidity of decay rates} \label{sec:rigidity}
\subsection{Representation formula of a coercive quantity of the radiation field}

In the following, we will need to analyze both $u$ and $v$ in terms of spherical harmonics. Recall the notations in Definition~\ref{def:spherical.harmonics} and we introduce some additional conventions as follows:
\begin{definition}\label{def:spherical.harmonics.2}
\begin{enumerate}
\item We fix the following orthonormal basis for the $\ell = 1$ and $\ell = 2$ eigenspaces, respectively:
$$\Big\{ Y_{(1),1}(\omg) :=\sqrt{\f{3}{4\pi}}\omg_1,Y_{(1),2}(\omg) :=\sqrt{\f{3}{4\pi}}\omg_2,Y_{(1),3}(\omg) :=\sqrt{\f{3}{4\pi}}\omg_3 \Big\}$$
and
\begin{align*}
\Big\{ Y_{(2), 1}(\omg) = \sqrt{\f{15}{4\pi}} \omg_1 \omg_2, Y_{(2), 2}(\omg) = \sqrt{\f{15}{4\pi}} \omg_2 \omg_3, Y_{(2), 3}(\omg) = \sqrt{\f{15}{4\pi}} \omg_3 \omg_1,\\
 Y_{(2), 4}(\omg) = \sqrt{\f{15}{16\pi}}(\omg_1^2 - \omg_2^2), Y_{(2), 5}(\omg) = \sqrt{\f{15}{16\pi}}(\omg_2^2 - \omg_3^2) \Big\}.
\end{align*}
\item Given a function $f$, define $f_{(1),i}$, $f_{(2),i}$ by
$$f_{(1)} = \sum_{i=1}^3 f_{(1),i} Y_{(1),i},\quad f_{(2)} = \sum_{i=1}^5 f_{(2),i} Y_{(2),i}.$$
In particular, we will often use the decomposition
\begin{equation}\label{eq:general.decomposition}
f = f_{(0)} + \sum_{i=1}^3 f_{(1),i} Y_{(1),i} + \sum_{i=1}^5 f_{(2),i} Y_{(2),i} + f_{(\geq 3)}.
\end{equation}
Note that according to our conventions, $f_{(1)}$ and $f_{(2)}$ depend on the angular variables, while $f_{(1),i}$ and $f_{(2),i}$ depend on $(t,r)$ alone.
\end{enumerate}
\end{definition}

We remind the reader that $u_{(0)}$ and $u_{(1)}$ (defined in Definition~\ref{def:spherical.harmonics}) are not to be confused with the initial functions $u_0$ and $u_1$!

\begin{proposition}\label{prop:rigidity.conserved.quantities}
In the region $t \geq 20$ and $r \leq t - 20$, the following holds:
    \begin{equation}
        \begin{split}
            &\:(\rd_ t+ \rd_r)(rv_{(0)})(t,r) \\
            =&\: -\f {\ep^2 }{8 \pi } \int_{\mathbb S^2}\int_{-\infty}^\infty G(\omega') A^2 (p,\omega')\Big(\f {t+r}2 \Big)^{-1+\frac{\ep}{2}G(\omega') A(p,\omega')} J(p,\omg';\ep) \ \ud p\, \ud \mathring{\slashed\sigma}(\omg'),
        \end{split}
    \end{equation}
    \begin{equation}
        \begin{split}
        &\: r^{-2} (\rd_t + \rd_r) \Big(r^2(\rd_t + \rd_r)(rv_{(1),i})\Big)(t,r) \\
        =&\: -\f { \ep^2 }{2}  \int_{\mathbb S^2}\int_{-\infty}^\infty Y_{(1),i}(\omg') G(\omega') A^2 (p,\omega') \Big(1 + \frac{\ep}{2}G(\omega') A(p,\omega')\Big)\Big(\f {t+r}2 \Big)^{-2+\frac{\ep}{2}G(\omega') A(p,\omega')} J(p,\omg';\ep) \ \ud p\, \ud \mathring{\slashed\sigma}(\omg'),
        \end{split}
    \end{equation}
	and
    \begin{equation}
        \begin{split}
        &\: r^{-4} (\rd_t + \rd_r) \Bigg( \Big(r^2(\rd_t + \rd_r)\Big)^2(rv_{(2),i})\Bigg)(t,r) \\
        =&\: -\f {\ep^2 }{2 }  \int_{\mathbb S^2}\int_{-\infty}^\infty Y_{(2),i}(\omg') G(\omega')A^2 (p,\omega') \Big(1 + \frac{\ep}{2}G(\omega') A(p,\omega')\Big)\\
        &\: \qquad\qquad\qquad\qquad\qquad \times \Big(2 + \frac{\ep}{2}G(\omega') A(p,\omega')\Big)\Big(\f {t+r}2 \Big)^{-3+\frac{\ep}{2}G(\omega')A(p,\omega')} J(p,\omg';\ep) \ \ud p\, \ud \mathring{\slashed\sigma}(\omg').
        \end{split}
    \end{equation}
    where for brevity, we have written $A (p,\omega')=A (p,\omega';\eps)$.
\end{proposition}
\begin{proof}
The main observation is that whenever $\Box_\bfm w = 0$,
\begin{align}
(\rd_t -\rd_r)(\rd_ t+ \rd_r)(rw_{(0)}) = &\: 0, \label{eq:conservation.0.mode} \\
(\rd_t -\rd_r)\Big( r^{-2} (\rd_t + \rd_r) \left(r^2(\rd_t + \rd_r)(rw_{(1),i})\right) \Big) = &\: 0, \quad i = 1,2,3, \label{eq:conservation.1.mode}\\
(\rd_t -\rd_r)\Big( r^{-4} (\rd_t + \rd_r) \left( \left(r^2(\rd_t + \rd_r)\right)^2(rw_{(2),i})\right)\Big) = &\: 0,\quad i = 1,2,3,4,5. \label{eq:conservation.2.mode}
\end{align}
The formulae \eqref{eq:conservation.0.mode}--\eqref{eq:conservation.2.mode} above are explicit computations that follow from
$$- \rd_{tt}^2 w_{(\ell)} + \rd_{rr}^2 w_{(\ell)} - \f {\ell(\ell+1)}{r^2} w_{(\ell)} = 0.$$
See \cite[equation (6.6)]{LO} for details of this computation. Combining \eqref{eq:conservation.0.mode}--\eqref{eq:conservation.2.mode}, the condition \eqref{eq:v.def} for the initial data for $v$, and the strong Huygens principle (Lemma~\ref{jl:lem:Huygens}), we obtain the desired conclusion. \qedhere
\end{proof}

Suitably differentiating the expressions from Proposition~\ref{prop:rigidity.conserved.quantities} above, we obtain the following:
\begin{proposition}\label{prop:A.in.terms.of.v}
Let
\begin{equation}\label{eq:frkA.def}
\frkA^2(p,\omg';\ep) = A^2 (p,\omega';\ep) \Big(1 + \frac{\ep}{2}G(\omega') A(p,\omega';\ep)\Big) \Big(2 + \frac{\ep}{2}G(\omega') A(p,\omega';\ep)\Big) J(p,\omg';\ep).
\end{equation}
(Note that the right-hand side is manifestly non-negative when $\ep$ is sufficiently small and thus the square root is well-defined.) Then the following holds:
\begin{enumerate}
\item
\begin{equation}\label{eq:A.in.terms.of.v.0}
\begin{split}
&\: \int_{\mathbb S^2} \int_{-\infty}^\infty G(\omg') \frkA^2(p,\omega';\ep) \Big(\f {t+r}2 \Big)^{-1+\frac{\ep}{2}G(\omega')A(p,\omega';\ep)} \, \ud p \,\ud \mathring{\slashed{\sigma}}(\omg') \\
= &\: -\f{8\pi }{\ep^2}\Big( \f{t+r}2\Big)^{-2}(\rd_t + \rd_r) \Bigg((\f{t+r}2)^2  (\rd_t + \rd_r) \Big((\f{t+r}2)^2(\rd_t + \rd_r)(rv_{(0)})\Big)\Bigg).
\end{split}
\end{equation}
\item For $i = 1,2,3$,
\begin{equation}\label{eq:A.in.terms.of.v.1}
\begin{split}
&\: \int_{\mathbb S^2} \int_{-\infty}^\infty Y_{(1),i}(\omega') G(\omg') \frkA^2(p,\omega';\ep) \Big(\f {t+r}2 \Big)^{-1+\frac{\ep}{2}G(\omega')A(p,\omega';\ep)} \, \ud p \,\ud \mathring{\slashed{\sigma}}(\omg') \\
= &\: -\f{2 }{\ep^2}\Big( \f{t+r}2\Big)^{-2}(\rd_t + \rd_r) \Bigg(\f{(t+r)^4}{2^4r^2}  (\rd_t + \rd_r) \Big(r^2(\rd_t + \rd_r)(rv_{(1),i})\Big)\Bigg).
\end{split}
\end{equation}
\item For $i = 1,\cdots, 5$,
\begin{equation}\label{eq:A.in.terms.of.v.2}
\begin{split}
&\: \int_{\mathbb S^2} \int_{-\infty}^\infty Y_{(2),i}(\omega') G(\omg') \frkA^2(p,\omg';\ep) \Big(\f {t+r}2 \Big)^{-1+\frac{\ep}{2}G(\omega')A(p,\omega';\ep)} \, \ud p \,\ud \mathring{\slashed{\sigma}}(\omg') \\
= &\: -\f{2 }{\ep^2} r^{-4}\Big( \f{t+r}2\Big)^{2} (\rd_t + \rd_r) \Bigg( \Big(r^2(\rd_t + \rd_r)\Big)^2(rv_{(2),i})\Bigg).
\end{split}
\end{equation}
\end{enumerate}
\end{proposition}

\subsection{Rigidity of decay rate in the self-similar region (Proof of Corollary~\ref{cor:main.1})}
The following simple calculus lemma is an adaption of \cite[Chapter 5 Problem 15]{Rudin} to a finite interval.
\begin{lemma}\label{lem:Rudin}
Let $f: [a,b]\to \mathbb R$ be a $C^2$ function. Then the following estimate holds for any $d \leq \f {b-a}2$:
$$\| f' \|_{L^\i([a,b])} \leq \f{2}{d} \| f \|_{L^\i([a,b])} + \f{d}{2}\| f'' \|_{L^\i([a,b])}.$$
\end{lemma}
\begin{proof}
Given an $x \in [a,b]$, pick $y \in [a,b]$ such that $|x-y|= d$ (which exists since $d \leq \f{b-a}2$).

By Taylor's theorem, there exists $\xi$ in between $x$ and $y$ such that
\begin{equation}
f(y) = f(x) + f'(x)(y-x) + \f 12 f''(\xi) (y-x)^2.
\end{equation}
Reshuffling, this implies
\begin{equation}
|f'(x)| \leq \f{|f(y)| + |f(x)|}{|y-x|} + \f 12 |f''(\xi)| |y-x| \leq \f{2}{d} \| f \|_{L^\i([a,b])} + \f{d}{2}\| f'' \|_{L^\i([a,b])}.
\end{equation}
Since $x \in [a,b]$ is arbitrary, this implies the desired bound. \qedhere
\end{proof}

We are now ready to prove Corollary~\ref{cor:main.1}.
\begin{proof}[Proof of Corollary~\ref{cor:main.1}]
In this proof, we allow all implicit constants to depend on $C_1$.

\pfstep{Step~1: Estimates for derivatives of $v$} We decompose $u$ and $v$ as in \eqref{eq:general.decomposition}. We will introduce the schematic notation that $u_{\bullet}$ denotes any of $\{u_{(0)}, u_{(1),i}, u_{(2),i}\}$, and similarly for $v_{\bullet}$.

To apply Lemma~\ref{lem:Rudin}, we will think of $u_{\bullet}$ as a one-variable function: for every fixed $(q_{\bfm} = t-r,\omg)$ (with $q_{\bfm} \geq 1$ sufficiently large), consider $u_{\bullet}(\sigma) = u_{\bullet}(t=\sigma , r= \sigma - q_{\bfm}, \omg)$ as a one-variable function with $u_{\bullet}: [(1-c_1)^{-1} q_\bfm, (1-c_2)^{-1} q_\bfm]$.

By \eqref{dy:est:ptw:lindblad} and Lemma~\ref{lem:vector.fields}, we know that there is an $\ep_0$ such that the following holds for all $\ep \in (0,\ep_0)$:
\begin{equation}\label{eq:main.cor.uell.fast.decay.vector.field.2}
\sup_{\sigma \in [(1-c_1)^{-1} q_\bfm, (1-c_2)^{-1} q_\bfm]} |u^{(k)}_{\bullet}|(\sigma) \ls \brk{q_\bfm}^{-1-k+C_*\ep},\quad k=1,2,3,
\end{equation}
where the superscripts ${}^{(k)}$ denotes the order of derivatives. We have labeled the constant by $C_*$ and it will be considered fixed for this proof. This constant depends only on $u_0$, $u_1$ and $g^{\alp\bt}$. We want to improve the decay in $q_\bfm$ for the higher derivatives. 

We now fix $C_0$ in the statement of the corollary.  Using \eqref{dy:est:hatA}, we can choose $C_0$ sufficiently large so that
\begin{equation}\label{eq:choose.C0}
\inf_{\omg' \in \bbS^2} \Big(\f{C_0}8-C_*+G(\omega')A(p,\omega';\ep)\Big) \geq 0,
\end{equation}
uniformly for all small enough $\ep$. Note that indeed $C_0$ depends on $u_0$, $u_1$ and $g^{\alp\bt}$.

By projecting $\bbS_{(\leq 2)} u$ to individual spherical harmonics, \eqref{eq:main.cor.assumption} implies that
\begin{equation}\label{eq:main.cor.uell.fast.decay}
\sup_{\sigma \in [(1-c_1)^{-1} q_\bfm, (1-c_2)^{-1} q_\bfm]} |u_{\bullet}|(\sigma) \ls \brk{q_\bfm}^{-1-C_0 \ep}.
\end{equation}

Applying Lemma~\ref{lem:Rudin} with $f = u_{\bullet}$, $[a,b] = [(1-c_1)^{-1} q_\bfm, (1-c_2)^{-1} q_\bfm]$ and $d = \f 12 \Big( (1-c_2)^{-1} - (1-c_1)^{-1} \Big)  q_\bfm^{1-\f {C_0\ep} 2}$, and using \eqref{eq:main.cor.uell.fast.decay} and \eqref{eq:main.cor.uell.fast.decay.vector.field.2}, we obtain
\begin{equation}\label{eq:main.cor.uell.fast.decay.der}
    \sup_{\sigma \in [(1-c_1)^{-1} q_\bfm, (1-c_2)^{-1} q_\bfm]} |u'_{\bullet}|(\sigma) \ls \brk{q_\bfm}^{-2-\f {C_0 \ep} 2 + C_*\ep}.
\end{equation}
We could repeat this argument, but with $f = u_{\bullet}'$, $d = \f 12 \Big( (1-c_2)^{-1} - (1-c_1)^{-1} \Big) q_\bfm^{1-\f {C_0\ep} 4}$, and use \eqref{eq:main.cor.uell.fast.decay.der} and \eqref{eq:main.cor.uell.fast.decay.vector.field.2} to obtain
\begin{equation}
    \sup_{\sigma \in [(1-c_1)^{-1} q_\bfm, (1-c_2)^{-1} q_\bfm]} |u''_{\bullet}|(\sigma) \ls \brk{q_\bfm}^{-3-\f {C_0 \ep} 4 + C_*\ep}.
\end{equation}
The same argument can be repeated for one higher derivative so that
\begin{equation}
    \sup_{\sigma \in [(1-c_1)^{-1} q_\bfm, (1-c_2)^{-1} q_\bfm]} |u'''_{\bullet}|(\sigma) \ls \brk{q_\bfm}^{-4-\f {C_0 \ep} 8 + C_*\ep}.
\end{equation}

We now change back to the original coordinates and note that in the region of interest $q_\bfm$ is comparable to $t$. Hence, we conclude that
\begin{equation}\label{eq:self.similar.zone.extra.decay.u}
\sup_{r: c_1 t \leq r \leq c_2 t} |(\rd_t+\rd_r)^{k} u_{\bullet}|(t,r) \ls \brk{t}^{-1-k-2^{-k} C_0 \ep + C_* \ep},\quad k = 0,1,2,3.
\end{equation}

 Since the difference $u-v$ decays faster by Proposition~\ref{prop:main.error}, it follows that $v$ obeys a similar bound as \eqref{eq:self.similar.zone.extra.decay.u}:
\begin{equation}\label{eq:self.similar.zone.extra.decay.v}
\sup_{r: c_1 t \leq r \leq c_2 t} |(\rd_t+\rd_r)^{k} v_{\bullet}|(t,r) \ls \brk{t}^{-1-k-2^{-k} C_0 \ep + C_* \ep},\quad k = 0,1,2,3.
\end{equation}

Since $r$, $t$ and $t-r$ are all comparable in the region $\{c_1 t \leq r \leq c_2 t\}$, it follows from \eqref{eq:self.similar.zone.extra.decay.v} that
\begin{equation}\label{eq:self.similar.zone.extra.decay.v.conclude}
\sup_{r: c_1 t \leq r \leq c_2 t}|\hbox{RHS of \eqref{eq:A.in.terms.of.v.0}--\eqref{eq:A.in.terms.of.v.2}}| \ls \brk{t}^{-1-(\f{C_0}8-C_*)\ep}.
\end{equation}

\pfstep{Step~2: Anomalous decay of the key coercive quantity} Recall the definition of $\frkA^2$ in \eqref{eq:frkA.def}.

Note that $\{1\} \cup \{ Y_{(1),i} \}_{i=1}^3 \cup \{ Y_{(2), i}\}_{i=1}^5$ spans all polynomials in $\omg_1,\omg_2,\omg_3$ of degree $\leq 2$ by dimension considerations (keeping in mind that $\omg_1^2 + \omg_2^2 + \omg_3^2 \equiv 1$). Using the representation in Proposition~\ref{prop:A.in.terms.of.v}, it thus follows from \eqref{eq:self.similar.zone.extra.decay.v.conclude} that for any polynomial $p(\omg)$ on $\bbS^2$ of degree $\leq 2$,
\begin{equation}\label{eq:frkA.extra.decay.p}
\begin{split}
&\: \sup_{r: c_1 t \leq r \leq c_2 t} \Big| \int_{\mathbb S^2} \int_{-\infty}^\infty p(\omega') G(\omg') \mathfrak A^2(p,\omg';\ep) \Big(\f {t+r}2 \Big)^{-1+\frac{\ep}{2}G(\omega')A(p,\omega')} \, \ud p \,\ud \mathring{\slashed{\sigma}}(\omg') \Big|  \ls \brk{t}^{-1-(\f{C_0}8-C_*)\ep}.
\end{split}
\end{equation}

By \eqref{eq:G.intro.def}, $G(\omg')$ is a polynomial in $(\omg_1', \omg_2', \omg_3')$ of degree up to $2$. Hence, we can take $p = G$ (and then rearrange) to obtain
\begin{equation}\label{eq:main.cor.almost}
\begin{split}
&\: \sup_{r: c_1 t \leq r \leq c_2 t} \Big| \int_{\mathbb S^2} \int_{-\infty}^\infty G^2(\omega') \mathfrak A^2(p,\omg';\ep) \Big(\f {t+r}2 \Big)^{\f 12(\f{C_0}8-C_*)\ep+\frac{\ep}{2}G(\omega')A(p,\omega')} \, \ud p \,\ud \mathring{\slashed{\sigma}}(\omg') \Big| \ls \brk{t}^{-\f 12 (\f{C_0}8-C_*)\ep}.
\end{split}
\end{equation}

\pfstep{Step~3: Conclusion of the argument} For the sake of contradiction, suppose now that $A$ is not identically vanishing. Then, recalling the definition of $\frkA$ (see \eqref{eq:frkA.def}), we see that by continuity of $A$, there is an open set $\calU \subset \mathbb R\times \mathbb S^2$ in which $|\frkA| \geq b$ for some $b>0$. Since $G(\omg')$ is a non-zero polynomial, its zero set is a union of embedded curves on $\mathbb S^2$. In particular, by restricting $\calU$ further, we can assume that $|G(\omg')|\geq d>0$ in $\calU$. Hence, using also \eqref{eq:choose.C0}, we conclude that if $\f{t+r}2 \geq 1$, then
\begin{equation}
\hbox{LHS of \eqref{eq:main.cor.almost}} \geq b^2 d^2 \int_{\calU}  \, \ud p \,\ud \mathring{\slashed{\sigma}}(\omg').
\end{equation}
In particular, this is uniformly bounded below as $t\to \infty$. On the other hand, clearly
\begin{equation}
\lim_{t\to \infty} \hbox{RHS of \eqref{eq:main.cor.almost}} = 0.
\end{equation}
This is thus a contradiction. Hence, $A$ vanishes identically. We finally use that $A \equiv 0$ implies $u \equiv 0$, which was proven in \cite[Section~6.2.4]{MR4772266}. \qedhere
\end{proof}

\subsection{Rigidity of decay rate in a finite-$r$ region (Proof of Corollary~\ref{cor:main.2})}

As for the proof of Corollary~\ref{cor:main.1}, the key to Corollary~\ref{cor:main.2} will be to establish an analogue of \eqref{eq:self.similar.zone.extra.decay.v.conclude}. The present case is more subtle because, as we will see in the proof, Lemma~\ref{lem:Rudin} only gives improvements in the higher $\rd_t$ derivatives but not the higher $\rd_r$ derivatives.

\begin{lemma}\label{lem:finite.r.0}
    Let $R>0$, $\eta \in (0,1)$ be fixed. Denote by $C_*>0$ a constant such that
    \begin{align}
        \sup_{r\leq R}|(t\rd_t)^a(r\rd_r)^b (u_{(0)}-v_{(0)})| \ls &\: \ep \eta^{-2} \brk{t}^{-\f 32 + \f \eta 2+C_*\ep},\quad a+b \leq 4, \label{eq:vector.field.in.r.u-v.C_*}\\
            \sup_{r\leq R}|(t\rd_t)^a(r\rd_r)^b u_{(0)}| \ls &\: \ep \brk{t}^{-1+C_*\ep},\quad a+b \leq 4, \label{eq:vector.field.in.r.u.C_*}     \\
            \sup_{r\leq R}|(t\rd_t)^a(r\rd_r)^b \Box_\bfm u_{(0)}|\ls &\: \ep^2 \brk{t}^{-4+C_*\ep},\quad a+b \leq 1. \label{eq:vector.field.wave.in.r.u.C_*}
        \end{align}
        Such a $C_*$ exists for $\ep \ll 1$ by Theorem~\ref{thm:main},  \eqref{dy:est:ptw:lindblad} and Lemma~\ref{lem:Boxu.est} (and the fact that $t\rd_t$, $S$ and $\Box_\bfm$ all commute with $\bbS_{(0)}$).

    Suppose $\sup_{r \leq R} |u_{(0)}|\leq C_1 \brk{t}^{-1-C_0\ep}$. Then
    \begin{equation}\label{eq:finite.r.bound.RHS.0}
        \sup_{\f R2 \leq r \leq R} |\hbox{RHS of \eqref{eq:A.in.terms.of.v.0}}|\ls \brk{t}^{-1-(2^{-3} C_0-C_*)\ep},
    \end{equation}
    where the implicit constant depends on $u_0$, $u_1$, $\eta$, $R$, $C_*$ and $C_0$.
\end{lemma}
\begin{proof}
From now on, we allow all implicit constants to depend on $u_0$, $u_1$, $\eta$, $R$, $C_*$ and $C_0$.

\pfstep{Step~1: Improved estimates for derivatives of $u_{(0)}$} Our goal in this step is to show that
\begin{align}\label{eq:rdrrdtu0.prelim}
    \sup_{\f R2\leq r \leq R} |\rd_r^a \rd_t^b u_{(0)}|(t,r) \ls \brk{t}^{-1 -b - 2^{-b-a} C_0\ep + C_* \ep} ,\quad a+b\leq 3.
\end{align}

We first consider the $a = 0$ case in \eqref{eq:rdrrdtu0.prelim}, i.e., we will first prove
\begin{equation}\label{eq:rdtu0.prelim}
    \sup_{r\leq R}|\rd_t^b u_{(0)}|(t,r) \ls \brk{t}^{-1 -b - 2^{-b} C_0\ep + C_* \ep} ,\quad b = 0,1,2,3.
\end{equation}
The $b=0$ case of \eqref{eq:rdtu0.prelim} is an immediate consequence of \eqref{eq:finite.r.bound.RHS.0}. For the $b=1$ case, we fix $r$ and think of this as a one-variable function. For every dyadic interval $[2^k,2^{k+1}]$, we apply Lemma~\ref{lem:Rudin} with $[a,b] = [2^k,2^{k+1}]$, $d = 2^{k(1-2^{-1}C_0\ep)}$, $f = u_{(0)}(\cdot,r)$ and use the bounds \eqref{eq:vector.field.in.r.u.C_*} to obtain
\begin{equation*}
    \begin{split}
        &\: \sup_{r\leq R, t \in [2^k,2^{k+1}]} |\rd_t u_{(0)}|(t,x) \\
        \leq &\:  2^{1-k(1-2^{-1}C_0\ep)} \Big(\sup_{r\leq R, t \in [2^k,2^{k+1}]} |u_{(0)}|(t,x) \Big) + 2^{k(1-2^{-1}C_0\ep)-1} \Big(\sup_{r\leq R, t \in [2^k,2^{k+1}]} |\rd_t^2 u_{(0)}|(t,x) \Big) \\
        \ls &\: 2^{-k(1-2^{-1}C_0\ep)} 2^{k(-1-C_0\ep)} + 2^{k(1-2^{-1}C_0\ep)} 2^{k(-3+C_*\ep)} \ls 2^{-2k-(2^{-1}C_0-C_*)k\ep}  \ls \brk{t}^{-2-2^{-1}C_0+C_*}.
    \end{split}
\end{equation*}
This gives the $b=1$ case of \eqref{eq:rdtu0.prelim}. Repeating this argument inductively for higher derivatives yields \eqref{eq:rdtu0.prelim} for $b = 0,1,2,3$.

    Our next goal is to prove \eqref{eq:rdrrdtu0.prelim}, i.e., we upgrade \eqref{eq:rdtu0.prelim} to include also the $\rd_r$ derivatives.

    In the case $a = 0$, \eqref{eq:rdrrdtu0.prelim} reduces to \eqref{eq:rdtu0.prelim}. The claim now is that for $\rd_r$ derivatives, even though we do not gain extra powers of $t$, we are still able to improve the decay rate in \eqref{eq:vector.field.in.r.u.C_*} by $2^{-1}C_0\ep$. To show this, in contrast to the above, we use Lemma~\ref{lem:Rudin} for every fixed $t$ and consider the one-variable function $u_{(0)}(t,\cdot)$. More precisely, applying Lemma~\ref{lem:Rudin} with the interval $[\f R2, R]$, $d = \min\{ \f R{4}, t^{-2^{-b-1} C_0\ep}\}$, $f = \rd_t^b u_{(0)}(t,\cdot)$ and using the bounds \eqref{eq:rdtu0.prelim}, \eqref{eq:vector.field.in.r.u.C_*}, we have
    \begin{equation*}
        \begin{split}
            &\: \sup_{\f R2 \leq r \leq R} |\rd_r \rd_t^b u_{(0)}|(t,r) \\
            \ls &\: \brk{t}^{2^{-b-1} C_0\ep} \sup_{\f R2 \leq r \leq R} |\rd_t^b u_{(0)}|(t,r) + \brk{t}^{-2^{-b-1} C_0\ep} \sup_{\f R2 \leq r \leq R} |\rd_r^2 \rd_t^b u_{(0)}|(t,r) \\
            \ls &\: \brk{t}^{2^{-b-1} C_0\ep} \brk{t}^{-1 -b - 2^{-b} C_0\ep + C_* \ep} + \brk{t}^{-2^{-b-1} C_0\ep} \brk{t}^{-1-b+C_*\ep} \\
            \ls &\: \brk{t}^{-1-b- (2^{-b-1} C_0-C_*)\ep}.
        \end{split}
    \end{equation*}
    This proves \eqref{eq:rdrrdtu0.prelim} when $a = 1$. Higher $\rd_r$ derivative estimates can be derived inductively in a similar manner.

    \pfstep{Step~2: Proof of main estimate} First observe that in order to prove \eqref{eq:finite.r.bound.RHS.0}, it suffices to show that
        \begin{equation}\label{eq:R.finte.0.goal}
            \sup_{\f R2 \leq r \leq R} |(\rd_t + \rd_r)^k (r u_{(0)})|(t,r) \ls \brk{t}^{-k-2^{-3} C_0 \ep + C_* \ep}, \quad k = 1,2,3,
        \end{equation}
        and
        \begin{equation}\label{eq:R.finte.0.goal.aux}
            \sup_{\f R2 \leq r \leq R} |(\rd_t + \rd_r)^k (r (u_{(0)}-v_{(0)}))|(t,r) \ls \brk{t}^{-k-2^{-3} C_0 \ep + C_* \ep}, \quad k = 1,2,3.
        \end{equation}

        We will only prove the estimate \eqref{eq:R.finte.0.goal}. The second estimate \eqref{eq:R.finte.0.goal.aux} can be proved in exactly the same way, except that it is easier, after making the following observations:
        \begin{enumerate}
            \item The proof of \eqref{eq:R.finte.0.goal} that we give only depends on the bound for $u_{(0)}$ in \eqref{eq:rdrrdtu0.prelim} and the bound for $\Box_\bfm u_{(0)}$ in \eqref{eq:vector.field.wave.in.r.u.C_*}.
            \item $u_{(0)} - v_{(0)}$ obeys the estimate \eqref{eq:vector.field.in.r.u-v.C_*} which is stronger than the estimate \eqref{eq:rdrrdtu0.prelim} we obtained for $u_{(0)}$ in Step~1.
            \item $\Box_{\bfm} (u_{(0)} - v_{(0)})$ and $\Box_{\bfm} u_{(0)}$ obey the same estimate since $\Box_{\bfm} (u_{(0)} - v_{(0)}) = \Box_{\bfm} u_{(0)}$.
        \end{enumerate}

    We now turn to the proof of \eqref{eq:R.finte.0.goal}. When $k = 1$, this is immediate from the \eqref{eq:rdrrdtu0.prelim}. When $k = 2$, the bound in \eqref{eq:rdrrdtu0.prelim} by itself is too weak, and thus we use the wave equation. We write
    \begin{equation}\label{eq:finite.r.bound.0.wave}
        \begin{split}
            (\rd_t + \rd_r)^2 (r u_{(0)}) = r\Box_\bfm u_{(0)} + 2 (\rd_t+\rd_r)(r \rd_t u_{(0)}),
        \end{split}
    \end{equation}
        since for spherically symmetric functions $r\Box_{\bfm} u_{(0)} = -(\rd_t - \rd_r)(\rd_t + \rd_r)(r u_{(0)})$. The first term in \eqref{eq:finite.r.bound.0.wave} is much better than needed by \eqref{eq:vector.field.wave.in.r.u.C_*}, and the second term in \eqref{eq:finite.r.bound.0.wave} is also acceptable by \eqref{eq:rdrrdtu0.prelim}.

    Finally, we deal with $k=3$, again using the wave equation, i.e.,
    \begin{equation}
        \begin{split}
            &\: (\rd_t + \rd_r)^3 (r u_{(0)}) \\
            = &\: (\rd_t + \rd_r) (-\rd_t + \rd_r) (\rd_t + \rd_r)  (r u_{(0)}) + 2 (- \rd_t + \rd_r) (\rd_t + \rd_r)  (r \rd_t u_{(0)}) + 4 (\rd_t + \rd_r)  (r \rd_t^2 u_{(0)}) \\
            = &\: \underbrace{(\rd_t + \rd_r) (r \Box_\bfm u_{(0)})}_{=:I} + \underbrace{2 r \rd_t\Box_{\bfm} u_{(0)}}_{=:II} + \underbrace{4 (\rd_t + \rd_r)  (r \rd_t^2 u_{(0)})}_{=:III}.
        \end{split}
    \end{equation}
    The term $I$ and $II$ are acceptable by \eqref{eq:vector.field.wave.in.r.u.C_*}, while the term $III$ is acceptable by \eqref{eq:rdrrdtu0.prelim}. \qedhere
\end{proof}

\begin{lemma}\label{lem:finite.r.1.2}
    Let $R>0$, $\eta \in (0,1)$ be fixed. Denote by $C_*>0$ a constant such that the following holds for $\ell = 1,2$.
    \begin{align}
        \sup_{r\leq R}|(t\rd_t)^a(r\rd_r)^b (u_{(\ell)}-v_{(\ell)})| \ls_R &\: \ep \eta^{-3} \brk{t}^{-\f 32-\ell + \eta+C_*\ep},\quad a+b \leq 4, \label{eq:vector.field.in.r.u-v.C_*.ell}\\
            \sup_{r\leq R}|(t\rd_t)^a(r\rd_r)^b u_{(\ell)}| \ls_R &\: \ep \brk{t}^{-1-\ell+C_*\ep},\quad a+b \leq 4, \label{eq:vector.field.in.r.u.C_*.ell}     \\
            \sup_{r\leq R}|(t\rd_t)^a(r\rd_r)^b \Box_\bfm u_{(\ell)}| \ls_R &\: \ep^2 \brk{t}^{-4+C_*\ep},\quad a+b \leq 1. \label{eq:vector.field.wave.in.r.u.C_*.ell}
        \end{align}
        Such a $C_*$ exists for $\ep \ll 1$ by Lemma~\ref{lem:elliptic.u-v}, Lemma~\ref{lem:elliptic.v} and Lemma~\ref{lem:Boxu.est} (and the fact that $t\rd_t$, $S$ and $\Box_\bfm$ all commute with $\bbS_{(0)}$).

    Suppose $\sup_{r\leq R}|u_{(\ell)}|\leq C_1 \brk{t}^{-1-\ell-C_0\ep}$ for $\ell = 1,2$. Then
    \begin{align}
        \sup_{\f R2 \leq r \leq R} |\hbox{RHS of \eqref{eq:A.in.terms.of.v.1}}|\ls &\: \brk{t}^{-1-(2^{-3} C_0-C_*)\ep}, \label{eq:finite.r.bound.RHS.1} \\
        \sup_{\f R2 \leq r \leq R} |\hbox{RHS of \eqref{eq:A.in.terms.of.v.2}}|\ls &\: \brk{t}^{-1-(2^{-3} C_0-C_*)\ep} \label{eq:finite.r.bound.RHS.2}
    \end{align}
    where the implicit constant depends on $u_0$, $u_1$, $\eta$, $R$, $C_*$ and $C_0$.
\end{lemma}
\begin{proof}
    Both estimates here are easier than Lemma~\ref{lem:finite.r.0} and we will be brief. As before, we allow all implicit constants to depend on $u_0$, $u_1$, $\eta$, $R$, $C_*$ and $C_0$. We first note that the following analogue of \eqref{eq:rdrrdtu0.prelim} holds:
    \begin{equation}\label{eq:rdrrdtu12.prelim}
        \sup_{\f R2\leq r \leq R} |\rd_r^a \rd_t^b u_{(\ell)}|(t,r) \ls \brk{t}^{-1 -b -\ell - 2^{-b-a} C_0\ep + C_* \ep} ,\quad a+b\leq 3, \quad \ell = 1,2.
    \end{equation}
    The proof is exactly the same as Step~1 of Lemma~\ref{lem:finite.r.0} and will be omitted. We now proceed to the proof of the main estimates \eqref{eq:finite.r.bound.RHS.1} and \eqref{eq:finite.r.bound.RHS.2}, and we will handle the second estimate first.

    \pfstep{Step~1: Proof of \eqref{eq:finite.r.bound.RHS.2}} This is the easier case. In fact, simply using \eqref{eq:rdrrdtu12.prelim} and \eqref{eq:vector.field.in.r.u-v.C_*.ell}, it already follows that
    $$\sup_{\f R2\leq r \leq R} \Bigg| r^{-4}\Big( \f{t+r}2\Big)^{2} (\rd_t + \rd_r) \Bigg( \Big(r^2(\rd_t + \rd_r)\Big)^2(ru_{(2),i})\Bigg)\Bigg|  \ls \brk{t}^{-1-(2^{-3}C_0-C_*)\ep}$$
    and
    $$\sup_{\f R2\leq r \leq R} \Bigg| r^{-4}\Big( \f{t+r}2\Big)^{2} (\rd_t + \rd_r) \Bigg( \Big(r^2(\rd_t + \rd_r)\Big)^2(r(u_{(2),i} - v_{(2),i}))\Bigg) \Bigg| \ls \brk{t}^{-1-(2^{-3}C_0-C_*)\ep}.$$
    (Note that we need not be concerned about the singular $r^{-1}$ factors since we only consider the region $r \in [\f R2,R]$. Similarly below.) The conclusion \eqref{eq:finite.r.bound.RHS.2} then follows from the triangle inequality.

    \pfstep{Step~2: Proof of \eqref{eq:finite.r.bound.RHS.1}} After using \eqref{eq:rdrrdtu12.prelim} and \eqref{eq:vector.field.in.r.u-v.C_*.ell} to discard terms that are acceptable, it remains to check the following two bounds:
    \begin{equation}\label{eq:R.finte.ell.goal}
            \sup_{\f R2 \leq r \leq R} \Big|(\rd_t + \rd_r)\Big(r^{-2} (\rd_t + \rd_r)\Big( r^2 (\rd_t + \rd_r) (r u_{(1)})\Big)\Big)\Big|(t,r) \ls \brk{t}^{-3-2^{-3} C_0 \ep + C_* \ep}
        \end{equation}
        and
        \begin{equation}\label{eq:R.finte.ell.goal.aux}
            \sup_{\f R2 \leq r \leq R} \Big|(\rd_t + \rd_r)\Big(r^{-2} (\rd_t + \rd_r) \Big( r^2 (\rd_t + \rd_r)(r (u_{(1)}-v_{(1)})\Big)\Big)\Big|(t,r) \ls \brk{t}^{-3-2^{-3} C_0 \ep + C_* \ep}.
        \end{equation}
    Aas in Step~2 of Lemma~\ref{lem:finite.r.0}, we focus on \eqref{eq:R.finte.ell.goal} since \eqref{eq:R.finte.ell.goal.aux} is similar but simpler. For this, we note that when $\ell = 1$,
    \begin{equation}\label{eq:Q1.ell=1}
        (-\rd_t + \rd_r)(r^{-2}(\rd_t + \rd_r)( r^2 (\rd_t + \rd_r) (r u_{(1)}))) = \Big( r^2(\rd_t+\rd_r)+ 2r \Big)(r \Box_{\bfm} u_{(1)}).
    \end{equation}
    Equation \eqref{eq:Q1.ell=1} can be verified by a direct computation which can be found by combining \cite[equation (5.63), equation (6.4), equation (6.6)]{LO}. Hence,
    \begin{equation}
        \begin{split}
            &\: (\rd_t + \rd_r)\Big(r^{-2} (\rd_t+\rd_r) \Big( r^2 (\rd_t + \rd_r) (r u_{(1)})\Big)\Big) \\
            = &\:  \Big( r^2(\rd_t+\rd_r)+ 2r \Big)(r \Box_{\bfm} u_{(1)}) +  2 r^{-2} (\rd_t+\rd_r) \Big( r^2 (\rd_t + \rd_r) (r \rd_t u_{(1)})\Big).
        \end{split}
    \end{equation}
    The first term is acceptable due to \eqref{eq:vector.field.wave.in.r.u.C_*.ell} and the second term is acceptable thanks to \eqref{eq:rdrrdtu12.prelim}. This concludes the proof of \eqref{eq:R.finte.ell.goal}. \qedhere
\end{proof}

We are now ready to prove Corollary~\ref{cor:main.2}.
\begin{proof}[Proof of Corollary~\ref{cor:main.2}]
   Using the bound \eqref{eq:finite.r.bound.RHS.0}, \eqref{eq:finite.r.bound.RHS.1} and \eqref{eq:finite.r.bound.RHS.2}, we have the following analogue of \eqref{eq:self.similar.zone.extra.decay.v.conclude}:
   \begin{equation}\label{eq:finite.zone.extra.decay.v.conclude}
\sup_{r: \f R2 \leq r \leq R}|\hbox{RHS of \eqref{eq:A.in.terms.of.v.0}--\eqref{eq:A.in.terms.of.v.2}}| \ls \brk{t}^{-1-(\f{C_0}8-C_*)\ep}.
\end{equation}
    The reminder of the proof then follows in exactly the same manner as Steps~2 and 3 of Corollary~\ref{cor:main.1}; we omit the details. \qedhere

\end{proof}

\subsection{Improved results in special cases (Proof of Corollary~\ref{cor:special.case})}

\begin{proof}[Proof of Corollary~\ref{cor:special.case}]
We look at the improvement for Corollary~\ref{cor:main.1}; the improvement for Corollary~\ref{cor:main.2} can be argued in a similar manner.

Following the argument in the proof of Corollary~\ref{cor:main.1}, since we only have control over the $\ell = 0$ mode, instead of \eqref{eq:frkA.extra.decay.p}, we only have the weaker conclusion
\begin{equation}
\begin{split}
&\: \sup_{r: c_1 t \leq r \leq c_2 t} \Big| \int_{\mathbb S^2} \int_{-\infty}^\infty G(\omg') \mathfrak A^2(p,\omg';\ep) \Big(\f {t+r}2 \Big)^{-1+\frac{\ep}{2}G(\omega')A(p,\omega')} \, \ud p \,\ud \mathring{\slashed{\sigma}}(\omg') \Big|  \ls \brk{t}^{-1-(\f{C_0}8-C_*)\ep}.
\end{split}
\end{equation}
Rearranging, we obtain
\begin{equation}\label{eq:main.cor.almost.weaker}
\begin{split}
&\: \sup_{r: c_1 t \leq r \leq c_2 t} \Big| \int_{\mathbb S^2} \int_{-\infty}^\infty G(\omega') \mathfrak A^2(p,\omg';\ep) \Big(\f {t+r}2 \Big)^{\f 12(\f{C_0}8-C_*)\ep+\frac{\ep}{2}G(\omega')A(p,\omega')} \, \ud p \,\ud \mathring{\slashed{\sigma}}(\omg') \Big| \ls \brk{t}^{-\f 12 (\f{C_0}8-C_*)\ep}.
\end{split}
\end{equation}
Compare this with \eqref{eq:main.cor.almost} where we had $G^2(\omg')$ instead of $G(\omg')$. The important point is that in the proof of Corollary~\ref{cor:main.1}, we only used that $G^2(\omg')$ is non-negative (and is strictly positive on an open dense set). Now by assumption $G(\omg')$ (or $-G(\omg')$) is already positive everywhere and thus we could just continue the rest of the argument using \eqref{eq:main.cor.almost.weaker}. \qedhere

\end{proof}

\bibliographystyle{amsplain}
\bibliography{main}

\end{document}